\def\UseBibTeX{1}
\documentclass[11pt,a4paper]{amsart}
\usepackage[T1]{fontenc}
\usepackage{lmodern}
\usepackage{microtype}
\usepackage{mathtools,amssymb,mathrsfs}
\usepackage[a4paper,margin=28mm,headheight=14pt]{geometry}
\usepackage{enumitem}
\usepackage{xcolor}
\usepackage{tikz}
\usetikzlibrary{arrows.meta,patterns}
\usepackage{placeins}
\definecolor{linkblue}{RGB}{25,53,85}
\usepackage[colorlinks=true,linkcolor=linkblue,citecolor=linkblue,urlcolor=linkblue,
pdfencoding=auto,psdextra]{hyperref}
\hypersetup{pdftitle={Global solutions to two-dimensional compressible MHD equations with density-dependent bulk viscosity},
pdfauthor={Yachun Li, Peng Lu, Zhaoyang Shang},
pdfsubject={Research manuscript: density-dependent bulk viscosity}}
\numberwithin{equation}{section}
\allowdisplaybreaks[2]
\setlist[enumerate]{label=\textup{(\roman*)},leftmargin=*,itemsep=3pt}
\setlist[itemize]{leftmargin=*,itemsep=3pt}
\newtheorem{theorem}{Theorem}[section]
\newtheorem{proposition}[theorem]{Proposition}
\newtheorem{lemma}[theorem]{Lemma}
\newtheorem{corollary}[theorem]{Corollary}
\theoremstyle{definition}
\newtheorem{definition}[theorem]{Definition}
\theoremstyle{remark}
\newtheorem{remark}[theorem]{Remark}
\newcommand{\TT}{\mathbb T^2}
\newcommand{\RR}{\mathbb R}

\newcommand{\Div}{\operatorname{div}}
\newcommand{\Curl}{\operatorname{curl}}

\newcommand{\Cov}{\operatorname{Cov}}

\newcommand{\M}{\mathbb M}
\newcommand{\dd}{\,\mathrm d}
\newcommand{\dx}{\,\mathrm dx}
\newcommand{\dt}{\,\mathrm dt}
\newcommand{\weakto}{\rightharpoonup}
\newcommand{\weakstar}{\overset{*}{\rightharpoonup}}
\newcommand{\avg}[1]{\langle #1\rangle}
\newcommand{\norm}[1]{\lVert #1\rVert}
\newcommand{\abs}[1]{\lvert #1\rvert}
\newcommand{\ip}[2]{\langle #1,#2\rangle}
\title[Compressible MHD with power-law bulk viscosity]{Global solutions to the two-dimensional compressible MHD equations with density-dependent bulk viscosity}
\author[Y. Li]{Yachun Li}
\address{School of Mathematical Sciences, MOE--LSC, and SHL--MAC,
Shanghai Jiao Tong University, Shanghai 200240, P.R. China}
\email{ycli@sjtu.edu.cn}
\author[P. Lu]{Peng Lu}
\address{Department of Mathematics, Zhejiang Sci-Tech University,
Hangzhou, Zhejiang 310018, P.R. China}
\email{plu25@zstu.edu.cn}
\author[Z. Shang]{Zhaoyang~Shang}
\address{School of Finance, Shanghai Lixin University of Accounting
and Finance, Shanghai 201209, P.R. China}
\email{shangzhaoyang@sjtu.edu.cn}
\thanks{Corresponding author: Zhaoyang Shang.}
\date{October 4, 2026}
\subjclass[2020]{35Q35, 35D30, 76N10, 76W05}
\keywords{Compressible magnetohydrodynamics, density-dependent viscosity, large data,
effective viscous flux, relative energy, acoustic damping}
\begin{document}
\begin{abstract}
We study the two-dimensional periodic resistive compressible
magnetohydrodynamic equations with constant shear viscosity and magnetic
diffusivity, bulk viscosity $\rho^\beta$, and pressure $\rho^\gamma$.
For every $\beta>1$ and $\gamma>1$, we construct global finite-energy
renormalized weak solutions from strictly positive bounded
$W^{1,q}$ densities, $q>2$, and arbitrary $H^1$ velocity and magnetic data.
The construction combines approximation-independent finite density
moments, generalized effective-flux identities, and a joint logarithmic
renormalization of the density and its entropy defect.
An augmented relative energy containing the $\beta$-power density
potential yields weak--strong uniqueness against positive $H^2$ strong
solutions for the same full pressure range.
We also establish global strong existence for a large-background family
when $\gamma>\max\{5-2\beta,\beta+1\}$.
Writing $m=\int\rho_0$, the density perturbation is bounded in the
scaled norm $m^{(\gamma-3)/2}\|\rho_0-m\|_{H^2}$, while the fixed
$H^2$ velocity and magnetic norms need not be small.
For sufficiently large $m$, frequency-dependent acoustic dissipation,
staged vorticity weights, and a separate initial-layer flux estimate
close a simultaneous bootstrap. The admissible pressure exponents have
infimum $7/3$, with the endpoint excluded.
For $\beta>2$ the relevant damping rate is the minimum of the
viscous and pressure-relaxation rates; keeping both rates extends
the argument without an upper restriction on $\beta$.
The strong theorem has no upper restriction on $\gamma$ under this
scaled density hypothesis; for fixed additive density perturbations it
applies in the stated lower range up to $\gamma=3$.
\end{abstract}

\maketitle
\section{Introduction and main results}\label{sec:introduction}

\subsection{The equations and the existence problem}
Let \(\TT=\RR^2/\mathbb Z^2\) be the unit flat torus.
We consider the resistive compressible magnetohydrodynamic system
\begin{equation}\label{eq:system}
\begin{aligned}
 &\partial_t\rho+\Div(\rho u)=0,\\
 &\partial_t(\rho u)+\Div(\rho u\otimes u)+\nabla\rho^\gamma
   =\mu\Delta u+\nabla\bigl((\mu+\rho^\beta)\Div u\bigr)
        +\Div\M(H),\\
 &\partial_tH+\Div(H\otimes u-u\otimes H)=\nu\Delta H,
       \qquad \Div H=0,
\end{aligned}
\end{equation}
where \(\rho\ge0\) is the density, \(u\in\RR^2\) the velocity, and
\(H\in\RR^2\) the magnetic field. The constants \(\mu,\nu\) are positive,
and
\[
 \M(H)=H\otimes H-\frac12|H|^2I.
\]
We use \((v\otimes w)_{ij}=v_iw_j\),
\((\nabla u)_{ij}=\partial_j u_i\), and row-wise divergence.
Thus the viscous operator is exactly the one displayed in
\eqref{eq:system}; its compressional coefficient is \(2\mu+\rho^\beta\).
All boundary conditions are periodic.

We address weak existence, comparison with strong solutions, and global
strong existence for a specified large-background family. The common
viscosity range is
\begin{equation}\label{eq:parameters}
             \beta>1,\qquad \gamma>1.
\end{equation}
The viscosity exponent has no upper restriction. Constants may depend
on its fixed value; no uniformity as $\beta\to\infty$ is asserted.
The weak-existence and weak--strong uniqueness results allow every
$\gamma>1$. The strong-existence theorem has additional pressure and
scaled-density hypotheses stated below; it is not an arbitrary-data
strong-existence theorem.
For weak existence no smallness condition is imposed on the initial energy.
The initial data satisfy, for some \(q_0>2\),
\begin{equation}\label{eq:data}
\begin{gathered}
 \rho_0\in W^{1,q_0}(\TT),\qquad
 0<c_0\le\rho_0\le C_0<\infty,\\
 u_0,H_0\in H^1(\TT;\RR^2),\qquad \Div H_0=0.
\end{gathered}
\end{equation}
In particular, \emph{large data} refers to arbitrary sizes of the norms
in \eqref{eq:data}. We do not claim existence for every datum having
only finite energy.

The distinction between weak and strong existence is important here.
Finite-energy weak existence requires sufficient compactness to identify
the pressure, the viscosity, and the convective and magnetic stresses.
It does not require a pointwise upper bound on the density.
The estimates below retain every fixed finite density moment, with
constants depending on the exponent. They do not justify passage to
an \(L^\infty\) density bound.

\subsection{Context and the compactness mechanism}
The density-dependent viscosity law originates in the work of
Va\u{\i}gant and Kazhikhov~\cite{VaigantKazhikhov1995} on planar
compressible Navier--Stokes equations. Their global theory for
\(\beta>3\) showed how a growing compressional viscosity can control
large density variations. For the periodic Navier--Stokes system,
Huang and Li~\cite{HuangLi2016} developed global strong and weak
existence results in the range \(\beta>4/3\), including initial vacuum
under their respective data assumptions.
For classical solutions of the MHD equations on the torus and
in the whole plane, we refer to Mei~\cite{Mei2015}, together
with the accompanying corrigendum~\cite{Mei2015Corrigendum}.
For compressible MHD in a bounded planar domain,
Chen, Huang, and Shi~\cite{ChenHuangShi2022} established large-data
strong and weak existence with \(\beta>4/3\) and appropriate boundary
conditions. These statements concern their specified geometries,
data classes, and notions of solution; a pure-power viscosity theorem
cannot simply be invoked for a mixed viscosity law.
In a different regime, Liu, Ma, and Xu~\cite{LiuMaXu2026}
proved global existence for periodic compressible MHD when
the constant volume viscosity is sufficiently large.
That coefficient condition is not imposed in the present
density-dependent setting.

Relative-energy methods for MHD weak--strong uniqueness were developed,
in particular, by Yan~\cite{Yan2013}. For density-dependent viscosities,
Bresch, Noble, and Vila~\cite{BreschNobleVila2016} introduced related
relative-entropy structures. Those works concern different constitutive
or thermodynamic settings; we derive the coefficient-difference and
magnetic cancellations for \eqref{eq:system} explicitly.
Large-density global strong existence has also been studied for
three-dimensional MHD with density-dependent viscosities by
Li, Lu, and Shang~\cite{LiLuShang2025}. Our shear viscosity remains a
fixed constant, the geometry is periodic and two-dimensional, and the
density hypothesis below is an explicit scaled perturbation condition.
We do not invoke a theorem for another viscosity law to supply the
strong estimates in this paper.

The general finite-energy compactness theory for compressible flows
provides another perspective; see, for example,
Feireisl, Novotn\'y, and Petzeltov\'a~\cite{FeireislNovotnyPetzeltova2001}.
In the present problem, however, identifying the density-dependent
viscous stress requires control of a product involving the density
and the divergence of the velocity. This issue remains after the
basic energy estimate.

A related recent Navier--Stokes preprint by Li and
Yang~\cite{LiYang2026} must be read with its version-dependent
hypotheses. In version~4, Theorem~4 in the range \(\beta>1\),
\(\gamma>1\) assumes an additional fixed-exponent uniform
high-momentum estimate for the approximating sequence.
We do not assume that estimate here. We also do not assert the
pointwise density bounds or large-time conclusions studied there.
The solution class considered here is characterized instead by
finite energy and finite moments, without pointwise density bounds.

The proof separates two estimates that have different purposes.
For fixed \(\delta>0\), the mixed law
\[
       \lambda_\delta(r)=r^\beta+\delta r^2
\]
permits a pointwise density estimate and a global continuation
argument for the regularized system. Constants in that argument may
depend on \(\delta\). By contrast, the subsequent compactness argument
uses only estimates uniform in \(\delta\): energy, every fixed finite
density and magnetic moment, and time compactness in negative
Sobolev spaces.

The density moment estimate is based on a transported viscosity
primitive. Set $B=\max\{\beta,2\}$, which accounts for both powers
in the approximate viscosity. With \(p=2m\beta+1\), the scalar
inequality has growth exponent
\[
          \alpha_m=
          \frac{\max\{1+1/(2m),B/2\}}{\beta+1/(2m)}<1.
\]
The additional $B/2$ term comes from the spatial mean of the viscous
flux. It cannot be discarded when $\beta>2$, but it remains strictly
sublinear. This is where the strict inequality \(\beta>1\) enters
the uniform moment argument. These moments make two bounded density multipliers,
\[
       b_0(r)=\frac1{2\mu+r^\beta},
       \qquad b_1(r)=\frac r{2\mu+r^\beta},
\]
available in a generalized effective-flux identity.
The renormalized source terms and the spatial mean terms in this
identity are kept explicitly. The remaining entropy defect \(D\)
is eliminated by jointly renormalizing its transport equation
and the continuity equation with
\[
                    \Phi_\varepsilon(\rho,D)
                      =\rho\log(1+D/\varepsilon).
\]
The density factor permits an integrable coefficient to replace
a spatially uniform Gronwall coefficient, while the zero-density
set is treated directly through the associated Young measure.

\subsection{Definition and statement}
For \(r\ge0\) and \(m\in\RR^2\), define the lower-semicontinuous
kinetic energy
\begin{equation}\label{eq:kinetic}
 \mathcal K(r,m)=
 \begin{cases}
  |m|^2/(2r),&r>0,\\
  0,&r=0,\ m=0,\\
  +\infty,&r=0,\ m\ne0.
 \end{cases}
\end{equation}

\begin{definition}\label{def:weak}
A finite-energy renormalized weak solution of
\eqref{eq:system} with data \eqref{eq:data} is a triple
\((\rho,u,H)\), with \(m=\rho u\), having the regularity in
\eqref{eq:regularity} below, satisfying \(\Div H=0\) and the
following identities. All test functions are smooth, periodic
in space, and compactly supported in time in \([0,\infty)\).
For scalar \(\phi\),
\begin{equation}\label{eq:weak-mass}
 \int_0^\infty\!\!\int_{\TT}
       (\rho\phi_t+m\cdot\nabla\phi)\dx\dt
       +\int_{\TT}\rho_0\phi(0)\dx=0.
\end{equation}
For vector fields \(\varphi,\psi\),
\begin{equation}\label{eq:weak-momentum}
\begin{aligned}
0={}&\int_0^\infty\!\!\int_{\TT}
 \bigl[m\cdot\varphi_t+(\rho u\otimes u):\nabla\varphi
                  +\rho^\gamma\Div\varphi
                  -\mu\nabla u:\nabla\varphi\\
 &\qquad
 -(\mu+\rho^\beta)(\Div u)(\Div\varphi)
                  -\M(H):\nabla\varphi\bigr]\dx\dt
       +\int_{\TT}\rho_0u_0\cdot\varphi(0)\dx,
\end{aligned}
\end{equation}
and
\begin{equation}\label{eq:weak-magnetic}
 \int_0^\infty\!\!\int_{\TT}
 \bigl[H\cdot\psi_t+(H\otimes u-u\otimes H):\nabla\psi
                     -\nu\nabla H:\nabla\psi\bigr]\dx\dt
 +\int_{\TT}H_0\cdot\psi(0)\dx=0.
\end{equation}
For every \(b\in C^1([0,\infty))\) with compactly supported
derivative,
\begin{equation}\label{eq:renormalized}
 \int_0^\infty\!\!\int_{\TT}
 \bigl[b(\rho)\phi_t+b(\rho)u\cdot\nabla\phi
          -(\rho b'(\rho)-b(\rho))(\Div u)\phi\bigr]\dx\dt
 +\int_{\TT}b(\rho_0)\phi(0)\dx=0.
\end{equation}
Finally, the energy inequality \eqref{eq:energy-main} holds
for the canonical time representatives specified below.
\end{definition}

\begin{theorem}\label{thm:main}
Assume \(\mu,\nu>0\), \eqref{eq:parameters}, and
\eqref{eq:data}. There exists a global finite-energy
renormalized weak solution of \eqref{eq:system}.
For every \(T<\infty\) and every fixed \(1\le p<\infty\),
\begin{equation}\label{eq:regularity}
\begin{gathered}
 \rho\ge0,\qquad
 \rho\in L^\infty(0,T;L^p(\TT))\cap C([0,T];L^p(\TT)),\\
 u,H\in L^2(0,T;H^1(\TT)),\qquad
 \sqrt\rho\,u\in L^\infty(0,T;L^2(\TT)),\\
 H\in C([0,T];L^2(\TT))\cap L^\infty(0,T;L^p(\TT)),\\
 \rho^{\beta/2}\Div u\in L^2((0,T)\times\TT).
\end{gathered}
\end{equation}
For every \(1<r<2\), the momentum has a representative
\(m\in C_{\mathrm{weak}}([0,T];L^r(\TT))\).
The density, momentum, and magnetic field attain the prescribed
initial data in these time topologies.

With these representatives, set
\begin{equation}\label{eq:energy-definition}
\begin{aligned}
 E(t)&=\int_{\TT}\left[
       \mathcal K(\rho(t),m(t))
          +\frac{\rho(t)^\gamma}{\gamma-1}
          +\frac{|H(t)|^2}{2}\right]\dx,\\
 \mathfrak d(t)&=\int_{\TT}\left[
      \mu|\nabla u|^2+(\mu+\rho^\beta)|\Div u|^2
                          +\nu|\nabla H|^2\right](t)\dx.
\end{aligned}
\end{equation}
Then
\begin{equation}\label{eq:energy-main}
                E(t)+\int_0^t\mathfrak d(s)\dd s\le E(0)
                    \qquad\text{for every }t\ge0.
\end{equation}
There is also a common full-measure set \(S\subset(0,\infty)\)
such that
\begin{equation}\label{eq:energy-start}
 E(t)+\int_s^t\mathfrak d(\tau)\dd\tau\le E(s),
       \qquad s\in S\cup\{0\},\quad t\ge s.
\end{equation}
Moreover,
\begin{equation}\label{eq:initial-kinetic}
 E(t)\longrightarrow E(0),\qquad
 \frac{m(t)}{\sqrt{\rho(t)}}\longrightarrow\sqrt{\rho_0}\,u_0
                      \quad\text{strongly in }L^2(\TT)
                      \quad(t\downarrow0),
\end{equation}
where the quotient is defined to be zero on the vacuum set.
\end{theorem}

\begin{remark}\label{rem:scope}
No smallness assumption is imposed on the initial data.
The weak-existence theorem alone does not assert a pointwise upper or positive lower
bound for the limiting density, uniqueness of general weak
solutions, or global existence of strong solutions.
Strong continuity in a given finite \(L^p\) space uses a higher
fixed moment; no estimate uniform as \(p\to\infty\) is involved.
The starting times in \eqref{eq:energy-start} are not enlarged
to all times, and no global energy equality is claimed.
The proof treats the full interval $\beta>1$ directly. With
$B=\max\{\beta,2\}$, the density-moment exponent in
\eqref{uni:sublinear-ode} remains strictly below one.
The fixed-parameter closure \eqref{fix:density-closure} compares
the powers $1+3B/8$ and $(B+2)/3$ with $B$; both are strictly
smaller because $B\ge2$. Thus the artificial quadratic viscosity
controls the lower-$\beta$ range, while the original viscosity
supplies the dominant power when $\beta>2$.
The bounded multipliers in \eqref{comp:quotient-sources} require
only $\beta>1$. In particular, $\beta=4/3$ and $\beta=2$ are
included in the weak theorem, without endpoint limiting arguments.
\end{remark}

\subsection{Weak--strong uniqueness and large-background strong solutions}

The comparison result uses a positive strong reference in the following
class, on each finite interval $[0,T]$:
\begin{equation}\label{eq:strong-class}
\begin{gathered}
 0<c_T\le r\le C_T,\qquad \Div B=0,\qquad
 r,U,B\in C([0,T];H^2),\\
 U,B\in L^2(0,T;H^3),\qquad
 U_t,B_t\in L^2(0,T;H^1),\qquad r_t\in C([0,T];H^1).
\end{gathered}
\end{equation}
\begin{theorem}[Weak--strong uniqueness]\label{thm:weak-strong}
Let $\beta>1$ and $\gamma>1$.
Suppose $(\rho,u,H)$ is a finite-energy renormalized weak solution in
Definition~\ref{def:weak}, and $(r,U,B)$ solves the same equations in
\eqref{eq:strong-class}. If their initial density, momentum, and magnetic
field agree, the solutions coincide on $[0,T]$.
\end{theorem}
The density moments and the initial energy trace in the weak class are
part of this statement. It is not a uniqueness assertion for an
unspecified, larger distributional solution class, nor does it extend
the lifespan of a reference strong solution by itself.

For the strong result define
\begin{equation}\label{eq:gamma-threshold}
 \Gamma_*(\beta)=
 \max\{5-2\beta,\beta+1\}.
\end{equation}
\begin{theorem}[Global strong solutions at large background density]
\label{thm:strong}
Fix $\mu,\nu>0$, $\beta>1$, $\gamma>\Gamma_*(\beta)$, and $R>0$.
There is $m_0=m_0(\mu,\nu,\beta,\gamma,R)$ with the following property.
If $m\ge m_0$ and $\rho_0,u_0,H_0\in H^2(\TT)$ satisfy
\begin{equation}\label{eq:scaled-data}
 \int_{\TT}\rho_0\dx=m,\qquad \Div H_0=0,\qquad
 m^{(\gamma-3)/2}\|\rho_0-m\|_{H^2}
       +\|u_0\|_{H^2}+\|H_0\|_{H^2}\le R,
\end{equation}
then \eqref{eq:system} has a unique global strong solution in
\eqref{eq:strong-class}, with $(r,U,B)$ replaced by $(\rho,u,H)$.
There is $C_R$, independent of $m$ and time, such that
\begin{equation}\label{eq:strong-uniform}
\begin{gathered}
 \tfrac34m\le\rho(t,x)\le\tfrac54m,\\
 \sup_{t\ge0}\left[
 m^{(\gamma-3)/2}\|\rho(t)-m\|_{H^2}
           +\|u(t)\|_{H^1}+\|H(t)\|_{H^1}\right]\le C_R.
\end{gathered}
\end{equation}
The uniform $H^2$ bounds for $u,H$ may depend on the fixed value of $m$.
Every weak solution in Definition~\ref{def:weak} with these same data
coincides with this strong solution on every finite interval.
\end{theorem}

\begin{remark}[The density scaling and the pressure range]
\label{rem:strong-scope}
The number $R$ is arbitrary and is fixed before $m$ is chosen. Thus the
velocity and magnetic norms need not be small. For $\gamma>3$, however,
\eqref{eq:scaled-data} requires an absolutely small density perturbation,
$\|\rho_0-m\|_{H^2}\le Rm^{-(\gamma-3)/2}$; the absence of an upper
bound on $\gamma$ must not be separated from this hypothesis.
For $\gamma<3$ the same condition even allows an additive density
perturbation of order $m^{(3-\gamma)/2}$, which is still small relative
to $m$. Nonzero mean magnetic fields are allowed.

In particular, for fixed $r_0,u_0,H_0\in H^2$ with $\int r_0=0$,
$\Div H_0=0$, and $\rho_0=m+r_0$, the theorem applies when
\begin{equation}\label{eq:fixed-additive-range}
 1<\beta<2,\qquad \Gamma_*(\beta)<\gamma\le3.
\end{equation}
If $r_0=0$, it applies to every $\gamma>\Gamma_*(\beta)$.
The lower threshold is $5-2\beta$ for $1<\beta\le4/3$ and
$\beta+1$ for $\beta\ge4/3$.
Its minimum is $7/3$, attained at $\beta=4/3$, but the theorem does not
include equality. For example, $\beta=33/25$ and $\gamma=12/5$ satisfy
both $\beta<4/3$ and $\gamma<5/2$. These are sufficient conditions from the present
estimates, not claimed necessary or optimal conditions for the equations.
Figure~\ref{fig:parameter-regions} summarizes these parameter ranges,
with $\gamma$ on the horizontal axis and $\beta$ on the vertical axis.
\end{remark}

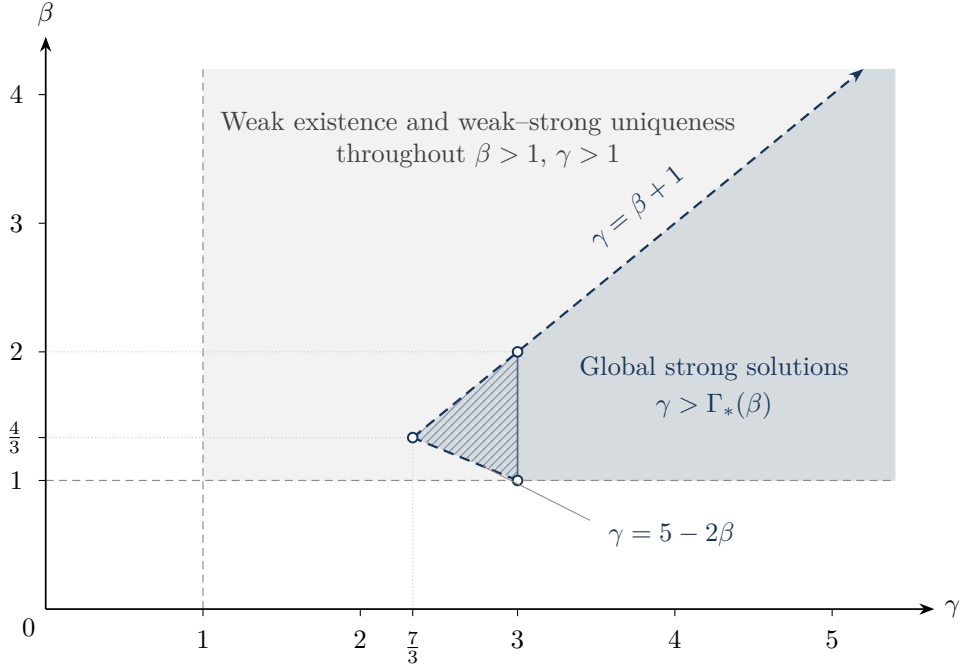
\begin{figure}[!htbp]
\centering
\begin{tikzpicture}[x=2.08cm,y=1.70cm,
    font=\small,
    threshold/.style={draw=linkblue,line width=0.85pt,dash pattern=on 4pt off 2.5pt},
    guide/.style={draw=black!25,line width=0.35pt,densely dotted}]
  \coordinate (lower) at (3,1);
  \coordinate (corner) at ({7/3},{4/3});
  \coordinate (upper) at (3,2);
  \fill[black!5] (1,1) rectangle (5.4,4.2);
  \fill[linkblue!18]
    (3,1) -- (5.4,1) -- (5.4,4.2) -- (5.2,4.2)
    -- (corner) -- cycle;
  \fill[pattern=north east lines,pattern color=linkblue!55]
    (lower) -- (upper) -- (corner) -- cycle;
  \draw[linkblue!75,line width=0.65pt] (3,1) -- (3,2);

  \draw[guide] ({7/3},0) -- (corner) -- (0,{4/3});
  \draw[guide] (0,2) -- (3,2);
  \draw[black!45,densely dashed,line width=0.45pt] (1,0) -- (1,4.2);
  \draw[black!45,densely dashed,line width=0.45pt] (0,1) -- (5.4,1);

  \draw[threshold] (lower) -- (corner);
  \draw[threshold,-{Stealth[length=2mm,width=1.4mm]}] (corner) -- (5.2,4.2);
  \foreach \point in {lower,corner,upper}
    \draw[linkblue,fill=white,line width=0.75pt] (\point) circle[radius=1.8pt];

  \node[align=center,text=black!75,font=\small] at (2.75,3.65)
    {Weak existence and weak--strong uniqueness\\
     throughout $\beta>1$, $\gamma>1$};
  \node[align=center,text=linkblue] at (4.25,1.72)
    {Global strong solutions\\[2pt]$\gamma>\Gamma_*(\beta)$};

  \node[rotate=39.25,anchor=south,text=linkblue] at (3.85,2.99)
    {$\gamma=\beta+1$};
  \node[text=linkblue] at (3.98,0.58) {$\gamma=5-2\beta$};
  \draw[black!45,line width=0.4pt] (3.46,0.69) -- (2.78,1.11);

  \draw[-{Stealth[length=2mm]},line width=0.65pt] (0,0) -- (5.65,0)
    node[right] {$\gamma$};
  \draw[-{Stealth[length=2mm]},line width=0.65pt] (0,0) -- (0,4.45)
    node[above] {$\beta$};
  \node[below left] at (0,0) {$0$};
  \foreach \x/\lab in {1/1,2/2,{7/3}/{\frac73},3/3,4/4,5/5}
    \draw ({\x},0) -- ++(0,-0.055) node[below=2pt] {$\lab$};
  \foreach \y/\lab in {1/1,{4/3}/{\frac43},2/2,3/3,4/4}
    \draw (0,{\y}) -- ++(-0.045,0) node[left=2pt] {$\lab$};
\end{tikzpicture}
\caption{Parameter regions for Theorems~\ref{thm:main},
\ref{thm:weak-strong}, and~\ref{thm:strong}.
Weak existence and weak--strong uniqueness hold throughout
$\beta>1$, $\gamma>1$, in their stated data and solution classes.
The blue region is the sufficient strong-solution range
$\gamma>\Gamma_*(\beta)$, subject to the large-background and
scaled-data hypotheses~\eqref{eq:scaled-data}.
The dashed threshold is excluded; its sole bend is
$(\gamma,\beta)=(7/3,4/3)$.
The hatched subset adds $\gamma\le3$ and describes the fixed-additive
density case~\eqref{eq:fixed-additive-range}; its vertical edge
$\gamma=3$ is included for $1<\beta<2$.
The plotting window is truncated; neither exponent has an upper
bound in the strong region. No necessity or optimality is asserted.}
\label{fig:parameter-regions}
\end{figure}
\FloatBarrier

The strong argument uses a corrected high acoustic energy and a weighted
vorticity energy. Frequency-dependent density dissipation pays for the
derivative strain without requiring an unavailable density third
derivative. For $1<\beta\le2$, a finite sequence of increasing
vorticity-weight exponents progressively improves the acoustic energy bound. Each stage
uses only the improvement already obtained at the preceding stage;
the energy does not increase when the weight changes. A separate
short-time flux estimate treats the initial layer, after which the
improved density bound makes the baroclinic and finite-exponent
vorticity feedback small. All bootstrap heights are fixed before
the density threshold is selected.
There are two damping regimes. For $1<\beta\le2$ the proof resolves
low and high density frequencies. For $\beta>2$, the minimum of
the viscous rate $D\asymp m^{\beta-1}$ and the pressure-relaxation
rate $\Gamma\asymp m^{\gamma-\beta}$ grows faster than $m$ under
$\gamma>\beta+1$. It then controls the entire differentiated
density strain without a moving cutoff. This second argument is
needed to cover the superquadratic viscosity range; simply keeping
the former damping rate would impose the unnecessary restriction
$\gamma>2\beta-1$.

\subsection{Organization}
Section~\ref{sec:preliminaries} records the analytic tools.
Section~\ref{sec:uniform} proves the approximation-independent
estimates on intervals of existence.
Section~\ref{sec:fixed} constructs the approximate solutions and
extends them globally for each fixed parameter.
Section~\ref{sec:compactness} proves the generalized flux identity
and removes the density defect.
Section~\ref{sec:assembly} passes to the original equations,
constructs the time representatives, and proves
Theorem~\ref{thm:main}. Section~\ref{sec:weak-strong} proves
Theorem~\ref{thm:weak-strong}. The remaining sections develop the
magnetic, acoustic, and vorticity estimates and the simultaneous
bootstrap for Theorem~\ref{thm:strong}; the final section states the
limits of the results.

\section{Preliminaries}\label{sec:preliminaries}

We normalize the measure of \(\TT\) to one. Unless otherwise indicated,
\(\norm{f}_p=\norm{f}_{L^p(\TT)}\), \(\norm{f}_{H^s}=\norm{f}_{H^s(\TT)}\),
and \(L^a_tX=L^a(0,T;X)\). Negative-order spaces are the usual periodic
Bessel-potential spaces. We write \(\avg{f}=\int_{\TT}f\dx\) and
\(\Curl v=\partial_1v_2-\partial_2v_1\).
The inverse \(\Delta^{-1}\) and all associated Fourier multipliers have
zero symbol at frequency zero. Constants may depend on the fixed torus
and on explicitly fixed exponents.

We use periodic Sobolev and Calder\'on--Zygmund estimates, including
\begin{align}
 \norm{v-\avg{v}}_q&\le C_q\norm{\nabla v}_2
       &&(1\le q<\infty),\label{pre:sobolev}\\
 \norm{v}_4^2&\le C\norm{v}_2
                    \bigl(\norm{\nabla v}_2+\norm{v}_2\bigr),\label{pre:gn}\\
 \norm{\nabla\Delta^{-1}\Div f}_q&\le C_q\norm{f}_q
       &&(1<q<\infty).\label{pre:elliptic}
\end{align}
The low-frequency term in \eqref{pre:gn} will always be retained.
For \(1<q<2\), we also use the finite critical embedding
\(W^{1,q}(\TT)\hookrightarrow L^{2q/(2-q)}(\TT)\).
Compactness in time is obtained from the Aubin--Lions--Simon theorem
\cite{Simon1987}; time traces in Hilbert triples are used only when
both the spatial and time-derivative hypotheses have been verified.

\begin{lemma}[Weighted Poincar\'e inequality]\label{pre:weighted-poincare}
Let \(\rho\ge0\), \(0<M_-\le\int\rho\le M_+\), and
\(\norm{\rho}_\kappa\le K\) for some fixed \(\kappa>1\).
For every \(1\le q<\infty\) and \(v\in H^1(\TT)\),
\begin{equation}\label{pre:weighted-poincare-est}
 \norm{v}_q\le
 C_{q,\kappa,M_-,M_+,K}
       \bigl(\norm{\sqrt\rho\,v}_2+\norm{\nabla v}_2\bigr).
\end{equation}
In particular, no positive lower bound for \(\rho\) is required.
\end{lemma}
\begin{proof}
Writing \(M=\int\rho\), we have
\[
 M\abs{\avg{v}}
 \le \abs{\int\rho v\dx}
       +\abs{\int\rho(v-\avg{v})\dx}
 \le \sqrt M\norm{\sqrt\rho\,v}_2
       +\norm{\rho}_\kappa\norm{v-\avg{v}}_{\kappa'}.
\]
Since \(\kappa'<\infty\), \eqref{pre:sobolev} controls the last norm.
Another application of \eqref{pre:sobolev} proves the assertion.
\end{proof}

\begin{lemma}[A derivative commutator estimate]\label{pre:commutator}
Let \(\mathcal R_{ij}=\partial_i\partial_j\Delta^{-1}\), and set
\([v,\mathcal R_{ij}]f=v\mathcal R_{ij}f-\mathcal R_{ij}(vf)\).
If \(1<p,q<\infty\) and \(1/r=1/p+1/q<1\), then
\begin{equation}\label{pre:commutator-est}
 \norm{\nabla[v,\mathcal R_{ij}]f}_r
 \le C_{p,q}\norm{\nabla v}_p\norm{f}_q.
\end{equation}
\end{lemma}
This is the nonendpoint derivative form of the
Coifman--Meyer commutator estimate \cite{CoifmanMeyer1978};
its periodic form follows by localization, the smooth remainder
being controlled by Poincar\'e's inequality after subtracting
\(\avg{v}\). The compactness argument uses only
\((p,q,r)=(2,6,3/2)\). An estimate of the undifferentiated commutator,
when needed, follows separately from H\"older's inequality and the
\(L^q\)-boundedness of \(\mathcal R_{ij}\).

\begin{lemma}[Friedrichs commutator in conservative form]
\label{pre:friedrichs}
Suppose \(z\in L^\infty_tL^p_x\), \(u\in L^2_tH^1_x\), and
\(2<p<\infty\). For a periodic spatial mollifier \(J_\ell\), write
\(z_\ell=J_\ell*z\). Then, with \(1/r=1/2+1/p\),
\begin{equation}\label{pre:friedrichs-est}
 c_\ell=\Div(uz_\ell)-J_\ell*\Div(uz)
       \longrightarrow0\quad\hbox{in }L^2_tL^r_x.
\end{equation}
Moreover,
\(\norm{c_\ell}_{L^2_tL^r_x}
 \le C\norm{\nabla u}_{L^2_{t,x}}\norm{z}_{L^\infty_tL^p_x}\).
\end{lemma}
\begin{proof}
In the distributional convolution formula,
\[
 c_\ell(x)=z_\ell(x)\Div u(x)
 +\int\nabla J_\ell(h)\cdot
             \bigl(u(x)-u(x-h)\bigr)z(x-h)\dd h.
\]
The fundamental theorem of calculus and H\"older's inequality give
the stated uniform bound. For smooth \(u\), spatial translation
continuity of \(z\) shows that the integral converges to
\(-z\Div u\), whereas the first term converges to \(z\Div u\).
Their sum therefore converges to zero in \(L^2_tL^r_x\).
Approximation of \(u\) in \(L^2_tH^1_x\), followed by the uniform
bound, proves \eqref{pre:friedrichs-est}. This is the commutator
underlying renormalization \cite{DiPernaLions1989}. It concerns fixed
limiting functions \(z,u\), not strong convergence between distinct
approximate solutions.
\end{proof}
When this lemma is applied to an equation with a source, the source
and the derivatives of the renormalizing function must additionally
have the stated dual integrability. Those products will be checked
explicitly in the density-compactness argument.

\begin{lemma}[Periodic Jacobian pairing]\label{pre:wente}
For smooth periodic \(v=(v_1,v_2)\) and \(\phi\),
\begin{equation}\label{pre:wente-est}
 \abs{\int\phi\det\nabla v\dx}
 \le C\norm{\nabla\phi}_2\norm{\nabla v}_2^2.
\end{equation}
For \(v,\phi\in H^1\), the left-hand side is understood as the
continuous distributional pairing obtained by approximation.
\end{lemma}
\begin{proof}
The periodic Wente estimate, equivalently the Hardy-space
Jacobian estimate \cite{CLMS1993}, gives
\[
 \norm{\nabla(-\Delta)^{-1}\det\nabla v}_2
 \le C\norm{\nabla v_1}_2\norm{\nabla v_2}_2.
\]
The Jacobian has zero mean. Integration by parts against
\(\phi-\avg{\phi}\) proves \eqref{pre:wente-est}.
\end{proof}
This estimate is used only at the fixed-parameter level to control
the Jacobian term in the effective-flux energy identity.

\begin{lemma}[Logarithmic div--curl estimate]\label{pre:bkm}
For each fixed \(q>2\) and \(v\in W^{2,q}(\TT;\RR^2)\),
\begin{equation}\label{pre:bkm-est}
 \norm{\nabla v}_\infty
 \le C_q\bigl(\norm{\Div v}_\infty+\norm{\Curl v}_\infty\bigr)
          \log\bigl(e+\norm{\nabla^2v}_q\bigr)
       +C_q\bigl(1+\norm{\nabla v}_2\bigr).
\end{equation}
\end{lemma}
\begin{proof}
A periodic dyadic decomposition and the div--curl representation
of \(\nabla v\) give, for every integer \(N\ge1\),
\[
 \norm{\nabla v}_\infty
 \le C\norm{\nabla v}_2
   +CN\bigl(\norm{\Div v}_\infty+\norm{\Curl v}_\infty\bigr)
   +C_q2^{-N(1-2/q)}\norm{\nabla^2v}_q.
\]
The finite low-frequency block is controlled by Bernstein's
inequality. Each intermediate block is a uniformly bounded
localized multiplier applied to \(\Div v,\Curl v\); the last term
follows from Bernstein's inequality and summation of the high
frequencies. Choose \(N\) comparable to
\(1+(1-2/q)^{-1}\log(e+\norm{\nabla^2v}_q)\).
\end{proof}
The logarithm in \eqref{pre:bkm-est} is used to close the density
gradient estimate for each fixed \(\delta>0\). It is not used to
assert a density maximum bound uniform as \(\delta\downarrow0\).

\section{Uniform estimates for the mixed-viscosity approximation}
\label{sec:uniform}

Throughout this section, \(\mu,\nu>0\), \(\beta>1\), and
\(\gamma>1\) are fixed. For \(0<\delta\le1\), set
\begin{equation}\label{uni:coefficients}
 \lambda_\delta(r)=r^\beta+\delta r^2,\qquad
 a_\delta(r)=2\mu+\lambda_\delta(r),\qquad P(r)=r^\gamma.
\end{equation}
We write \(B=\max\{\beta,2\}\); this is the largest growth exponent
in the mixed viscosity, not a restriction on \(\beta\).
Consider an exact positive-density solution, on an interval where
it already exists, of
\begin{equation}\label{uni:mixed-system}
\begin{aligned}
 &\rho_t+\Div(\rho u)=0,\\
 &(\rho u)_t+\Div(\rho u\otimes u)+\nabla P(\rho)
   =\mu\Delta u+\nabla((\mu+\lambda_\delta(\rho))\Div u)
      +\Div\M(H),\\
 &H_t+\Div(H\otimes u-u\otimes H)=\nu\Delta H,
       \qquad \Div H=0.
\end{aligned}
\end{equation}
Here \(\M(H)=H\otimes H-\frac12|H|^2I\), with row-wise divergence.
Subscripts \(\delta\) on the solution are suppressed. We first
prove the estimates for smooth solutions; their validity for
positive-density \(H^2\) strong solutions is explained below.

\begin{proposition}\label{uni:uniform-moments}
Let \(T<\infty\). Suppose solutions of \eqref{uni:mixed-system}
exist on \([0,\tau]\), where \(\tau\le T\), and their initial data obey
\begin{equation}\label{uni:initial-bounds}
 0<M_-\le\int\rho_0\dx\le M_+,\qquad
 \norm{\rho_0}_\infty+\norm{u_0}_{H^1}+\norm{H_0}_{H^1}\le K.
\end{equation}
Then, for every fixed \(1\le p<\infty\),
\begin{equation}\label{uni:moment-bound}
 \sup_{0\le t\le\tau}
      \bigl(\norm{\rho(t)}_p+\norm{H(t)}_p\bigr)\le C_{p,T}.
\end{equation}
The constants depend only on the indicated initial bounds,
\(\mu,\nu,\beta,\gamma,T\), and the fixed exponent. They are
independent of \(\delta\), of \(\tau\le T\), of the density maximum
during the evolution, and of a common positive density lower bound.
\end{proposition}

\begin{proof}
Write \(d=\Div u\), \(\omega=\Curl u\), \(M=\int\rho_0\dx\), and
\[
 E(t)=\int\left(\frac12\rho|u|^2
                 +\frac{\rho^\gamma}{\gamma-1}
                 +\frac12|H|^2\right)\dx,\qquad
 e_\delta(t)=\int\left(a_\delta d^2+\mu\omega^2
                               +\nu|\nabla H|^2\right)\dx.
\]
The kinetic and magnetic work terms cancel. The periodic identity
\(\norm{\nabla u}_2^2=\norm{d}_2^2+\norm{\omega}_2^2\) identifies
the viscous dissipation exactly, and yields
\begin{equation}\label{uni:energy}
 E(t)+\int_0^t e_\delta(s)\dd s=E(0),\qquad
 \norm{\nabla u(t)}_2^2\le C_\mu e_\delta(t).
\end{equation}
In particular \(E(0)\) is bounded uniformly by
\eqref{uni:initial-bounds}. Applying Lemma~\ref{pre:weighted-poincare}
with \(\kappa=\gamma\) gives, for every fixed finite \(q\),
\begin{equation}\label{uni:velocity}
 \norm{u(t)}_q\le C_q(1+\norm{\nabla u(t)}_2).
\end{equation}
Only the basic pressure energy is used here.

\medskip\noindent\textit{Magnetic moments.}
The induction equation is
\(H_t+u\cdot\nabla H=H\cdot\nabla u-Hd+\nu\Delta H\).
For a fixed \(q\ge2\), testing by \(|H|^{q-2}H\) gives
\begin{equation}\label{uni:magnetic-test}
\begin{aligned}
 &\frac1q\frac{\dd}{\dt}\int|H|^q\dx
 +\nu\int|H|^{q-2}|\nabla H|^2\dx
   +\nu(q-2)\int|H|^{q-2}|\nabla|H||^2\dx\\
 &=\int|H|^{q-2}H_iH_j\partial_j u_i\dx
       -\left(1-\frac1q\right)\int d|H|^q\dx.
\end{aligned}
\end{equation}
Put \(Z=|H|^{q/2}\). The dissipation controls
\(4\nu(q-1)q^{-2}\norm{\nabla Z}_2^2\); the right-hand side is at most
\(C_q\norm{\nabla u}_2\norm{Z}_4^2\).
Using \eqref{pre:gn}, including its low-frequency term, and Young's
inequality, we obtain
\begin{equation}\label{uni:magnetic-moments}
 \frac{\dd}{\dt}\norm{H}_q^q
   +c_{q,\nu}\norm{\nabla Z}_2^2
 \le C_{q,\nu}(1+\norm{\nabla u}_2^2)\norm{H}_q^q,\qquad
 \sup_{t\le\tau}\norm{H(t)}_q\le C_{q,T}.
\end{equation}
The last estimate follows from \eqref{uni:energy} and
\(H_0\in H^1\hookrightarrow L^q\). In particular, magnetic moments
have been obtained without any higher density moment.

\medskip\noindent\textit{A transported scalar.}
Define zero-mean potentials by
\begin{equation}\label{uni:potentials}
 -\Delta\xi=\Div(\rho u),\qquad
 -\Delta\eta=\Div\Div(\rho u\otimes u-H\otimes H),
\end{equation}
and introduce the full effective flux and its primitive
\begin{equation}\label{uni:flux-primitive}
 \mathcal B=a_\delta d-P(\rho)-\frac12|H|^2,\qquad
 \Lambda_\delta(r)=2\mu\log r+\frac{r^\beta-1}{\beta}
                                  +\frac\delta2(r^2-1).
\end{equation}
Taking the divergence of the momentum equation gives
\(\xi_t+\eta+\mathcal B-\avg{\mathcal B}=0\).
Since \(r\Lambda_\delta'(r)=a_\delta(r)\), the continuity equation
implies \(D_u\Lambda_\delta(\rho)=-a_\delta d\), where
\(D_u=\partial_t+u\cdot\nabla\). Consequently,
\begin{equation}\label{uni:transport}
 D_u(\Lambda_\delta(\rho)-\xi)+P(\rho)+\frac12|H|^2
       =\eta-u\cdot\nabla\xi-\avg{\mathcal B}.
\end{equation}
Both the magnetic pressure and the fluid pressure on the left
are nonnegative. No modification of the continuity equation
has been made.

\medskip\noindent\textit{A fixed density moment.}
Fix an integer \(m\ge1\), and set
\begin{equation}\label{uni:indices}
\begin{gathered}
 k=2m,\quad \epsilon=k^{-1},\quad p=k\beta+1,\quad
 s=\frac p\beta=\frac{kp}{p-1}>2,\\
 y=\norm{\rho}_p,\quad
 W=(\Lambda_\delta(\rho)-\xi)_+,\quad
 I=\int\rho W^k\dx,\quad f=I^{1/k}.
\end{gathered}
\end{equation}
The continuity equation and the \(C^1\) chain rule for the positive
part to the power \(k\) give
\begin{equation}\label{uni:positive-test}
 \frac1k I'
 \le\int\rho W^{k-1}|\eta-u\cdot\nabla\xi|\dx
       +|\avg{\mathcal B}|\int\rho W^{k-1}\dx.
\end{equation}
H\"older's inequality, first in the measure \(\rho\dx\) and then
with spatial exponents \(p,p'\), uses the identity
\begin{equation}\label{uni:weighted-holder}
 \left(\int\rho|G|^k\dx\right)^{1/k}
       \le y^\epsilon\norm{G}_{kp'}
       =y^\epsilon\norm{G}_s.
\end{equation}
To include times when \(I=0\), let \(f_\zeta=(I+\zeta)^{1/k}\),
\(\zeta>0\). Then \eqref{uni:positive-test} yields
\begin{equation}\label{uni:f-derivative}
 f_\zeta'
 \le \left(\frac I{I+\zeta}\right)^{(k-1)/k}
       \left[y^\epsilon\norm{\eta-u\cdot\nabla\xi}_s
                         +M^\epsilon|\avg{\mathcal B}|\right]
 \le y^\epsilon\norm{\eta-u\cdot\nabla\xi}_s
                         +M^\epsilon|\avg{\mathcal B}|.
\end{equation}

Choose the fixed, finite exponents
\begin{equation}\label{uni:elliptic-indices}
\begin{gathered}
 q=\frac{2p}{\beta-1},\qquad r=\frac{2p}{\beta+1};\\
 \frac1s=\frac1p+\frac2q,\quad
 \frac1r=\frac1p+\frac1q,\quad
 \frac1s=\frac1q+\frac1r.
\end{gathered}
\end{equation}
Here \(\beta>1\) ensures \(q<\infty\), and all singular-integral
exponents are strictly between one and infinity. Thus
\begin{equation}\label{uni:eta-xi}
\begin{aligned}
 \norm{\eta}_s
 &\le C_s\bigl(y\norm{u}_q^2+\norm{H}_{2s}^2\bigr),\\
 \norm{u\cdot\nabla\xi}_s
 &\le\norm{u}_q\norm{\nabla\xi}_r
 \le C_r\norm{u}_q\norm{\rho u}_r
 \le C_r y\norm{u}_q^2.
\end{aligned}
\end{equation}
Equations \eqref{uni:velocity} and \eqref{uni:magnetic-moments}
therefore bound their sum by
\(C_{m,T}[y(1+e_\delta)+1]\).
The spatial mean of the flux is controlled separately:
\begin{equation}\label{uni:mean-flux}
 |\avg{\mathcal B}|
 \le\left(\int a_\delta d^2\dx\right)^{1/2}
        \left(\int a_\delta\dx\right)^{1/2}
        +\int(P(\rho)+|H|^2/2)\dx
 \le C\bigl[1+\sqrt{e_\delta}(1+y^{B/2})\bigr].
\end{equation}
Indeed \(p=k\beta+1>B\), so that \(\delta\le1\) and finite measure give
\(\int a_\delta\le2\mu+y^\beta+y^2\le C(1+y^B)\).
The pressure integral is controlled by energy for every
\(\gamma>1\); no comparison between \(\gamma\) and \(\beta\) is needed.
It follows that
\begin{equation}\label{uni:scalar-ode-input}
 f_\zeta'\le C_{m,T}(1+e_\delta)
                 (1+y^{\max\{1+\epsilon,B/2\}}).
\end{equation}

\medskip\noindent\textit{Coercivity and closure.}
For \(r\ge2\),
\(\Lambda_\delta(r)\ge c_\beta r^\beta\), where
\(c_\beta=(1-2^{-\beta})/\beta\). Since
\(\Lambda_\delta(\rho)\le W+|\xi|\), splitting at \(\rho=2\) and
using \eqref{uni:weighted-holder} gives
\begin{equation}\label{uni:coercivity-first}
 \int\rho^p\dx\le C_m\left(1+I+\int\rho|\xi|^k\dx\right),\qquad
 y^{\beta+\epsilon}\le C_m(1+f+y^\epsilon\norm{\xi}_s).
\end{equation}
Set \(q_0=2s/(s+2)\in(1,2)\); then
\(1/q_0=1/2+1/s\). The finite critical Sobolev estimate and
the kinetic energy bound imply
\begin{equation}\label{uni:xi-lower-power}
 \norm{\xi}_s
 \le C_s\norm{\xi}_{W^{1,q_0}}
 \le C_s\norm{\rho u}_{q_0}
 \le C_s\norm{\sqrt\rho u}_2\norm{\sqrt\rho}_s
 \le C_{m,T}\norm{\rho}_{s/2}^{1/2}
 \le C_{m,T}y^{1/2}.
\end{equation}
The last step uses \(1<s/2=p/(2\beta)<p\).
The term \(y^{\epsilon+1/2}\) is strictly lower order than
\(y^{\epsilon+\beta}\). Absorbing it yields
\begin{equation}\label{uni:coercivity}
 y^{\beta+\epsilon}\le C_{m,T}(1+f)
                         \le C_{m,T}(1+f_\zeta).
\end{equation}
Combining this with \eqref{uni:scalar-ode-input}, we obtain
\begin{equation}\label{uni:sublinear-ode}
 \begin{gathered}
 f_\zeta'\le C_{m,T}(1+e_\delta)(1+f_\zeta)^{\alpha_m},
 \\
 \alpha_m=\frac{\max\{1+1/(2m),B/2\}}{\beta+1/(2m)}<1.
 \end{gathered}
\end{equation}
The strict inequality follows from \(1<\beta\) and
\(B/2=\max\{\beta/2,1\}<\beta\); no upper bound on \(\beta\)
is used.
Hence
\begin{equation}\label{uni:integrated-ode}
 (1+f_\zeta(t))^{1-\alpha_m}
 \le(1+f_\zeta(0))^{1-\alpha_m}
   +C_{m,T}(1-\alpha_m)\left(T+\int_0^\tau e_\delta\dt\right).
\end{equation}
The initial term is uniform: \((\Lambda_\delta(\rho_0))_+\)
is bounded by \eqref{uni:initial-bounds}, whereas
\(\norm{\xi_0}_\infty\le C\norm{\rho_0u_0}_4\le C\).
In particular the negative logarithm at small density does not
enter this upper bound. Taking \(\zeta\downarrow0\) proves the
uniform \(L^{2m\beta+1}\) density bound. Given any finite target
exponent, one chooses a fixed larger \(m\) and uses the finite
measure of \(\TT\). No bound uniform as \(m\to\infty\) is asserted.
\end{proof}

\begin{corollary}[Vanishing additional stress]\label{uni:small-stress}
Under the assumptions of Proposition~\ref{uni:uniform-moments},
for every fixed \(1\le r<2\),
\begin{equation}\label{uni:small-stress-est}
 \norm{\delta\rho^2\Div u}_{L^2_tL^r_x}
 \le C_{r,T}\sqrt\delta.
\end{equation}
\end{corollary}
\begin{proof}
Use \(1/r=1/2+1/q_r\), \(q_r=2r/(2-r)<\infty\), to obtain
\[
 \norm{\delta\rho^2d}_{L^2_tL^r_x}
 \le\sqrt\delta\,\norm{\rho}_{L^\infty_tL^{q_r}_x}
                    \norm{\sqrt\delta\,\rho d}_{L^2_{t,x}}.
\]
The two factors are bounded by
\eqref{uni:moment-bound} and \eqref{uni:energy}.
\end{proof}

\begin{remark}\label{uni:strong-solution-tests}
The preceding identities also hold on compact subintervals for
positive-density strong solutions satisfying
\(\rho,u,H\in C_tH^2_x\), \(u_t,H_t\in L^2_tH^1_x\), and
\(\rho_t\in C_tH^1_x\). In fact
\((\rho u)_t\in L^2_tH^1_x\), so that
\(\xi_t\in L^2_tH^2_x\) and
\(\partial_t(\Lambda_\delta(\rho)-\xi)\in L^2_tH^1_x\).
The positive-part integral is therefore absolutely continuous,
and its chain rule follows by smooth approximation of the
positive part. Magnetic power tests are justified in the same
way. The individual density bounds needed for these chain rules
do not enter the estimates above.
The present argument neither constructs the approximate solutions
globally nor proves that the additional dissipation
\(\delta\int\rho^2d^2\) tends to zero. It supplies uniform finite
moments and the vanishing stress \eqref{uni:small-stress-est}.
\end{remark}

\section{Finite-parameter global solvability}\label{sec:fixed}

Fix \(\beta>1\), \(\gamma>1\), and \(0<\delta\le1\), and set
\begin{equation}\label{fix:coefficients}
 \lambda_\delta(r)=r^\beta+\delta r^2,\qquad
 a_\delta(r)=2\mu+\lambda_\delta(r),\qquad P(r)=r^\gamma.
\end{equation}
As in Section~\ref{sec:uniform}, write \(B=\max\{\beta,2\}\ge2\).
The equations are \eqref{eq:system}, with \(\rho^\beta\) in the bulk
viscosity replaced by \(\lambda_\delta(\rho)\).
We abbreviate \(a=a_\delta(\rho)\), \(\lambda=\lambda_\delta(\rho)\), and set
\begin{equation}\label{fix:notation}
 \begin{gathered}
 D_t=\partial_t+u\cdot\nabla,\qquad d=\Div u,\qquad v=D_tu,\\
 \mathcal B=ad-P-\tfrac12|H|^2,\qquad
 \omega=\partial_1u_2-\partial_2u_1,\qquad
 \nabla^\perp=(-\partial_2,\partial_1).
 \end{gathered}
\end{equation}
In particular, \(\mathcal B\) includes the magnetic pressure.
We use the matrix convention \(L_{ij}=\partial_j u_i\) and row divergence
for tensors. The momentum equation becomes
\begin{equation}\label{fix:flux-equations}
 \rho v=\nabla\mathcal B+\mu\nabla^\perp\omega+H\cdot\nabla H,\qquad
 \Delta\mathcal B=\Div(\rho v-H\cdot\nabla H).
\end{equation}
Taking the scalar curl gives the corresponding constant-coefficient
elliptic equation for \(\omega\).
Constants in this section may depend on \(\delta\), the time horizon, the
positive lower bound of the initial density, and the initial \(H^2\)
norms. They will not be used as parameter-uniform bounds in the weak limit.

\subsection{Positive-density local construction}

\begin{proposition}\label{fix:local}
Let \((\rho_0,u_0,H_0)\in H^2(\TT)\), with
\(0<c_0\le\rho_0\le C_0\) and \(\Div H_0=0\).
There is a time \(T_0>0\) and a unique solution such that
\begin{equation}\label{fix:local-class}
 \begin{gathered}
 \rho\in C([0,T_0];H^2),\qquad c_0/2\le\rho\le2C_0,\\
 u,H\in C([0,T_0];H^2)\cap L^2(0,T_0;H^3),\\
 u_t,H_t\in L^2(0,T_0;H^1),\qquad
 \rho_t\in C([0,T_0];H^1).
 \end{gathered}
\end{equation}
The time \(T_0\) has a positive lower bound on each bounded set of such
initial data with a common positive density interval.
The same assertion holds for any nonnegative \(C^3\) bulk viscosity
and \(C^3\) pressure on a neighborhood of that interval.
\end{proposition}

\begin{proof}
We give the construction and the estimates needed for the continuation
argument. For a prescribed density \(r\), write
\[
 \mathcal L_r z=\mu\Delta z+
       \nabla\bigl((\mu+\lambda_\delta(r))\Div z\bigr).
\]
First consider
\begin{equation}\label{fix:linear}
 r z_t-\mathcal L_r z=f,\qquad
 c\le r\le C,\qquad
 \sup_t(\|r\|_{H^2}+\|r_t\|_{H^1})\le N.
\end{equation}
Let \(P_m\) be a symmetric Fourier projection. The finite-dimensional
mass matrix is positive definite. Its inverse also has the following
bound, uniform in \(m\):
\begin{equation}\label{fix:mass-inverse}
 P_m(rq)=P_m g,\quad q=P_mq
 \quad\Longrightarrow\quad
 \|q\|_2\le c^{-1}\|g\|_2,\qquad
 \|q\|_{H^1}\le C_N\|g\|_{H^1}.
\end{equation}
Indeed, testing successively with \(q\) and \((1-\Delta)q\) leaves only
the additional term \(\int\nabla r\cdot\nabla q\,q\). It satisfies
\[
 \left|\int\nabla r\cdot\nabla q\,q\right|
 \le C_N\|q\|_2^{1/2}\|q\|_{H^1}^{3/2}
 \le\varepsilon\|q\|_{H^1}^2+C_{N,\varepsilon}\|q\|_2^2.
\]
Applied to the projected equation, this gives
\begin{equation}\label{fix:linear-time}
 \|z_{m,t}\|_{H^1}
 \le C_N\bigl(\|z_m\|_{H^3}+\|f\|_{H^1}\bigr).
\end{equation}
Here \(\|\mathcal L_r z\|_{H^1}\le C_N\|z\|_{H^3}\);
the coefficient terms use \(L^4\times L^4\) and \(L^2\times L^\infty\).

Testing the projected equation with \(z+\Delta^2z\) retains the weighted
inertia. Its highest-order part is exactly
\begin{equation}\label{fix:inertial-identity}
 \int rz_t\Delta^2z
 =\frac12\frac d{dt}\int r|\Delta z|^2
       -\frac12\int r_t|\Delta z|^2
 +2\int(\nabla r\cdot\nabla z_t)\Delta z
       +\int(\Delta r)z_t\Delta z .
\end{equation}
Put \(Z_j=\|z\|_{H^j}\). The last two terms are bounded, by
\eqref{fix:linear-time}, by
\[
 C_N(Z_3+\|f\|_{H^1})Z_2^{1/2}Z_3^{1/2}
 \le\varepsilon Z_3^2+
       C_{N,\varepsilon}(Z_2^2+\|f\|_{H^1}^2).
\]
The \(r_t\) term is bounded by \(C_NZ_2Z_3\).
For \(b=\mu+\lambda_\delta(r)\) and \(d_z=\Div z\), the other
highest-order contribution is
\[
 \int b|\Delta d_z|^2+
 \int\bigl(2\nabla b\cdot\nabla d_z+(\Delta b)d_z\bigr)\Delta d_z.
\]
Its second integral is bounded by \(C_NZ_2^{1/2}Z_3^{3/2}\).
The shear term supplies \(\mu\|\nabla\Delta z\|_2^2\).
The forcing requires only \(H^1\), since
\(\int f\Delta^2z=-\int\nabla f\cdot\nabla\Delta z\).
Together with the zeroth-order estimate, this proves
\begin{equation}\label{fix:linear-H2}
 \mathcal E_2'+c_N Z_3^2
 \le C_N\mathcal E_2+C_N\|f\|_{H^1}^2,\qquad
 \mathcal E_2=\int r(|z|^2+|\Delta z|^2)\asymp Z_2^2.
\end{equation}
The energy-equivalence constants depend on the density interval, not on
its \(H^2\) norm. The same argument with \(z-\Delta z\), using first the
\(L^2\) part of \eqref{fix:mass-inverse}, proves
\begin{equation}\label{fix:linear-H1}
 \mathcal E_1'+c_N\|z\|_{H^2}^2
 \le C_N\mathcal E_1+C_N\|f\|_2^2,\qquad
 \mathcal E_1=\int r(|z|^2+|\nabla z|^2).
\end{equation}
Galerkin compactness gives the linear solution; the same energy
estimates give its uniqueness.

Start with heat extensions \(u^0,H^0\) of the initial data and iterate
\begin{equation}\label{fix:iteration}
 \begin{aligned}
 &\rho^{n+1}_t+\Div(\rho^{n+1}u^n)=0,\\
 &\rho^{n+1}u^{n+1}_t-\mathcal L_{\rho^{n+1}}u^{n+1}=f^n,\\
 &f^n=-\rho^{n+1}u^n\cdot\nabla u^n-\nabla P(\rho^{n+1})
                      +\Div\M(H^n),\\
 &H^{n+1}_t-\nu\Delta H^{n+1}
       =\Div(u^n\otimes H^n-H^n\otimes u^n).
 \end{aligned}
\end{equation}
All iterates have the prescribed initial data. The last tensor is
antisymmetric, so its double divergence vanishes and
\(\Div H^{n+1}=0\).
On a fixed \(H^2\) ball, the two parabolic right-hand sides are bounded
in \(L^\infty H^1\). The transport estimates give
\begin{equation}\label{fix:transport-ball}
 \begin{aligned}
 \|\rho^{n+1}\|_{L^\infty(0,t;H^2)}
 &\le\|\rho_0\|_{H^2}
          \exp\bigl(C\sqrt t\,\|u^n\|_{L^2(0,t;H^3)}\bigr),\\
 \|\rho_t^{n+1}\|_{L^\infty(0,t;H^1)}&\le C_N .
 \end{aligned}
\end{equation}
The characteristic formula keeps the density in
\([c_0/2,2C_0]\) after reducing \(t\).
Choose first an \(H^2\) ball with strict margin over the initial weighted
energies, then a ball for the \(L^2H^3\) norms, and finally a small time.
In \eqref{fix:linear-H2} the initial energy is multiplied by a Gronwall
factor that can be kept below \(2\); after absorption the shear and
magnetic dissipation retain fixed positive coefficients.
Consequently this order of choices preserves both balls.
It also bounds \(u_t^n,H_t^n\) in \(L^2H^1\).

For completeness, the contraction does not require the entire
\(L^2H^3\) norm of the initial heat layer to be small.
Write \(R^{n+1}=\rho^{n+1}-\rho^n\), \(U^{n+1}=u^{n+1}-u^n\),
and \(Z^{n+1}=H^{n+1}-H^n\). Since
\[
 R^{n+1}_t+\Div(u^nR^{n+1})
       =-\Div(\rho^n U^n),\qquad R^{n+1}(0)=0,
\]
the \(H^1\) transport estimate implies
\begin{equation}\label{fix:density-difference}
 \|R^{n+1}\|_{L^\infty H^1}^2
       \le Ct\|U^n\|_{L^2H^2}^2.
\end{equation}
The velocity difference has the new density as its inertia coefficient,
and its additional source terms are
\[
 -R^{n+1}u_t^n+
 \nabla\bigl((\lambda_\delta(\rho^{n+1})
                 -\lambda_\delta(\rho^n))\Div u^n\bigr).
\]
Their \(L^2\) norm is bounded by
\[
 C\|R^{n+1}\|_{H^1}
                  (\|u_t^n\|_{H^1}+\|u^n\|_{H^3}).
\]
For example, the differentiated coefficient terms use respectively
\(L^2\times L^\infty\), \(L^4\times L^4\times L^\infty\), and
\(L^4\times L^4\), not an \(H^1\)-to-\(L^\infty\) embedding.
The differences of the remaining sources are bounded in \(L^2\) by
\(C(\|R^{n+1}\|_{H^1}+\|U^n\|_{H^1}+\|Z^n\|_{H^1})\).
Thus \eqref{fix:linear-H1} and \eqref{fix:density-difference} give
\(\mathfrak D_{n+1}\le Ct\mathfrak D_n\), where
\(\mathfrak D_n\) is the squared \(L^\infty H^1\cap L^2H^2\)
difference norm of \(u^n,H^n\).
The small factor is the time integral in
\eqref{fix:density-difference}. Contraction, weak compactness in the
higher norms, and the equations give \eqref{fix:local-class}.
Strong \(H^2\) continuity follows from the linear energy identities
and the \(H^2\) transport estimate. The same difference argument proves
uniqueness and the asserted uniform restart time.
\end{proof}

All subsequent identities can be justified in this local class.
In particular, the momentum equation gives \(u_t\in C L^2\), and hence
\(v\in C L^2\cap L^2H^1\). The conservative acceleration equation used
below has a remainder in \(L^2L^2\) on each local interval.
Multiplication by \(\rho^{-1}\in C H^2\), an \(H^1\) multiplier in two
dimensions, gives \(v_t\in L^2H^{-1}\).
Spatial difference quotients and time Steklov averages therefore
justify its weighted energy identity, with
\(v(0)=\rho_0^{-1}\Div S(0)\).
The other identities follow in the same way from the local equations.
No material differentiation of a projected Galerkin equation is needed.

\subsection{Low-order logarithmic energy}

Consider a local solution on \([0,\tau]\), with \(\tau\le T\), and put
\[
 R=1+\sup_{t\le\tau}\|\rho(t)\|_\infty.
\]
The preceding uniform estimates, namely
\eqref{uni:energy}, Proposition~\ref{uni:uniform-moments}, and
\eqref{uni:magnetic-moments}, imply, for each fixed finite \(p\),
\begin{equation}\label{fix:uniform-input}
 \begin{gathered}
 \sup_{t\le\tau}(\|\rho(t)\|_p+\|H(t)\|_p)\le C_{p,T},\\
 \int_0^\tau\!\!\int
       (ad^2+\mu\omega^2+\nu|\nabla H|^2)\le C_T .
 \end{gathered}
\end{equation}
These constants do not depend on \(R\) or \(\tau\le T\).
All initial moments needed here are finite for positive \(H^2\) data.
We also use
\begin{equation}\label{fix:coefficient-bounds}
 a\ge2\mu,\qquad
 \|a\|_\infty\le CR^B,\qquad
 0\le\frac{\rho a'}a\le B,\qquad
 \sup_{t\le\tau}\int a\le C_T .
\end{equation}
Define
\begin{equation}\label{fix:AX}
 \begin{gathered}
 A^2=1+\frac12\int
       \bigl(\mu\omega^2+\mathcal B^2/a+|\nabla H|^2\bigr),\qquad
 X^2=\int\bigl(\rho|v|^2+\nu|\nabla^2H|^2\bigr),\\
 G=\|\nabla\mathcal B\|_2+\|\nabla\omega\|_2,\qquad
 g=\|\nabla u\|_2,\quad h=\|\nabla H\|_2 .
 \end{gathered}
\end{equation}
Here and below the equivalence of \(\|\nabla^2H\|_2\) and
\(\|\Delta H\|_2\) is understood on the torus.
The div--curl estimate, \eqref{fix:flux-equations}, and
\eqref{fix:uniform-input} yield
\begin{equation}\label{fix:low-input}
 \begin{gathered}
 g+h\le CA,\qquad \int_0^\tau A^2\dt\le C_T,\qquad
 \|\mathcal B\|_2\le CR^{B/2}A,\\
 G\le C(R^{1/2}X+A^{1/2}X^{1/2}),\qquad
 |\avg{\mathcal B}|\le CA.
 \end{gathered}
\end{equation}
For the last estimate, one must retain the periodic mean:
\(\|\mathcal B\|_1\le
\|\mathcal B/\sqrt a\|_2(\int a)^{1/2}\).
The magnetic term in the bound for \(G\) uses
\(\|H\|_4\|\nabla H\|_4\le CA^{1/2}X^{1/2}\).

An elementary weighted interpolation is useful:
\begin{equation}\label{fix:W}
 W:=\|\mathcal B^2/a\|_2
 \le C_\varepsilon
 \|\mathcal B/\sqrt a\|_2^{1-\varepsilon}
 \|\mathcal B\|_{2(1+\varepsilon)/\varepsilon}^{1+\varepsilon}
 \le C_\varepsilon(R^{B\varepsilon/2}AG+A^2),
 \qquad 0<\varepsilon<1.
\end{equation}
The first inequality follows by writing the fourth power as a product
with \((\mathcal B^2/a)^{1-\varepsilon}\) and using \(a\ge2\mu\).
For the second use the nonhomogeneous periodic inequality
\[
 \|\mathcal B\|_q
 \le C_q\|\mathcal B\|_2^{2/q}G^{1-2/q}+CA
\]
with \(q=2(1+\varepsilon)/\varepsilon\).

We record all terms in the low-order energy identity.
Let \(Q=H\otimes H\) and \(J_0=\int Q:L\).
Multiplication of the momentum equation by \(v\) and of the magnetic
equation by \(-\Delta H\) gives
\begin{equation}\label{fix:ten-identity}
 (A^2)'+X^2=-J_0'+\sum_{j=1}^{10}J_j,
\end{equation}
where, with repeated indices summed,
\begin{equation}\label{fix:ten-terms}
 \begin{aligned}
 J_1&=-\frac\mu2\int\omega^2d,
 &J_2&=2\int\mathcal B\det L,\\
 J_3&=-\frac12\int\mathcal B^2d(a^{-1}+\rho a'a^{-2}),\\
 J_4&=\int\mathcal B P d((\gamma-1)a^{-1}-\rho a'a^{-2}),\\
 J_5&=\frac12\int\mathcal B|H|^2d(a^{-1}-\rho a'a^{-2}),\\
 J_6&=-\int(\mathcal B/a)Q:L,
 &J_7&=-\nu\int(\mathcal B/a)H\cdot\Delta H,\\
 J_8&=\int H_iL_{ij}(\nu\Delta H_j-H_jd),\\
 J_9&=\int(LH)_i(\nu\Delta H_i+(LH)_i),\\
 J_{10}&=\int(u\cdot\nabla H-H\cdot\nabla u+Hd)\cdot\Delta H .
 \end{aligned}
\end{equation}
Here is a direct verification of the magnetic part of this identity.
Since \(D_tL=\nabla v-L^2\), integration by parts gives
\begin{equation}\label{fix:magnetic-cancellation}
 \int(H\cdot\nabla H)\cdot v
 =-\int Q:\nabla v
 =-J_0'+\int\bigl[(D_tQ):L+dQ:L-Q:L^2\bigr].
\end{equation}
Substituting \(D_tH=\nu\Delta H+LH-dH\), the term
\(H^{\mathsf T}L^2H\) cancels \(Q:L^2\).
The remaining nondiffusive terms are
\(|LH|^2-dQ:L\), giving \(J_8+J_9\).
Also,
\[
 D_t(|H|^2/2)=\nu H\cdot\Delta H+Q:L-|H|^2d.
\]
Together with \(D_ta=-\rho a'd\),
\(D_tP=-\gamma Pd\), and
\(\Div v=D_td+d^2-2\det L\), this gives \(J_3,\ldots,J_7\).
Finally
\(\partial_1v_2-\partial_2v_1=D_t\omega+\omega d\)
gives \(J_1\). Thus the signs and the transposes in
\eqref{fix:ten-terms} are fixed by the stated conventions.

We next bound every term. Periodic Wente duality
\eqref{pre:wente-est}, applied after subtracting the mean of \(\mathcal B\),
and the two-dimensional interpolation for \(\omega\) give
\[
 |J_1|\le CgAG,\qquad |J_2|\le CGg^2.
\]
The coefficient bounds in \eqref{fix:coefficient-bounds}, fixed
\(L^4\) moments of \(P\), and fixed \(L^8\) moments of \(H\), give
\begin{equation}\label{fix:J3-J7}
 |J_3|\le CgW,\qquad
 |J_4|+|J_5|+|J_6|\le Cg\sqrt W,\qquad
 |J_7|\le CX\sqrt W.
\end{equation}
For the remaining terms, integrate the transport term once:
\[
 \int u\cdot\nabla H\,\Delta H
 =-\int\partial_k u_j\partial_jH_i\partial_kH_i
       +\frac12\int d|\nabla H|^2.
\]
Consequently
\begin{equation}\label{fix:J8-J10}
 \begin{aligned}
 |J_8|+|J_9|+|J_{10}|
 &\le C\|\nabla u\|_4^2+
             C\|\nabla u\|_4X+CghX,\\
 \|\nabla u\|_4^2&\le C(W+AG+1),\qquad
 \|\nabla H\|_4^2\le ChX .
 \end{aligned}
\end{equation}
For example \(\int|H|^2|\nabla u|^2\) uses the fixed \(L^4\)
magnetic moment, and
\(\int|H||\nabla u||\Delta H|\) uses
\(\|H\|_4\|\nabla u\|_4\|\Delta H\|_2\).
With \(b_0=1+g^2+h^2\), all these estimates and
\eqref{fix:W} imply
\begin{equation}\label{fix:ten-bound}
 \sum_{j=1}^{10}|J_j|
 \le \eta X^2+C_{\eta,\varepsilon}
                    R^{1+B\varepsilon}b_0A^2 .
\end{equation}
To see the largest power, the first part of \(X\sqrt W\) is
\(CR^{1/4+B\varepsilon/4}A^{1/2}X^{3/2}\);
Young's inequality leaves precisely \(CR^{1+B\varepsilon}A^2\).
For \(gW\) the term \(R^{B\varepsilon/2}gAG\) leaves
\(CR^{1+B\varepsilon}g^2A^2\).
The other powers are smaller, using \(g+h\le CA\).

The energy on the left must be made positive.
Subtracting the constant tensor \(\avg H\otimes\avg H\) from \(Q\)
does not change \(J_0\), and periodic interpolation gives
\[
 |J_0|\le C\|H\|_2hg
          \le\eta g^2+C_\eta\|H\|_2^2h^2.
\]
Choose \(\eta\) small, and then \(L_0\) large, so that
\[
 \mathscr E=A^2+J_0+\frac{L_0}2h^2\asymp A^2.
\]
Adding \(L_0\) times the magnetic gradient identity to
\eqref{fix:ten-identity}, and rechoosing the small Young constants,
gives
\begin{equation}\label{fix:enhanced-energy}
 \mathscr E'+\tfrac12X^2
       \le C_\varepsilon R^{1+B\varepsilon}b_0\mathscr E .
\end{equation}
Since \(\int_0^\tau b_0\dt\le C_T\), division by \(\mathscr E\)
and integration prove
\begin{equation}\label{fix:log-energy}
 \sup_{t\le\tau}\log A^2(t)
       +\int_0^\tau(X/A)^2\dt
       \le C_{T,\varepsilon}R^{1+B\varepsilon}.
\end{equation}
All mean contributions have been retained in this estimate.

We will also need, for fixed \(q>2\) and \(0<\varepsilon<2/q\),
\begin{equation}\label{fix:velocity-interpolation}
 \|\nabla u\|_q\le C_{q,\varepsilon}
 \bigl(R^{1/2-1/q+B\varepsilon/2}A^{2/q}X^{1-2/q}
                         +R^{B\varepsilon/2}A\bigr).
\end{equation}
For clarity, set \(\theta=2/q-\varepsilon\) and
\(b_q=1-2/q+\varepsilon\). Weighted H\"older gives
\[
 \|\mathcal B/a\|_q
 \le C\|\mathcal B/a\|_2^\theta
                 \|\mathcal B\|_{2b_q/\varepsilon}^{b_q}
 \le C\bigl(R^{B\varepsilon/2}A^{2/q}G^{1-2/q}+A\bigr).
\]
In the last step use the nonhomogeneous interpolation and
\(\|\mathcal B\|_2\le CR^{B/2}A\).
Insert \eqref{fix:low-input}. The term containing
\(A^{1/2}X^{1/2}\) is bounded by the displayed right-hand side of
\eqref{fix:velocity-interpolation} and \(CR^{B\varepsilon/2}A\).
The vorticity and the fixed pressure and magnetic moments are treated
by the same div--curl estimate.

\subsection{Density bounds at a fixed parameter}

For the present interval choose a time-independent exponent
\begin{equation}\label{fix:alpha}
 \alpha=c_\delta R^{-B/2}\le\tfrac14,\qquad
 C\alpha^2\|\mu+\lambda\|_\infty\le\mu/8 .
\end{equation}
Testing momentum with \(|u|^\alpha u\) is legitimate by a smooth
regularization at \(u=0\). The additional bulk term is bounded by
\[
 \alpha\int(\mu+\lambda)|u|^\alpha|d||\nabla u|
 \le\tfrac12\int(\mu+\lambda)|u|^\alpha d^2
 +C\alpha^2\|\mu+\lambda\|_\infty
                         \int|u|^\alpha|\nabla u|^2 .
\]
The last term is absorbed by the shear dissipation.
After Young's inequality the pressure contribution is bounded by
\[
 C\int P^2|u|^\alpha
 \le C\int\rho|u|^{2+\alpha}
         +C\int\rho^{\gamma(2+\alpha)-\alpha/2}.
\]
All density exponents here lie in one fixed finite interval.
The magnetic contribution is at most
\[
 C\int|H|^4|u|^\alpha
 \le C\|u\|_2^\alpha
 \le C(1+\|\nabla u\|_2^2).
\]
The last inequality uses the positive total mass and the weighted
Poincar\'e estimate \eqref{pre:weighted-poincare-est}.
Thus Gronwall's inequality yields
\begin{equation}\label{fix:small-momentum}
 \sup_{t\le\tau}\int\rho|u|^{2+\alpha}\le C_{\delta,T}.
\end{equation}
The constant is uniform for \(0<\alpha\le1/4\); the initial moment
has this uniform bound. No fixed higher momentum exponent is asserted.

We specify the periodic singular-integral convention:
\[
 \mathcal R_{ij}=\partial_i\partial_j\Delta^{-1},
 \qquad
 C_u=\sum_{i,j}[u_i,\mathcal R_{ij}](\rho u_j),\qquad
 C_H=\sum_{i,j}[H_i,\mathcal R_{ij}]H_j .
\]
Every inverse Laplacian has zero spatial mean.
For a fixed \(q\in(2,3)\), the full periodic kernel and Morrey's
inequality give
\begin{equation}\label{fix:kernel}
 |C_u(x)|\le C_q\|\nabla u\|_q
       \int_{\TT}\operatorname{dist}(x,y)^{-1-2/q}
                                  \rho(y)|u(y)|\,\mathrm dy .
\end{equation}
The local delta component cancels in the commutator; the smooth part of
the periodic kernel is dominated by the same distance bound.
Only velocity differences enter, so there is no omitted constant mode.

Split the integral in \eqref{fix:kernel} at radius \(\ell\).
The precise large-exponent Sobolev estimate is
\[
 \|u\|_s\le C\sqrt s\,(1+\|\nabla u\|_2),\qquad 2\le s<\infty .
\]
For the zero-mean part, a dyadic Bernstein decomposition and
Cauchy--Schwarz bound
\(\sum_j2^{-2j/s}(2^j\|\Delta_j u\|_2)\)
by \(C\sqrt s\|\nabla u\|_2\); the weighted mean estimate supplies
the constant mode.
Use \(s=(1+\varepsilon_0)/\varepsilon_0\) in the near region,
and \eqref{fix:small-momentum} in the far region. The result is
\begin{equation}\label{fix:near-far}
 \begin{aligned}
 I_{\rm near}
 &\le C_qR\varepsilon_0^{-1/2}A
       \ell^{(q-2)/q-2\varepsilon_0/(1+\varepsilon_0)},\\
 I_{\rm far}
 &\le C_qR^{(1+\alpha)/(2+\alpha)}
       \ell^{-2/q+\alpha/(2+\alpha)} .
 \end{aligned}
\end{equation}
For example the far-region H\"older conjugates are \(2+\alpha\)
and \((2+\alpha)/(1+\alpha)\), after writing
\(\rho|u|=(\rho|u|^{2+\alpha})^{1/(2+\alpha)}
\rho^{(1+\alpha)/(2+\alpha)}\).
The radial integral yields exactly the indicated power of \(\ell\).
Its denominator stays away from zero for fixed \(q\in(2,3)\) and
\(0<\alpha\le1/4\).

Set
\begin{equation}\label{fix:matching}
 b_q=\frac2q-\frac{\alpha}{2+\alpha},\qquad
 \varepsilon_0=\frac{(q-2)\alpha}{8+(6-q)\alpha},\qquad
 \ell^{-b_q}=A^{2/q}.
\end{equation}
After an immaterial adjustment to the fixed torus scale, this is an
admissible radius. Direct computation gives
\[
 \frac{q-2}{q}-\frac{2\varepsilon_0}{1+\varepsilon_0}
       =\frac{q-2}{2}b_q,\qquad
 A\ell^{(q-2)b_q/2}=A^{2/q}.
\]
The near radial denominator is also uniformly positive.
The only degenerating constant is
\(\varepsilon_0^{-1/2}\asymp_q\alpha^{-1/2}\).
Therefore
\begin{equation}\label{fix:matched-integral}
 I_{\rm near}+I_{\rm far}
       \le C_{\delta,q}R^{1+B/4}A^{2/q}.
\end{equation}
In particular the cost \(R^{B/4}\) from the small momentum exponent
has not been discarded.

Fix \(0<\varepsilon<1/6\) and take \(q=2/(1-\varepsilon)\).
Combining \eqref{fix:velocity-interpolation},
\eqref{fix:kernel}, and \eqref{fix:matched-integral} yields
\begin{equation}\label{fix:commutator-pointwise}
 \begin{aligned}
 \|C_u\|_\infty
 &\le CR^{1+B/4+(B+1)\varepsilon/2}
                        A^{2-\varepsilon}(X/A)^\varepsilon\\
 &\quad+CR^{1+B/4+B\varepsilon/2}A^2 .
 \end{aligned}
\end{equation}
For any fixed \(0<\zeta<1\), use \eqref{fix:log-energy} with
\(\varepsilon=\zeta\), and the exact time H\"older inequality
\[
 \int_0^\tau A^{2-\varepsilon}(X/A)^\varepsilon\dt
 \le\left(\int_0^\tau A^2\dt\right)^{1-\varepsilon/2}
       \left(\int_0^\tau(X/A)^2\dt\right)^{\varepsilon/2}.
\]
Consequently
\begin{equation}\label{fix:commutator-time}
 \int_0^\tau\|C_u\|_\infty\dt
       \le C_{\delta,T,\varepsilon,\zeta}
                    R^{1+B/4+[1+B(1+\zeta)/2]\varepsilon}.
\end{equation}

Let \(-\Delta\xi=\Div(\rho u)\) and \(\avg\xi=0\). The periodic
Green gradient and \eqref{fix:small-momentum} give
\begin{equation}\label{fix:xi}
 \|\xi\|_\infty
 \le C\alpha^{-(1+\alpha)/(2+\alpha)}
           R^{(1+\alpha)/(2+\alpha)}
 \le C_\delta R^{(B+2)/3}.
\end{equation}
Indeed, the \(L^{(2+\alpha)/(1+\alpha)}\) norm of the Green gradient
is bounded by \(C\alpha^{-(1+\alpha)/(2+\alpha)}\).
The last inequality follows from \(\alpha=c_\delta R^{-B/2}\) and
\((1+\alpha)/(2+\alpha)\le2/3\), since \(\alpha\le1\).
Since \(\Div H=0\), \(C_H=-\mathcal R_{ij}(H_iH_j)\) is an
operator of order \(-1\) applied to \(H\cdot\nabla H\).
Thus
\begin{equation}\label{fix:magnetic-commutator}
 \begin{aligned}
 \|C_H\|_\infty
 &\le C\|H\|_{12}\|\nabla H\|_4\le CA^{1/2}X^{1/2},\\
 \int_0^\tau\|C_H\|_\infty\dt
 &\le C R^{(1+B\zeta)/4}.
 \end{aligned}
\end{equation}
In the time estimate write \(A^{1/2}X^{1/2}=A(X/A)^{1/2}\)
and use \(\int A^2\) and \eqref{fix:log-energy}.
Moreover,
\begin{equation}\label{fix:mean-time}
 \int_0^\tau|\avg{\mathcal B}|\dt\le C_T,
\end{equation}
by the basic energy, \(\sup_t\int a\le C_T\), and the fixed pressure
and magnetic moments.

Define the increasing function
\[
 \Theta_\delta(r)=2\mu\log r+r^\beta/\beta+\delta r^2/2 .
\]
The continuity equation gives \(D_t\Theta_\delta(\rho)=-ad\).
Applying \(-\Delta^{-1}\Div\) to momentum gives, with the above
singular-integral convention,
\[
 D_t\xi=-\mathcal B+\avg{\mathcal B}-C_u+C_H.
\]
Hence the exact material identity is
\begin{equation}\label{fix:trajectory}
 D_t(\Theta_\delta(\rho)-\xi)
       +P+\tfrac12|H|^2-C_u+C_H+\avg{\mathcal B}=0 .
\end{equation}
For \(r\ge1\), this primitive satisfies
\[
 \Theta_\delta(r)\ge c_{\beta,\delta}r^B,\qquad
 c_{\beta,\delta}=\min\{1/\beta,\delta/2\}>0.
\]
Integrate \eqref{fix:trajectory} along trajectories and discard the
nonnegative pressure terms. With \(\varepsilon=1/12\), \(\zeta=1/4\),
the exponent in \eqref{fix:commutator-time} obeys
\[
 1+\frac B4+\frac1{12}\left(1+\frac{5B}{8}\right)
 =\frac{13}{12}+\frac{29B}{96}
 \le1+\frac{3B}{8}.
\]
The magnetic exponent \((1+B/4)/4\) is smaller. Hence
\eqref{fix:commutator-time}--\eqref{fix:mean-time} imply
\begin{equation}\label{fix:density-closure}
 c_{\beta,\delta}(R-1)^B
 \le C_{\delta,T}(1+R^{1+3B/8}+R^{(B+2)/3}).
\end{equation}
If \(R\le2\) there is nothing to prove; otherwise use
the lower bound for \(\Theta_\delta\) at a density maximum with \(r\ge1\).
Since \(B\ge2\), both \(1+3B/8<B\) and \((B+2)/3<B\).
Thus \(R\le C_{\delta,T}\).
Substituting this bound in \eqref{fix:log-energy} gives
\begin{equation}\label{fix:bounded-low-energy}
 \sup_{t\le\tau}A^2(t)+\int_0^\tau X^2\dt\le C_{\delta,T}.
\end{equation}
In particular \(\int_0^\tau\|H\|_\infty^2\dt\le C_{\delta,T}\)
by two-dimensional interpolation.
Return to \eqref{fix:trajectory}, now keeping \(P+|H|^2/2\).
The pressure is bounded by the density upper bound and the magnetic
term is integrable as just noted. This proves
\(\Theta_\delta(\rho)\ge-C_{\delta,T}\).
Since \(\Theta_\delta(r)\to-\infty\) as \(r\downarrow0\), we obtain
\begin{equation}\label{fix:density-interval}
 0<c_{\delta,T}\le\rho(t,x)\le C_{\delta,T}
       \quad(0\le t\le\tau).
\end{equation}
This lower bound does not presuppose a bound for
\(\int\|\nabla u\|_\infty\).

\subsection{Material acceleration and the magnetic tensor}

We now work on the fixed positive density interval
\eqref{fix:density-interval}. All coefficient bounds in this and the
next subsection are allowed to depend on that interval.
From \eqref{fix:bounded-low-energy},
\[
 u,H\in L^\infty H^1,\qquad v,\Delta H\in L^2L^2,\qquad
 \mathcal B,\omega\in L^\infty L^2.
\]
Set \(h_2(t)=\|\Delta H(t)\|_2\). Constant-coefficient elliptic
estimates and the fixed magnetic moments give
\begin{equation}\label{fix:before-acceleration}
 \begin{gathered}
 \|\nabla\mathcal B\|_2+\|\nabla\omega\|_2
          \le C(1+\|v\|_2+h_2^{1/2}),\\
 \|\nabla u\|_4^4\le C(1+\|v\|_2^2+h_2),\qquad
 \|\nabla H\|_4^4\le C(1+h_2^2).
 \end{gathered}
\end{equation}
Thus both fourth powers are already integrable in time, before an
acceleration supremum has been proved.

For \(Q=H\otimes H\), the induction equation gives
\begin{equation}\label{fix:Q-equation}
 D_tQ-\nu\Delta Q
       =LQ+QL^{\mathsf T}-2dQ
          -2\nu\sum_k\partial_kH\otimes\partial_kH .
\end{equation}
Testing with \(-\Delta Q\), the transport part is bounded by
\(C\|\nabla u\|_2\|\nabla Q\|_4^2\).
Using the bounded \(H^1\) norm of \(u\), interpolation, and absorption,
we get
\begin{equation}\label{fix:Q-energy}
 (\|\nabla Q\|_2^2)'+c\nu\|\Delta Q\|_2^2
 \le C\|\nabla Q\|_2^2
       +C\|H\|_8^4\|\nabla u\|_4^2
       +C\|\nabla H\|_4^4 .
\end{equation}
Its right-hand coefficients and sources are time integrable by
\eqref{fix:before-acceleration}; hence
\(\sup\|\nabla Q\|_2^2+\int\|\Delta Q\|_2^2\le C\).
The pointwise identities
\[
 \Delta Q=H\otimes\Delta H+\Delta H\otimes H+
                   2\sum_k\partial_kH\otimes\partial_kH,\qquad
 |H\otimes z+z\otimes H|^2\ge2|H|^2|z|^2
\]
then imply
\begin{equation}\label{fix:weighted-magnetic-dissipation}
 \int_0^\tau\!\!\int |H|^2|\Delta H|^2
       \le C\int_0^\tau
             (\|\Delta Q\|_2^2+\|\nabla H\|_4^4)\dt\le C .
\end{equation}
This joint estimate is needed; separate integrability of the two
factors \(\|H\|_\infty^2\) and \(\|\Delta H\|_2^2\) would not suffice.

Introduce the total stress
\[
 S=\mu(L+L^{\mathsf T})+\lambda dI-PI+\M(H).
\]
The momentum equation reads \(\rho v_i=\partial_jS_{ij}\).
Using the continuity equation and commuting \(D_t\) with one spatial
derivative gives the exact conservative identity
\begin{equation}\label{fix:acceleration-equation}
 \rho D_tv_i
       =\partial_j\bigl(D_tS_{ij}+dS_{ij}-S_{ik}\partial_k u_j\bigr).
\end{equation}
The spatial derivatives of \(d\) cancel in obtaining this divergence
form. The expressions
\[
 D_tL=\nabla v-L^2,\qquad
 D_td=\Div v-\operatorname{tr}(L^2),\qquad
 D_t\lambda=-\rho\lambda'd
\]
show that its principal stress is
\(\mu(\nabla v+\nabla v^{\mathsf T})+\lambda\Div vI\).
The remainder \(N\) satisfies
\begin{equation}\label{fix:acceleration-remainder}
 |N|\le C\bigl(|L|^2+|L|+|H||\Delta H|+|H|^2|L|\bigr).
\end{equation}
More explicitly, the material derivative of the bulk stress contributes
\(-[\lambda\operatorname{tr}(L^2)+\rho\lambda'd^2]I\)
outside the principal part. The pressure remainder is
\((\gamma-1)PdI+PL^{\mathsf T}\).
The magnetic remainder follows from
\(D_tH=LH-dH+\nu\Delta H\).
None of these terms contains a density gradient.
Testing \eqref{fix:acceleration-equation} with \(v\), and using
periodic Korn coercivity, yields
\begin{equation}\label{fix:acceleration-energy}
 \frac d{dt}\int\rho|v|^2+c\mu\|\nabla v\|_2^2
 \le C\int\bigl(|L|^4+|L|^2+|H|^2|\Delta H|^2+|H|^4|L|^2\bigr).
\end{equation}
The last term is bounded by
\(\|H\|_8^4\|\nabla u\|_4^2\).
Equations \eqref{fix:before-acceleration} and
\eqref{fix:weighted-magnetic-dissipation} therefore prove
\begin{equation}\label{fix:acceleration-bound}
 \sup_{t\le\tau}\|v(t)\|_2^2
             +\int_0^\tau\|\nabla v\|_2^2\dt\le C .
\end{equation}
The initial acceleration is
\(\rho_0^{-1}\Div S(0)\in L^2\) for positive \(H^2\) data.
Its norm may depend on the particular smoothed initial datum.

\subsection{Density gradients and high-order continuation}

Fix \(q>2\). By \eqref{fix:acceleration-bound},
\(v\in L^2L^q\). The magnetic moment and interpolation give
\[
 \|H\cdot\nabla H\|_q\le C(1+h_2^{1-1/q})\in L^2(0,\tau).
\]
The flux and vorticity equations thus supply functions
\(L_1,K\in L^2(0,\tau)\), with uniform bounds for \(\tau\le T\), such that
\begin{equation}\label{fix:L1K}
 \|\nabla\mathcal B\|_q+\|\nabla\omega\|_q\le CL_1,\qquad
 1+\|d\|_\infty+\|\omega\|_\infty\le K.
\end{equation}
In the second estimate we used the means and
\(\|H\|_\infty^2\le C(1+h_2)\).
No magnetic \(H^2\) supremum is used here.

The useful density-gradient unknown is \(\Xi=a\nabla\rho\).
Differentiating the continuity equation and substituting
\(ad=\mathcal B+P+|H|^2/2\), we obtain
\begin{equation}\label{fix:Xi-equation}
 D_t\Xi_i+(\partial_i u_j)\Xi_j+
       (d+\rho P'/a)\Xi_i
       =-\rho\partial_i\mathcal B-\rho\partial_i(|H|^2/2).
\end{equation}
The two terms containing \(\rho a'd\nabla\rho\) cancel exactly.
For \(Y=\|\Xi\|_q\), this identity and div--curl reconstruction give
\begin{equation}\label{fix:Xi-inequality}
 \begin{aligned}
 Y'&\le C(1+\|\nabla u\|_\infty+K)Y+CL_1,\\
 \|\nabla^2u\|_q&\le C(L_1+KY).
 \end{aligned}
\end{equation}
The periodic BKM inequality \eqref{pre:bkm-est}, in the form
\[
 \|\nabla u\|_\infty
 \le C_q(\|d\|_\infty+\|\omega\|_\infty)
          \log(e+\|\nabla^2u\|_q)+C\|\nabla u\|_2+C,
\]
implies
\begin{equation}\label{fix:BKM}
 \|\nabla u\|_\infty\le CK\log(e+Y)+J,\qquad
 J=C[1+K\log(e+L_1+K)]\in L^1(0,\tau).
\end{equation}
The integrability follows from \(L_1,K\in L^2\) and the elementary
bound \(K\log(e+L_1+K)\le C(1+K^2+L_1^2)\).
For \(Z=\log(e+Y)\), \eqref{fix:Xi-inequality} becomes
\[
 Z'\le CKZ+C(1+K+L_1+J).
\]
Gronwall's inequality now proves
\begin{equation}\label{fix:W1q}
 \rho\in L^\infty(0,\tau;W^{1,q}),\qquad
 u\in L^2(0,\tau;W^{2,q}).
\end{equation}
In particular \(\nabla u\in L^2L^\infty\) is a consequence, not an
input, of this argument.

Write the induction equation as \(H_t-\nu\Delta H=\mathcal N_H\).
Direct product estimates give
\[
 \|\nabla\mathcal N_H\|_2
 \le C\bigl(\|\nabla u\|_\infty\|\nabla H\|_2
       +\|u\|_\infty\|H\|_{H^2}
       +\|H\|_\infty\|u\|_{H^2}\bigr).
\]
Test with \(\Delta^2H\). Since
\(\|H\|_\infty^2\le C\|H\|_{H^2}^2\), Young's inequality yields
\begin{equation}\label{fix:H2-magnetic}
 \begin{split}
 (\|\Delta H\|_2^2)'+c\nu\|\nabla\Delta H\|_2^2
 \le C\bigl(1+\|\nabla u\|_\infty^2+\|u\|_\infty^2
       +\|u\|_{H^2}^2\bigr)(1+\|\Delta H\|_2^2).
 \end{split}
\end{equation}
The coefficient is already integrable by \eqref{fix:W1q}.
Consequently
\begin{equation}\label{fix:magnetic-high}
 H\in L^\infty H^2\cap L^2H^3,\qquad H_t\in L^2H^1.
\end{equation}
Now \(\nabla\mathcal B,\nabla\omega\) are bounded in \(L^\infty L^2\).
The only further product needed to reconstruct \(u\) in \(H^2\) is
\[
 \|d\nabla\rho\|_2
 \le\|\nabla\rho\|_q\|d\|_{2q/(q-2)}
 \le C_q(1+\|\nabla^2u\|_2^{2/q}).
\]
The exponent \(2/q<1\) allows absorption, so \(u\in L^\infty H^2\).

It remains to close the \(H^2\) density norm without presupposing
the \(H^3\) velocity dissipation. Choose \(q=6\) in
\eqref{fix:W1q}. A second elliptic estimate gives
\begin{equation}\label{fix:second-flux}
 \|\nabla^2\mathcal B\|_2+\|\nabla^2\omega\|_2
 \le C\bigl(\|\nabla(\rho v)\|_2
               +\|\nabla(H\cdot\nabla H)\|_2\bigr)
 \le C(1+\|\nabla v\|_2).
\end{equation}
Here \(\nabla\rho\in L^\infty L^6\),
\(\|v\|_3\le C(1+\|\nabla v\|_2)\), and
\eqref{fix:magnetic-high} controls the magnetic products.
Differentiating \eqref{fix:Xi-equation} once gives
\begin{equation}\label{fix:Xi-H1}
 (\|\nabla\Xi\|_2)'
 \le C(1+\|\nabla u\|_\infty)\|\nabla\Xi\|_2+G_1,
       \qquad G_1\in L^2(0,\tau).
\end{equation}
Writing \(M=|H|^2/2\), all source classes are exhibited by
\[
 G_1\le C\bigl[1+\|\nabla^2u\|_3\|\Xi\|_6
  +\|\nabla\rho\|_6(\|\nabla\mathcal B\|_3+\|\nabla M\|_3)
  +\|\nabla^2\mathcal B\|_2+\|\nabla^2M\|_2\bigr].
\]
Differentiating the coefficient \(\rho P'/a\) produces only
a bounded coefficient times \(\nabla\rho\,\Xi\), controlled in \(L^2\).
All the displayed sources are in \(L^2_t\) by
\eqref{fix:W1q}, \eqref{fix:magnetic-high}, and
\eqref{fix:second-flux}.
Gronwall therefore gives \(\nabla\Xi\in L^\infty L^2\).
Since
\(\nabla\Xi=a\nabla^2\rho+a'\nabla\rho\otimes\nabla\rho\),
we conclude that \(\rho\in L^\infty H^2\).

Only now differentiate \(ad=\mathcal B+P+|H|^2/2\) twice.
The term \(d\nabla^2a\) uses
\(d\in L^2L^\infty\) and \(\rho\in L^\infty H^2\);
the cross term uses
\(\nabla a\in L^\infty L^6\) and \(\nabla d\in L^2L^3\).
The \(a''\) contribution is a bounded coefficient times
\(d\nabla\rho\otimes\nabla\rho\), also in \(L^2L^2\).
Div--curl reconstruction and \eqref{fix:second-flux} prove
\begin{equation}\label{fix:continuation-bounds}
 \begin{gathered}
 \rho,u,H\in L^\infty(0,\tau;H^2),\qquad
 u,H\in L^2(0,\tau;H^3),\\
 u_t,H_t\in L^2(0,\tau;H^1),\qquad
 \rho_t\in L^\infty(0,\tau;H^1).
 \end{gathered}
\end{equation}
The time-derivative assertions follow from the equations.
All derivatives of the coefficients have been used only on the fixed
positive density interval. In particular no bound for
\(\lambda_\delta''\) at vacuum is asserted.

\begin{proposition}\label{fix:global}
For every fixed \(0<\delta\le1\), every positive-density \(H^2\)
initial datum with \(\Div H_0=0\) generates a unique global solution
in the local class \eqref{fix:local-class} on every finite time interval.
The density remains bounded above and away from zero on such intervals,
with bounds allowed to depend on \(\delta\) and the initial data.
\end{proposition}

\begin{proof}
Let \(T_*\) be the maximal local existence time.
If \(T_*<\infty\), fix \(T>T_*\) and apply all the preceding estimates
on \([0,\tau]\), \(\tau<T_*\).
The uniform moment bounds hold on each such interval with constants
independent of \(\tau\).
The logarithmic energy estimate and
\eqref{fix:density-closure} first give the common positive density
interval \eqref{fix:density-interval}.
Equations \eqref{fix:bounded-low-energy} and
\eqref{fix:continuation-bounds} then give a common \(H^2\) ball.
For any sequence \(t_j\uparrow T_*\), Proposition~\ref{fix:local}
therefore provides the same positive lower bound for the restart
lifetime at \(t_j\). Restarting and using local uniqueness extends the
solution beyond \(T_*\), a contradiction.
This does not require first constructing a strong \(H^2\) limit at
the terminal time. Hence \(T_*=\infty\).
\end{proof}

Applied to the separately smoothed initial data used for
Theorem~\ref{thm:main}, Proposition~\ref{fix:global} constructs the global
finite-parameter sequence rather than assuming its existence.
The original initial density need only belong to the positive bounded
\(W^{1,q_0}\) class, and the original \(u_0,H_0\) need only be in \(H^1\);
each smoothed datum is in \(H^2\).
The constants in \eqref{fix:density-interval} and
\eqref{fix:continuation-bounds}, and the initial acceleration norms,
may vary with the smoothing index.
Only the parameter-uniform estimates from Section~\ref{sec:uniform}
are carried into the compactness argument.

\section{Generalized effective flux and density compactness}
\label{sec:compactness}

We now establish the density compactness needed to remove the additional
bulk viscosity. Throughout this section, $T<\infty$ is fixed and
$Q_T=(0,T)\times\TT$. We abbreviate $L^q(0,T;L^r(\TT))$ by $L^qL^r$;
all Sobolev spaces below are periodic. Spatial averages are denoted by
$\avg{f}=\int_{\TT}f\dx$, with $|\TT|=1$.

\begin{proposition}[Density compactness]\label{comp:density}
Let $\beta>1$, $\gamma>1$, and $\mu,\nu>0$. Let
$(\rho_n,u_n,H_n)$ be positive-density strong solutions on $[0,T]$
of the approximate system with bulk viscosity
$\rho_n^\beta+\delta_n\rho_n^2$, where $\delta_n\downarrow0$.
Set
\begin{equation}\label{comp:notation}
 m_n=\rho_nu_n,\qquad d_n=\Div u_n,\qquad
 P_n=\rho_n^\gamma,\qquad M_n=\tfrac12|H_n|^2.
\end{equation}
Assume that the total masses have a uniform positive lower bound and that,
for every fixed finite $k\ge1$,
\begin{equation}\label{comp:moments}
 \sup_n\norm{\rho_n}_{L^\infty L^k}<\infty.
\end{equation}
Assume also the uniform energy bounds
\begin{equation}\label{comp:energy}
\begin{aligned}
 \sup_n\bigg\{&\sup_{0\le t\le T}
       \int_{\TT}(\rho_n|u_n|^2+|H_n|^2)\dx\\
 &+\int_{Q_T}\Big[|\nabla u_n|^2+|\nabla H_n|^2
       +(2\mu+\rho_n^\beta+\delta_n\rho_n^2)d_n^2\Big]\dx\dt
                         \bigg\}<\infty.
\end{aligned}
\end{equation}
Suppose the initial densities satisfy
\begin{equation}\label{comp:initial}
 0<c_0\le\rho_{0n}\le C_0,\qquad
 \rho_{0n}\longrightarrow\rho_0
       \quad\text{in every finite }L^k(\TT).
\end{equation}
Then, along a subsequence, there exist $\rho,u,H$ such that
\begin{equation}\label{comp:conclusion}
 \rho_n\longrightarrow\rho
       \quad\text{strongly in every finite }L^k(Q_T),
 \qquad u_n\weakto u\quad\text{in }L^2H^1.
\end{equation}
Moreover, $H_n\to H$ strongly in $L^2L^4$, and, with $d=\Div u$,
\begin{equation}\label{comp:stress-conclusion}
\begin{aligned}
 \rho_n^\gamma&\longrightarrow\rho^\gamma
                  &&\text{in every finite }L^k(Q_T),\\
 \rho_n^\beta d_n&\longrightarrow\rho^\beta d
                  &&\text{in }\mathcal D'(Q_T),\\
 \rho_n^{\beta/2}d_n&\weakto\rho^{\beta/2}d
                  &&\text{in }L^2(Q_T).
\end{aligned}
\end{equation}
\end{proposition}

The strong solutions constructed above satisfy the hypotheses of this
proposition. Its proof uses no uniform pointwise bound or positive lower
bound for the densities at positive times. We give the argument for the
exact approximate equations and retain their nonlinear renormalization
sources throughout.

\subsection{Compactness before identifying the density}

Write
\begin{equation}\label{comp:flux-definition}
 a(z)=2\mu+z^\beta,\qquad
 F_n=P_n-a(\rho_n)d_n,\qquad e_n=\delta_n\rho_n^2d_n.
\end{equation}
In particular, $F_n$ does not include the magnetic pressure. The equations
used in this section are
\begin{equation}\label{comp:equations}
\begin{aligned}
 &\partial_t\rho_n+\Div m_n=0,\\
 &\partial_tm_n+\Div(m_n\otimes u_n)+\nabla P_n\\
 &\qquad=\mu\Delta u_n+\nabla((\mu+\rho_n^\beta)d_n)
       +\Div(H_n\otimes H_n)-\nabla M_n+\nabla e_n,\\
 &\partial_tH_n+\Div(H_n\otimes u_n-u_n\otimes H_n)
       =\nu\Delta H_n,\qquad \Div H_n=0.
\end{aligned}
\end{equation}
The artificial stress is small in the precise space needed below:
\begin{equation}\label{comp:small-stress}
 \norm{e_n}_{L^2L^{8/7}}
 \le\sqrt{\delta_n}\norm{\rho_n}_{L^\infty L^{8/3}}
       \norm{\sqrt{\delta_n}\rho_nd_n}_{L^2(Q_T)}
 \longrightarrow0.
\end{equation}

The mass-weighted Poincare inequality yields a uniform $L^2H^1$ bound for
$u_n$. Indeed, if $\overline u_n=\avg{u_n}$, then
\begin{equation}\label{comp:mean-velocity}
\begin{aligned}
 \Big(\int_{\TT}\rho_n\dx\Big)|\overline u_n|
 &\le\Big(\int_{\TT}\rho_n\dx\Big)^{1/2}
          \norm{\sqrt{\rho_n}u_n}_2
       +\norm{\rho_n}_2\norm{u_n-\overline u_n}_2.
\end{aligned}
\end{equation}
The last norm is controlled by $\norm{\nabla u_n}_2$. Also,
\begin{equation}\label{comp:momentum-bounds}
\begin{aligned}
 \norm{m_n}_{4/3}
     &\le\norm{\sqrt{\rho_n}u_n}_2\norm{\sqrt{\rho_n}}_4,\\
 \norm{m_n\otimes u_n}_{L^2L^{8/7}}
     &\le\norm{m_n}_{L^\infty L^{4/3}}\norm{u_n}_{L^2L^8}.
\end{aligned}
\end{equation}
Furthermore,
\begin{equation}\label{comp:bulk-bound}
 \norm{a(\rho_n)d_n}_{L^2L^{8/7}}
 \le\norm{\sqrt{a(\rho_n)}}_{L^\infty L^{8/3}}
       \norm{\sqrt{a(\rho_n)}d_n}_{L^2(Q_T)}.
\end{equation}
The pressure is controlled by a fixed higher density moment. The magnetic
energy and the two-dimensional interpolation inequality imply
$H_n\in L^4L^4$, uniformly in $n$. Consequently,
\begin{equation}\label{comp:time-momentum}
 \partial_tm_n\text{ is bounded in }L^2W^{-1,8/7},\qquad
 F_n\text{ is bounded in }L^2L^{8/7}.
\end{equation}
We will also use $F_n$ bounded in $L^2L^{3/2}$, which follows by replacing
the spatial exponents in \eqref{comp:bulk-bound} with $3/2,6,2$.

The compactness theorem of \cite{Simon1987} applies to
\begin{equation}\label{comp:embedding-chain}
 L^{4/3}(\TT)\Subset H^{-3/5}(\TT)
       \hookrightarrow W^{-1,8/7}(\TT).
\end{equation}
Here the first embedding is dual to $H^{3/5}\Subset L^4$, and the second
is dual to $W^{1,8}\hookrightarrow H^{3/5}$. Thus
\begin{equation}\label{comp:strong-negative}
 m_n\longrightarrow m\quad\text{in }L^2H^{-3/5}.
\end{equation}
Similarly, $\rho_nu_n$ is bounded in $L^2L^2$, so the continuity equation
gives $\partial_t\rho_n$ bounded in $L^2H^{-1}$. For each fixed
$0<\eta<1$, we may arrange that
\begin{equation}\label{comp:rho-negative}
 \rho_n\longrightarrow\rho\quad\text{in }L^2H^{-\eta}.
\end{equation}
For a smooth test function $\varphi$, split
\begin{equation}\label{comp:product-identification}
\begin{aligned}
 \int_{Q_T}\varphi(\rho_nu_n-\rho u)\dx\dt
 &=\int_0^T\ip{\rho_n-\rho}{\varphi u_n}\dt
    +\int_{Q_T}\varphi\rho(u_n-u)\dx\dt.
\end{aligned}
\end{equation}
The first term tends to zero by \eqref{comp:rho-negative}, and the second
by weak convergence of $u_n$ in $L^2H^1$. Hence $m=\rho u$. The same
argument, using \eqref{comp:strong-negative}, identifies
\begin{equation}\label{comp:convection-limit}
 m_n\otimes u_n\weakto m\otimes u
          \quad\text{in }L^2L^{8/7}.
\end{equation}
These identifications use only strong convergence in negative spaces.

The induction equation and $H_nu_n\in L^2L^{4/3}$ give
$\partial_tH_n$ bounded in $L^2W^{-1,4/3}$. Since $H^1(\TT)\Subset L^4(\TT)$,
\begin{equation}\label{comp:magnetic-compactness}
 H_n\to H\text{ in }L^2L^4,\qquad
 H_n\otimes H_n\to H\otimes H\text{ in }L^1L^2.
\end{equation}

\subsection{The generalized effective-flux identity}

\begin{lemma}\label{comp:flux}
Under the hypotheses of Proposition~\ref{comp:density}, let
$b\in C^1([0,\infty))$ satisfy
\begin{equation}\label{comp:b-condition}
 \norm{b}_\infty+\norm{k_b}_\infty<\infty,\qquad
 k_b(z)=zb'(z)-b(z).
\end{equation}
If $F_n\weakto\overline F$ in $L^2L^{8/7}$ and
$b(\rho_n)\weakstar\overline b$ in $L^\infty(Q_T)$, then, after passage
to a subsequence,
\begin{equation}\label{comp:flux-factorization}
 F_nb(\rho_n)\longrightarrow\overline F\,\overline b
                    \quad\text{in }\mathcal D'(Q_T).
\end{equation}
The same subsequence can be used for any prescribed finite collection of
functions satisfying \eqref{comp:b-condition}.
\end{lemma}

\begin{proof}
Put $b_n=b(\rho_n)$ and $K_n=k_b(\rho_n)d_n$. The exact nonlinear
continuity equation is
\begin{equation}\label{comp:b-continuity}
 \partial_tb_n+\Div(b_nu_n)=-K_n.
\end{equation}
In view of \eqref{comp:b-condition}, $K_n$ is bounded in $L^2(Q_T)$.
Hence $b_n$ is compact in $L^2H^{-\eta}$ for $0<\eta<1$. The argument
in \eqref{comp:product-identification} gives
$b_nu_n\weakto\overline b\,u$ in $L^2L^2$. Writing
$K_n\weakto\overline K$ in $L^2(Q_T)$, we obtain
\begin{equation}\label{comp:b-limit}
 \partial_t\overline b+\Div(\overline b\,u)=-\overline K.
\end{equation}
Neither $\overline b=b(\rho)$ nor
$\overline K=k_b(\rho)\Div u$ is used here.

Let $\Delta^{-1}$ be the inverse of the negative-spectrum Laplacian on
mean-zero functions, with its zero mode set to zero, and define
\begin{equation}\label{comp:potential}
 \Phi_n=\Delta^{-1}\Div m_n,\qquad
 \mathcal R_{ij}=\partial_i\partial_j\Delta^{-1}.
\end{equation}
The multiplier $\mathcal R_{ij}$ is self-adjoint and has symbol
$\xi_i\xi_j/|\xi|^2$ at nonzero frequencies. The order $-1$ of the
potential, together with \eqref{comp:strong-negative}, gives
\begin{equation}\label{comp:potential-compactness}
\begin{aligned}
 &\Phi_n\to\Phi\quad\text{in }L^2H^{2/5}\subset L^2L^2,\\
 &\Phi_n\text{ bounded in }L^\infty W^{1,4/3},\qquad
   \partial_t\Phi_n\text{ bounded in }L^2L^{8/7}.
\end{aligned}
\end{equation}

Fix $\psi\in C_c^\infty(Q_T)$ and set
\begin{equation}\label{comp:commutator-definition}
 C_{n,i}=\sum_j[u_{n,j},\mathcal R_{ij}](\psi b_n),\qquad
 [v,\mathcal R]f=v\mathcal Rf-\mathcal R(vf).
\end{equation}
The nonendpoint derivative commutator estimate
of \cite{CoifmanMeyer1978} gives
\begin{equation}\label{comp:commutator-bound}
 \norm{\nabla[v,\mathcal R_{ij}]f}_{3/2}
       \le C\norm{\nabla v}_2\norm{f}_6.
\end{equation}
The zero-order term is controlled by H\"older's inequality and the
$L^q$ boundedness of $\mathcal R_{ij}$. Thus $C_n$ is bounded in
$L^2W^{1,3/2}\hookrightarrow L^2H^{3/5}$. No derivative of $b_n$ is
required. Strong negative-space convergence of $b_n$ identifies both
products in \eqref{comp:commutator-definition}, so
\begin{equation}\label{comp:commutator-limit}
\begin{aligned}
 C_n&\weakto C\quad\text{in }L^2H^{3/5},\\
 C_i&=\sum_j[u_j,\mathcal R_{ij}](\psi\overline b),\qquad
 \int_0^T\ip{m_n}{C_n}\dt\longrightarrow
       \int_0^T\ip{m}{C}\dt.
\end{aligned}
\end{equation}
For instance, testing $u_n\mathcal R(\psi b_n)$ transfers the difference
$b_n-\overline b$ to a strong $H^{-\eta}$ factor paired with a bounded
$H^1$ factor. The last convergence in \eqref{comp:commutator-limit}
uses \eqref{comp:strong-negative}, not strong convergence of $C_n$.

Define the magnetic potential
\begin{equation}\label{comp:magnetic-potential}
 \mathcal A_n=M_n-\avg{M_n}
                    -\sum_{i,j}\mathcal R_{ij}(H_{n,i}H_{n,j}).
\end{equation}
It converges strongly in $L^1L^2$ by
\eqref{comp:magnetic-compactness}. Applying $\Delta^{-1}\Div$ to the
momentum equation yields the exact identity
\begin{equation}\label{comp:potential-equation}
\begin{aligned}
 \partial_t\Phi_n+\mathcal R_{ij}(m_{n,i}u_{n,j})+\mathcal A_n
       &=-F_n+\avg{F_n}
     +e_n-\avg{e_n},
\end{aligned}
\end{equation}
where repeated indices are summed. Testing with $\psi b_n$, integrating
in time by means of \eqref{comp:b-continuity}, and using self-adjointness
of $\mathcal R_{ij}$ give
\begin{equation}\label{comp:localized-identity}
\begin{aligned}
 \int_{Q_T}\psi b_nF_n
 ={}&\int_{Q_T}\bigl[\Phi_nb_n(\psi_t+u_n\cdot\nabla\psi)
                    -m_n\cdot C_n-\psi\Phi_nK_n\bigr]\\
    &+\int_{Q_T}\psi b_n(e_n-\avg{e_n}-\mathcal A_n)
       +\int_0^T\avg{F_n}\Big(\int_{\TT}\psi b_n\dx\Big)\dt.
\end{aligned}
\end{equation}
Here and in similar formulas the unmarked space-time integrals use
$\dx\dt$. In particular, the source term has the negative sign displayed
in \eqref{comp:localized-identity}. The commutator sign follows from
$\partial_j\Phi_n=\mathcal R_{ij}m_{n,i}$.

Every contribution in \eqref{comp:localized-identity} has a justified limit.
The time-cutoff and transport terms use strong $\Phi_n$ in $L^2L^2$
and weak convergence of $b_n$ and $b_nu_n$. The commutator term uses
\eqref{comp:commutator-limit}. Crucially, the actual source converges to
$-\int_{Q_T}\psi\Phi\overline K$, by strong $\Phi_n$ and weak $K_n$ in
$L^2(Q_T)$. The magnetic term uses strong $L^1L^2$ convergence of
$\mathcal A_n$ and the uniform bound for $b_n$.

The mean term requires a separate time compactness argument. The function
$B_{n,\psi}(t)=\int_{\TT}\psi b_n\dx$ satisfies
\begin{equation}\label{comp:mean-time}
 B_{n,\psi}'=\int_{\TT}
       (\psi_tb_n+\nabla\psi\cdot b_nu_n-\psi K_n)\dx.
\end{equation}
It is bounded in $H^1(0,T)$ and therefore converges strongly in
$C([0,T])$ along a subsequence. This is sufficient to pair it with the
weak $L^2$ limit of $\avg{F_n}$. Finally, the artificial-stress contribution
$\int_{Q_T}\psi b_n(e_n-\avg{e_n})$ vanishes by \eqref{comp:small-stress}.

To identify the resulting expression, one must also justify the time
integration by parts at the limit. The appropriate spaces are
\begin{equation}\label{comp:limit-time-spaces}
\begin{aligned}
 &\Phi\in L^2W^{1,8/7},\qquad \Phi_t\in L^2L^{8/7},\\
 &\overline b\in L^2L^8,\qquad
       \overline b_t\in L^2W^{-1,8}.
\end{aligned}
\end{equation}
Indeed, $\overline b u\in L^2L^8$ and $\overline K\in L^2L^2$; moreover,
$L^2\hookrightarrow W^{-1,8}$ is dual to
$W^{1,8/7}\hookrightarrow L^2$. Regularizing the fixed limit functions
in these dual spaces proves
\begin{equation}\label{comp:limit-time-product}
 -\int_{Q_T}\psi\overline b\Phi_t
 =\int_{Q_T}\bigl[\Phi\overline b(\psi_t+u\cdot\nabla\psi)
       +\psi\overline b u\cdot\nabla\Phi-\psi\Phi\overline K\bigr].
\end{equation}
The spatial products are integrable since
$u\cdot\nabla\Phi\in L^2L^{8/7}$ and
$\Phi\overline K\in L^2L^{4/3}$. No $L^2H^1$ estimate for $\Phi$ is
being assumed. Apply \eqref{comp:limit-time-product} to the limit of
\eqref{comp:potential-equation}. It reproduces precisely the limit of
the right-hand side of \eqref{comp:localized-identity}, including the
same term $-\int\psi\Phi\overline K$. The result is
$\int\psi\overline F\,\overline b$, proving
\eqref{comp:flux-factorization}. A countable dense family of tests,
followed by the uniform bounds, provides a common subsequence.
\end{proof}

\subsection{Entropy transport and its initial trace}

Let $s_n=\rho_n\log\rho_n$, with $0\log0=0$. Its exact equation is
\begin{equation}\label{comp:entropy-sequence}
 \partial_ts_n+\Div(s_nu_n)=-\rho_nd_n.
\end{equation}
All fixed density moments imply that $s_n$ is bounded in every finite
$L^\infty L^k$. In particular, $s_nu_n$ is bounded in $L^2L^2$ and
$\rho_nd_n$ in $L^2L^{4/3}$. Thus $\partial_ts_n$ is bounded in
$L^2H^{-1}$, and $s_n$ is strongly compact in $L^2H^{-\eta}$ for
$0<\eta<1$. Pairing this convergence with weak $u_n$ in $L^2H^1$
identifies the entropy transport flux. Consequently,
\begin{equation}\label{comp:entropy-limit}
 \partial_t\overline s+\Div(\overline s\,u)
                      =-\overline{\rho d}.
\end{equation}
There is no entropy measure source: the approximate continuity equation
has not been modified by density diffusion.

For each smooth spatial test function, the scalar pairings with
$\rho_n$ and $s_n$ have uniformly bounded $L^2$ time derivatives.
They therefore converge uniformly in time along a subsequence. The
initial convergence in \eqref{comp:initial} gives
\begin{equation}\label{comp:linear-initial-traces}
 \rho(0)=\rho_0,\qquad
 \overline s(0)=\rho_0\log\rho_0
                   \quad\text{as weak time traces}.
\end{equation}
This does not yet identify the nonlinear trace of $\rho\log\rho$.

To do so, extend the fixed limit density to a negative time interval by
$\rho(t)=\rho_0$ and set $u(t)=0$ there. The weak trace in
\eqref{comp:linear-initial-traces} ensures that the extended continuity
equation has no jump at $t=0$. Its moment and $L^2H^1$ bounds persist.
Use the continuous-form Friedrichs commutator underlying the
renormalization argument of \cite{DiPernaLions1989}: for $q\in L^\infty L^8$,
\begin{equation}\label{comp:friedrichs}
 R_\ell(q):=\Div(uq_\ell)-[\Div(uq)]_\ell
        \longrightarrow0\quad\text{in }L^2L^{8/5},
\end{equation}
where the subscript denotes spatial convolution. This includes the
compression difference $q_\ell\Div u-(q\Div u)_\ell$.

For $\kappa>0$, take
\begin{equation}\label{comp:regularized-entropy}
\begin{aligned}
 b_\kappa(z)&=(z+\kappa)\log(z+\kappa)-\kappa\log\kappa,\\
 zb_\kappa'(z)-b_\kappa(z)&=z-\kappa\log(1+z/\kappa)\in[0,z].
\end{aligned}
\end{equation}
For fixed $\kappa$, the derivative grows at most logarithmically and is
bounded in $L^\infty L^{8/3}$ after composition with the mollified density.
It can therefore be paired with \eqref{comp:friedrichs}. At a fixed
mollification scale the density has bounded range and an $L^2L^\infty$
time derivative, so the chain rule is applicable without a large-value
cutoff. Passing first in the spatial scale and then in $\kappa$ gives
\begin{equation}\label{comp:renormalized-limit-entropy}
 \partial_t(\rho\log\rho)+\Div(\rho\log\rho\,u)=-\rho d,\qquad
 (\rho\log\rho)(0)=\rho_0\log\rho_0.
\end{equation}
The source passage follows from \eqref{comp:regularized-entropy}; the
other passages use a fixed higher density moment. The initial trace is
the value on the frozen negative-time interval. All these regularizations
concern the fixed limit functions and do not use strong convergence of
the approximate densities.

\subsection{The defect equation for every pressure exponent}

Let $\mathsf N_{t,x}$ be the probability Young measure generated by
$\rho_n$. All fixed higher moments exclude concentration for every
polynomially growing density function used below. Write
$\avg{k}_{\mathsf N}=\int_0^\infty k(z)\,\mathrm d\mathsf N_{t,x}(z)$.
Then $\rho=\avg{z}_{\mathsf N}$ and
$\overline s=\avg{z\log z}_{\mathsf N}$. Define
\begin{equation}\label{comp:defect}
 D=\overline s-\rho\log\rho\ge0,\qquad
 \Cov_{\mathsf N}(z,k)=\avg{zk}_{\mathsf N}
                       -\rho\avg{k}_{\mathsf N}.
\end{equation}
The defect is bounded in every finite $L^\infty L^k$ and has zero initial
weak trace by \eqref{comp:linear-initial-traces} and
\eqref{comp:renormalized-limit-entropy}. Moreover,
\begin{equation}\label{comp:vacuum}
 \rho=0\quad\Longrightarrow\quad\mathsf N_{t,x}=\delta_0
                         \quad\Longrightarrow\quad D=0.
\end{equation}

Set $h(z)=1/a(z)$ and $g(z)=z^\gamma/a(z)$. Lemma~\ref{comp:flux}
applies to $b=h$ and $b=zh$, since
\begin{equation}\label{comp:quotient-sources}
 k_h=-\frac1a-\frac{\beta z^\beta}{a^2},\qquad
 k_{zh}=-\frac{\beta z^{\beta+1}}{a^2},
\end{equation}
and $h,zh,k_h,k_{zh}$ are bounded when $\beta>1$. Passing to the limit in
$d_n=g(\rho_n)-F_nh(\rho_n)$ and its product with $\rho_n$ yields
\begin{equation}\label{comp:quotient-identities}
\begin{aligned}
 d&=\avg{g}_{\mathsf N}-\overline F\avg{h}_{\mathsf N},\\
 \overline{\rho d}&=\avg{zg}_{\mathsf N}
                         -\overline F\avg{zh}_{\mathsf N}.
\end{aligned}
\end{equation}
Subtracting the two entropy equations gives the exact equation
\begin{equation}\label{comp:defect-equation}
 \partial_tD+\Div(Du)=f,\qquad
 f=\overline F\Cov_{\mathsf N}(z,h)-\Cov_{\mathsf N}(z,g).
\end{equation}

\begin{lemma}\label{comp:covariance}
Let \(k\in C^1((0,\infty))\) be bounded and suppose
\[
                    L_k=\sup_{z>0}z|k'(z)|<\infty.
\]
Then
\begin{equation}\label{comp:covariance-bound}
 |\Cov_{\mathsf N}(z,k)|\le C(\norm{k}_\infty+L_k)D.
\end{equation}
\end{lemma}
\begin{proof}
Let $R=\avg{z}_{\mathsf N}>0$ and
$\mathcal H(z\mid R)=z\log(z/R)-z+R$. On $R/2\le z\le2R$,
$|k(z)-k(R)|\le2L_k|z-R|/R$ and
$\mathcal H(z\mid R)\ge|z-R|^2/(4R)$. On $z\ge2R$, respectively
$z\le R/2$, the relative entropy bounds a fixed positive multiple of
$z$, respectively $R$. Boundedness of $k$ then proves
$|(z-R)(k(z)-k(R))|\le C(\norm{k}_\infty+L_k)\mathcal H(z\mid R)$.
Taking expectations proves the assertion. If $R=0$, use
\eqref{comp:vacuum}.
\end{proof}

Since $h$ is decreasing, its covariance is nonpositive. If $\gamma\ge\beta$,
\begin{equation}\label{comp:pressure-derivative}
 g'(z)=\frac{z^{\gamma-1}[2\mu\gamma+(\gamma-\beta)z^\beta]}
                   {(2\mu+z^\beta)^2}\ge0,
\end{equation}
so the pressure covariance in \eqref{comp:defect-equation} is favorable.
If $1<\gamma<\beta$, both $g$ and
$zg'=g[\gamma-\beta z^\beta/(2\mu+z^\beta)]$ are bounded. Applying
Lemma~\ref{comp:covariance} in this case gives, for every $\gamma>1$,
\begin{equation}\label{comp:one-sided-source}
 f\le AD,\qquad A=C(1+|\overline F|),\qquad
 \int_{Q_T}\rho(A+|d|)\dx\dt<\infty.
\end{equation}
This is not an ordinary Gronwall estimate with an $L^1L^\infty$ coefficient.

We also record the actual source space needed for the next step:
\begin{equation}\label{comp:source-space}
 \overline F\in L^2L^{3/2},\qquad f\in L^2L^{6/5}.
\end{equation}
Indeed, $|\Cov_{\mathsf N}(z,h)|\le2\norm{h}_\infty\rho$ and
$\rho\in L^\infty L^6$. The pressure covariance belongs to every needed
finite $L^\infty L^k$ by fixed higher moments and Jensen's inequality.

\subsection{Joint logarithmic renormalization and closure}

Fix $\varepsilon>0$ and define
\begin{equation}\label{comp:logarithmic-test}
 \Psi_\varepsilon(r,z)=r\log(1+z/\varepsilon).
\end{equation}
Mollify $\rho,D,f$ in space with a nonnegative kernel, leaving $u$
unmollified. For fixed $\ell$, the mollified variables have bounded
range and $L^2L^\infty$ time derivatives. Their continuity equations have
sources $R_\ell(\rho)$ and $f_\ell+R_\ell(D)$, respectively. The chain
rule therefore gives
\begin{equation}\label{comp:joint-mollified}
\begin{aligned}
 &\partial_t\Psi_\varepsilon(\rho_\ell,D_\ell)
       +\Div\big(u\Psi_\varepsilon(\rho_\ell,D_\ell)\big)\\
 &=\log(1+D_\ell/\varepsilon)R_\ell(\rho)
       +\frac{\rho_\ell}{\varepsilon+D_\ell}R_\ell(D)
+\frac{\rho_\ell}{\varepsilon+D_\ell}f_\ell
       -\frac{\rho_\ell D_\ell}{\varepsilon+D_\ell}d.
\end{aligned}
\end{equation}
The last term is retained: it follows from
$\Psi-r\Psi_r-z\Psi_z=-rz/(\varepsilon+z)$.

We justify each passage as $\ell\downarrow0$. First, \eqref{comp:friedrichs}
applies to both $\rho$ and $D$. The two derivative coefficients in
\eqref{comp:joint-mollified} are bounded in $L^\infty L^{8/3}$ for fixed
$\varepsilon$, so the commutator products vanish in $L^1(Q_T)$.
For completeness, the continuous-form commutator has the representation
\begin{equation}\label{comp:commutator-representation}
 R_\ell(q)(x)=d(x)q_\ell(x)
   +\int\nabla\varphi_\ell(z)\cdot[u(x)-u(x-z)]q(x-z)\dd z.
\end{equation}
Difference quotients bound its $L^{8/5}$ norm by
$C\norm{\nabla u}_2\norm{q}_8$. Spatial approximation shows cancellation
of its limiting principal term with $dq$, and dominated convergence in
time proves the asserted strong $L^2L^{8/5}$ convergence.

For the source coefficient, put
$a_\ell=\rho_\ell/(\varepsilon+D_\ell)$ and
$a_* =\rho/(\varepsilon+D)$. The decomposition
\begin{equation}\label{comp:source-coefficient}
 a_\ell-a_*=\frac{\rho_\ell-\rho}{\varepsilon+D_\ell}
   +\rho\left(\frac1{\varepsilon+D_\ell}
                         -\frac1{\varepsilon+D}\right)
\end{equation}
gives strong convergence in $L^2L^6$. The first term follows from spatial
convolution; for the second, use dominated convergence first in space
with dominating function $C_\varepsilon\rho(t)^6$, then in time with
$C_\varepsilon\norm{\rho(t)}_6^2$. Together with
$f_\ell\to f$ in $L^2L^{6/5}$, this proves $a_\ell f_\ell\to a_*f$ in
$L^1$. The denominator is bounded below by the fixed $\varepsilon$, not
by a positive lower bound for $\rho$.

The compression coefficient $\rho_\ell D_\ell/(\varepsilon+D_\ell)$
converges strongly in $L^2L^2$ by the same argument and can be paired
with $d\in L^2L^2$. Finally,
$\Psi_\varepsilon(\rho_\ell,D_\ell)\to
\Psi_\varepsilon(\rho,D)$ strongly in $L^2L^2$: use
$\log(1+z/\varepsilon)\le z/\varepsilon$, its Lipschitz constant
$1/\varepsilon$, and the fixed $L^\infty L^4$ bounds for $\rho,D$.
The transport flux thus converges in $L^1$ as well. We have proved
\begin{equation}\label{comp:joint-limit}
 \partial_t\Psi_\varepsilon(\rho,D)
       +\Div(u\Psi_\varepsilon(\rho,D))
       =\frac\rho{\varepsilon+D}f
                         -\frac{\rho D}{\varepsilon+D}d.
\end{equation}
Only now do we use $f\le AD$. No inequality of the form
$f_\ell\le A D_\ell$ is assumed at the mollified level.

The initial term also vanishes. Testing \eqref{comp:defect-equation}
with the spatial constant gives an absolutely continuous representative
of $\int_{\TT}D(t)\dx$, whose initial value is zero. Since $D\ge0$,
$\norm{D(t)}_1\to0$ as $t\downarrow0$; interpolation with a fixed higher
moment gives $\norm{D(t)}_2\to0$. Hence, for fixed $\varepsilon$,
\begin{equation}\label{comp:logarithmic-initial}
 0\le\int_{\TT}\Psi_\varepsilon(\rho(t),D(t))\dx
       \le\varepsilon^{-1}\norm{\rho(t)}_2\norm{D(t)}_2
       \longrightarrow0.
\end{equation}
These limits can be taken through the full-measure set on which the
moment bounds hold; they determine the initial value of the absolutely
continuous spatial integral in \eqref{comp:joint-limit}. Integration gives
\begin{equation}\label{comp:logarithmic-bound}
 \int_{\TT}\rho(t)\log(1+D(t)/\varepsilon)\dx
       \le\int_0^t\int_{\TT}\rho(A+|d|)\dx\dd s.
\end{equation}
The right-hand side is independent of $\varepsilon$. Integrate also in
$t$ and let $\varepsilon\downarrow0$. Fatou's lemma implies $D=0$ on
$\{\rho>0\}$, since the integrand diverges wherever $\rho>0$ and $D>0$.
On the vacuum set, use \eqref{comp:vacuum}, rather than division by $\rho$.
We conclude that
\begin{equation}\label{comp:defect-zero}
 D=0\quad\text{almost everywhere in }Q_T.
\end{equation}

\begin{proof}[Proof of Proposition~\ref{comp:density}]
All weak compactness statements were obtained before density identification.
The preceding argument applies to the same subsequence and proves
\eqref{comp:defect-zero}. Strict convexity of $z\log z$ on $[0,\infty)$
now yields $\mathsf N_{t,x}=\delta_{\rho(t,x)}$ almost everywhere.
Thus $\rho_n\to\rho$ in measure. For each fixed finite $k$, a higher
moment in \eqref{comp:moments} gives uniform integrability of
$|\rho_n-\rho|^k$. Vitali's theorem proves \eqref{comp:conclusion}, and
the same reasoning applies to each fixed power of the density.

In particular, $\rho_n^\beta\to\rho^\beta$ strongly in $L^2(Q_T)$.
Pairing this with weak $d_n\to d$ in $L^2(Q_T)$ identifies the bulk
stress in distributions. Likewise, the uniform $L^2$ bound for
$\rho_n^{\beta/2}d_n$ and the strong convergence of
$\rho_n^{\beta/2}$ identify its weak $L^2$ limit as
$\rho^{\beta/2}d$. This proves \eqref{comp:stress-conclusion}.
\end{proof}

\begin{remark}\label{comp:dissipation-remark}
The proof establishes strong density compactness and identifies the
original constitutive stress. It does not require, or assert, strong
convergence of the velocity gradients or disappearance of
$\delta_n\rho_n^2d_n^2$. That nonnegative dissipation can be discarded
when passing to the energy inequality. The exact continuity equation,
and hence the absence of a measure source in
\eqref{comp:defect-equation}, is essential to the argument given here.
\end{remark}

\section{Passage to the limit and proof of the main theorem}
\label{sec:assembly}

\subsection{A global approximating sequence}
Choose a nonnegative periodic approximate identity \(J_n\) and
a sequence \(0<\delta_n\le1\) tending to zero. Set
\begin{equation}\label{as:data}
 \rho_{0n}=J_n*\rho_0,\qquad
 u_{0n}=J_n*u_0,\qquad H_{0n}=J_n*H_0.
\end{equation}
The density bounds \(c_0\le\rho_{0n}\le C_0\) are preserved,
as is \(\Div H_{0n}=0\). The data converge strongly in
\(W^{1,q_0}\times H^1\times H^1\); in particular
\(\rho_{0n}\to\rho_0\) uniformly and \(E_n(0)\to E(0)\).
Each datum is smooth, but no uniform bound on its higher
derivatives is required.

By Proposition~\ref{fix:global}, the system with viscosity
\(\lambda_{\delta_n}(\rho)=\rho^\beta+\delta_n\rho^2\)
has a global positive-density strong solution
\((\rho_n,u_n,H_n)\) for these data.
Proposition~\ref{uni:uniform-moments} and
\eqref{uni:energy} provide, on every fixed finite interval, the
uniform bounds required in Section~\ref{sec:compactness}.
After subsequence extraction, Proposition~\ref{comp:density} gives
\begin{equation}\label{as:density}
 \rho_n\to\rho\quad\text{almost everywhere and strongly in }
                 L^p(Q_T)\quad(1\le p<\infty),
 \qquad Q_T=(0,T)\times\TT.
\end{equation}
Throughout this section a larger fixed density moment is
chosen whenever needed. This does not require uniformity of
the moment bounds with respect to their exponent.

\subsection{Identification of the original equations}
Equation~\eqref{as:density} and uniform integrability imply,
for every fixed finite \(p\),
\begin{equation}\label{as:powers}
 \rho_n^\gamma\to\rho^\gamma,\qquad
 \rho_n^\beta\to\rho^\beta,\qquad
 \rho_n^{\beta/2}\to\rho^{\beta/2}
                     \quad\text{strongly in }L^p(Q_T).
\end{equation}
Since \(d_n=\Div u_n\weakto d=\Div u\) in \(L^2(Q_T)\),
\[
       \rho_n^\beta d_n\longrightarrow\rho^\beta d
                      \quad\text{in distributions}.
\]
The additional stress tends to zero by
Corollary~\ref{uni:small-stress}:
\begin{equation}\label{as:additional}
       \delta_n\rho_n^2d_n\longrightarrow0
             \quad\text{in }L^2(0,T;L^r(\TT)),\qquad 1\le r<2.
\end{equation}
Thus neither an unidentified pressure nor an unidentified
viscous stress remains.

The compactness argument in Section~\ref{sec:compactness} also gives
\begin{equation}\label{as:momentum}
 \begin{gathered}
 m_n=\rho_nu_n\to m=\rho u
             \quad\text{strongly in }L^2(0,T;H^{-3/5}(\TT)),\\
 u_n\weakto u\quad\text{in }L^2(0,T;H^1(\TT)),\qquad
 H_n\to H\quad\text{strongly in }L^2(0,T;L^4(\TT)).
 \end{gathered}
\end{equation}
The strong-negative/weak-positive pairing identifies
\(m_n\otimes u_n\to\rho u\otimes u\) in distributions.
Moreover,
\(\M(H_n)\to\M(H)\) strongly in \(L^1(0,T;L^2)\).
The induction flux converges by the last two limits in
\eqref{as:momentum}. The remaining linear terms pass to the limit
by weak convergence. Consequently the distributional equations
have precisely the form \eqref{eq:system}.

\begin{proposition}[Weighted kinetic compactness]\label{as:kinetic}
For the same subsequence,
\begin{equation}\label{as:weighted-strong}
 \sqrt{\rho_n}\,u_n\to\sqrt\rho\,u
                 \quad\text{strongly in }L^2(Q_T).
\end{equation}
The convergence holds also in \(L^r(0,T;L^2(\TT))\) for every
finite \(1\le r<\infty\), and
\begin{equation}\label{as:convective-strong}
        \rho_nu_n\otimes u_n\to\rho u\otimes u
                           \quad\text{strongly in }L^1(Q_T).
\end{equation}
\end{proposition}

\begin{proof}
Let \(v_n=\sqrt{\rho_n}\,u_n\). The energy estimate bounds \(v_n\)
in \(L^\infty_tL^2_x\). The elementary inequality
\[
             |\sqrt a-\sqrt b|^4\le|a-b|^2,\qquad a,b\ge0,
\]
and \eqref{as:density} imply strong convergence of
\(\sqrt{\rho_n}\) in \(L^4(Q_T)\), hence in \(L^2_tL^4_x\).
Testing against smooth functions and using weak convergence of
\(u_n\) in \(L^2_tH^1_x\) identifies every weak \(L^2(Q_T)\)
limit of \(v_n\) as \(v=\sqrt\rho\,u\).

The duality in \eqref{as:momentum} is a Bochner-space duality
on the entire time interval. Therefore
\begin{equation}\label{as:norm-convergence}
\begin{aligned}
 \int_{Q_T}|v_n|^2\dx\dt
   &=\int_0^T\ip{m_n}{u_n}_{H^{-3/5},H^{3/5}}\dt\\
   &\longrightarrow
       \int_0^T\ip{m}{u}_{H^{-3/5},H^{3/5}}\dt
     =\int_{Q_T}|v|^2\dx\dt.
\end{aligned}
\end{equation}
Indeed, the pairing with \(m_n-m\) tends to zero by strong
convergence and boundedness of \(u_n\); the pairing of \(m\)
with \(u_n-u\) tends to zero by weak convergence.
No nonsmooth endpoint test in the momentum equation is used.
Weak convergence and convergence of Hilbert norms prove
\eqref{as:weighted-strong}. Interpolation with the uniform
\(L^\infty_tL^2_x\) bound gives every finite time exponent.
Finally \(v_n\otimes v_n\to v\otimes v\) in \(L^1\).
\end{proof}

\subsection{Weak time traces and renormalization}
For any smooth spatial scalar test \(\phi\), the continuity equation
gives the absolutely continuous representative
\begin{equation}\label{as:weak-density-trace}
        \frac{\dd}{\dt}\int_{\TT}\rho\phi\dx
                      =\int_{\TT}m\cdot\nabla\phi\dx.
\end{equation}
A countable dense family of tests and the uniform moment bounds
define a representative in \(C_{\mathrm{weak}}([0,T];L^p)\)
for each \(1<p<\infty\). This representative is nonnegative at every time
and has the same \(L^p\) bound as the essential supremum.

Similarly, the momentum equation and its integrable right-hand
side define a weakly continuous momentum representative.
For each \(1<r<2\), the estimate
\[
  \norm{m(t)}_r\le
       \norm{\sqrt\rho\,u(t)}_2
       \norm{\rho(t)}_{r/(2-r)}^{1/2}
\]
provides the corresponding \(L^\infty_tL^r_x\) bound.
Here the norm on the right is a density norm.
Thus
\begin{equation}\label{as:weak-time}
  m\in C_{\mathrm{weak}}([0,T];L^r(\TT)),\qquad 1<r<2.
\end{equation}
The uniform time-derivative estimates in the compactness argument
give equicontinuity of all smooth spatial pairings for the
approximate sequence. Its initial values converge by
\eqref{as:data}. Hence these representatives attain
\(\rho(0)=\rho_0\), \(m(0)=\rho_0u_0\), and \(H(0)=H_0\)
in the weak sense. Passing to the limit in tests supported up
to \(t=0\) proves \eqref{eq:weak-mass}--\eqref{eq:weak-magnetic}
with the stated initial terms.

For completeness, the renormalized continuity equation is recovered
from the limiting continuity equation itself. Convolve in space
and use Lemma~\ref{pre:friedrichs}. If \(b'\) is compactly supported,
the chain-rule derivative is bounded, and the commutator converges
strongly in an integrable space. The lower-order term
\((\rho b'(\rho)-b(\rho))d\) is integrable.
To obtain the initial trace, extend \(\rho\) as \(\rho_0\)
and set \(u=0\) on a short negative-time interval.
The weak trace already proved excludes a jump in the continuity
equation. Apply the same convolution argument to this extended
equation. The resulting identity is \eqref{eq:renormalized}.
The renormalization with \(\rho\log\rho\), including its behavior
at zero, was established separately in
Section~\ref{sec:compactness}; it is not an immediate substitution
in the compactly supported derivative class.

\subsection{Strong time representatives}\label{as:strong-traces}
The density square satisfies
\begin{equation}\label{as:square}
    \partial_t\rho^2+\Div(u\rho^2)=-\rho^2d.
\end{equation}
Here \(\rho^2d\in L^2_tL^1_x\), using a fixed density moment
at least four. Thus the almost-everywhere defined square norm
has an absolutely continuous representative. We verify that
its value at every time is the square norm of the weak
representative constructed above.

Fix \(\tau<T\), keep the solution for \(t>\tau\), and on a
short interval to the left of \(\tau\) set
\[
            \rho^{(\tau)}(t)=\rho(\tau),\qquad
            u^{(\tau)}(t)=0.
\]
The weak trace \eqref{as:weak-density-trace} makes this an
extension satisfying the continuity equation across \(\tau\).
Take a fixed density exponent \(P>4\).
For a spatial mollifier, put
\[
 R_\ell(\rho)=\Div(u\rho_\ell)-[\Div(u\rho)]_\ell.
\]
Lemma~\ref{pre:friedrichs} gives
\[
 R_\ell(\rho)\to0\quad\text{in }L^2_tL^s_x,
       \qquad \frac1s=\frac12+\frac1P.
\]
At fixed \(\ell\), the square chain rule yields
\[
 \partial_t\rho_\ell^2+\Div(u\rho_\ell^2)
       =-\rho_\ell^2d+2\rho_\ell R_\ell(\rho).
\]
Since \(1/s+1/P=1/2+2/P<1\), the last product tends to zero
in \(L^1\). The source and flux terms also converge, using
the same fixed moments and \(u\in L^2_tH^1_x\).
Consequently \eqref{as:square} holds for the frozen extension.
Its square norm is constant and equal to
\(\norm{\rho(\tau)}_2^2\) on the frozen side. Absolute continuity
identifies that value with the trace from the solution side.
At \(T\) use a right-hand frozen extension, and at \(0\)
use the prescribed initial density. We have proved the
identity for all \(s,t\in[0,T]\):
\begin{equation}\label{as:square-trace}
        \norm{\rho(t)}_2^2-\norm{\rho(s)}_2^2
                  =-\int_s^t\!\!\int_{\TT}\rho^2d\dx\dd\tau.
\end{equation}
Weak \(L^2\) continuity and continuity of the norm give
\(\rho\in C([0,T];L^2)\).
For any finite \(p>2\), choose \(P>p\) and interpolate:
\[
 \norm{\rho(t)-\rho(s)}_p
   \le\norm{\rho(t)-\rho(s)}_2^\theta
        \bigl(2\norm{\rho}_{L^\infty_tL^P_x}\bigr)^{1-\theta},
 \qquad \frac1p=\frac\theta2+\frac{1-\theta}P.
\]
For \(p<2\) use finite measure. This proves strong time
continuity in every fixed finite \(L^p\), without taking an
endpoint in a single moment estimate.

For the magnetic field choose any fixed \(Q>2\). The magnetic
moments and two-dimensional Sobolev embedding give
\begin{equation}\label{as:magnetic-trace}
 \norm{Hu}_{L^2_tL^2_x}
   \le\norm{H}_{L^\infty_tL^Q_x}
         \norm{u}_{L^2_tL^{2Q/(Q-2)}_x}<\infty.
\end{equation}
The induction equation therefore gives
\(H_t\in L^2_tH^{-1}_x\).
Together with \(H\in L^2_tH^1_x\), the Hilbert-triple trace
theorem yields \(H\in C([0,T];L^2)\).
One may prove the trace theorem here by spatial Fourier
truncation, temporal mollification, and the identity
\[
        \frac12\frac{\dd}{\dt}\norm{H}_2^2
                         =\ip{H_t}{H}_{H^{-1},H^1}.
\]
Its representative agrees with the previously identified weak
trace, in particular \(H(0)=H_0\).

\subsection{Energy inequality at every terminal time}
The energy estimate bounds
\(z_n=\rho_n^{\beta/2}d_n\) in \(L^2(Q_T)\).
Its weak limits are identified by \eqref{as:powers}:
\begin{equation}\label{as:weighted-dissipation}
                z_n\weakto\rho^{\beta/2}d
                                     \quad\text{in }L^2(Q_T).
\end{equation}
Thus the original density-weighted dissipation is weakly
lower semicontinuous on every fixed time subinterval.
The same is true of the shear, unweighted divergence, and
magnetic dissipation. The additional dissipation
\(\delta_n\rho_n^2d_n^2\) is nonnegative and may be discarded;
\eqref{as:additional} does not imply that its integral tends to zero.

The approximate energy identity reads
\begin{equation}\label{as:approx-energy}
\begin{aligned}
 E_n(t)&+\int_0^t\!\!\int_{\TT}
 \bigl[\mu|\nabla u_n|^2+
    (\mu+\rho_n^\beta+\delta_n\rho_n^2)d_n^2
          +\nu|\nabla H_n|^2\bigr]\dx\dd s=E_n(0).
\end{aligned}
\end{equation}
By Proposition~\ref{as:kinetic}, \eqref{as:powers}, and
strong magnetic convergence, \(E_n\to E\) in \(L^1(0,T)\).
After extraction, \(E_n(t)\to E(t)\) for almost every \(t\).
The lower semicontinuity above first proves
\eqref{eq:energy-main} for almost every terminal time.

We now use the canonical representatives, not arbitrary
pointwise values of \(u\). The kinetic functional has the
convex dual representation
\begin{equation}\label{as:kinetic-dual}
 \int_{\TT}\mathcal K(\rho,m)\dx
   =\sup_{w\in C^\infty(\TT;\RR^2)}
       \int_{\TT}\left(m\cdot w-\frac12\rho|w|^2\right)\dx.
\end{equation}
It is therefore jointly lower semicontinuous under the weak
density and momentum time topologies.
The pressure energy is continuous by the strong \(L^\gamma\)
density trace, and the magnetic energy is continuous by
\eqref{as:magnetic-trace}.
Hence \(E(t)\) is lower semicontinuous in time.
Since \(\int_0^t\mathfrak d(s)\dd s\) is continuous, approximation
of an arbitrary terminal time by the full-measure set gives
\eqref{eq:energy-main} for every \(t\in[0,T]\).
Finiteness of this canonical energy also implies
\(m(t)=0\) almost everywhere on \(\{\rho(t)=0\}\) at every time.

For the stronger starting-time statement, perform the extraction
so that \(E_n(s)\to E(s)\) on a common full-measure set \(S\)
on all finite intervals. Restrict \eqref{as:approx-energy}
to \((s,t)\) with \(s,t\in S\). The same lower semicontinuity
proves \eqref{eq:energy-start} at these endpoints.
Keeping \(s\in S\) fixed and approximating any \(t\ge s\)
extends the inequality to every terminal time.
The initial energy convergence in \eqref{as:data} includes \(s=0\).
This argument does not provide every possible starting time.

Lower semicontinuity at zero and \eqref{eq:energy-main} show
\(E(t)\to E(0)\) as \(t\downarrow0\).
The pressure and magnetic energies converge strongly there,
so the kinetic norms converge as well.
Put \(v(t)=m(t)/\sqrt{\rho(t)}\), zero on the vacuum set.
Every weak \(L^2\) cluster point \(v_*\) along \(t\downarrow0\)
satisfies
\(\sqrt{\rho_0}\,v_*=\rho_0u_0\), by strong density and weak
momentum traces. Since \(\rho_0\ge c_0>0\),
\(v_*=\sqrt{\rho_0}\,u_0\). Convergence of the norms then proves
the strong convergence in \eqref{eq:initial-kinetic}.
This is a statement about the weighted velocity and does
not prescribe an unweighted velocity trace on vacuum.

\subsection{Global diagonal extraction and scope}
All approximate solutions exist globally.
Apply the preceding construction successively on \(T=1,2,\ldots\)
and to a countable unbounded family of density and magnetic
moment exponents, and take a diagonal subsequence.
The limiting distributions and canonical traces agree on
overlapping intervals. A larger integer moment controls any
given noninteger finite exponent.
The weak equations, renormalization, regularity, and energy
statements therefore hold on every finite interval for one
global solution. This proves Theorem~\ref{thm:main}.

There is no remaining pressure, viscous-stress, or convective
defect in the distributional equations. This fact should not
be confused with strong convergence of the dissipative
gradients. To specify what is not proved, the approximate
dissipation measures may have a nonnegative limit \(\eta\).
On the open time half-line one has
\[
               \eta\ge\mathfrak d(t)\dt,\qquad
                    \partial_tE+\eta=0
                    \quad\text{in distributions},
\]
after a further local subsequence extraction.
The difference \(\eta-\mathfrak d(t)\dt\) is not shown to vanish.
The conclusions are consequently the energy inequalities and
time quantifiers in Theorem~\ref{thm:main}, rather than a
global energy equality or a strong-solution continuation
statement for the original viscosity law.

\section{Weak--strong uniqueness by an augmented relative energy}
\label{sec:weak-strong}

The weak solutions constructed above can be compared with a strictly
positive strong solution without assuming a pointwise bound for the
weak density. When \(\gamma<\beta\), the physical relative internal
energy alone does not control the bulk-viscosity difference. We add
a relative potential of order \(\beta\), justified by the given
renormalized continuity equation. This introduces neither an
additional equation nor an additional entropy assumption.

\subsection{A relative stability statement}

For \(a>1\), set
\begin{equation}\label{ws:potentials}
 f_a(z)=\frac{z^a}{a-1},\qquad
 h_a(z\mid r)=f_a(z)-f_a(r)-f_a'(r)(z-r).
\end{equation}
Write \((\rho,u,H)\) for the weak solution and \((r,U,B)\) for
the strong comparison solution, and put
\[
 w=u-U,\qquad Z=H-B,\qquad D_U=\Div U,\qquad\delta\rho=\rho-r.
\]
Here \(D_U\) is the divergence, not a material derivative. Define
\begin{equation}\label{ws:functionals}
\begin{aligned}
 \mathcal F(t)&=\int_{\TT}\left[
 \frac12\rho|w|^2+h_\gamma(\rho\mid r)
                  +h_\beta(\rho\mid r)+\frac12|Z|^2\right]\dx,\\
 \mathcal D(t)&=\int_{\TT}\left[
 \mu|\nabla w|^2+(\mu+\rho^\beta)|\Div w|^2
                              +\nu|\nabla Z|^2\right]\dx.
\end{aligned}
\end{equation}
At a canonical time, the kinetic term means
\(\mathcal K(\rho,m)-m\cdot U+\rho|U|^2/2\), with the convex
extension in \eqref{eq:kinetic}. No velocity value on a vacuum
set is prescribed.

\begin{proposition}[Augmented relative stability]\label{ws:relative-theorem}
Let \(\beta>1\) and \(\gamma>1\). Suppose that
\((\rho,u,H)\) satisfies Definition~\ref{def:weak} on \([0,T]\)
and has positive total mass. Let \((r,U,B)\) solve the same
system with
\begin{equation}\label{ws:strong-class}
\begin{gathered}
 0<c_T\le r\le C_T,\qquad \Div B=0,\qquad
 r,U,B\in C([0,T];H^2),\\
 U,B\in L^2(0,T;H^3),\qquad
 U_t,B_t\in L^2(0,T;H^1),\qquad r_t\in C([0,T];H^1).
\end{gathered}
\end{equation}
The initial states need not coincide. There is a nonnegative
\(G_T\in L^1(0,T)\) such that, for every \(t\in[0,T]\),
\begin{equation}\label{ws:relative-estimate}
 \mathcal F(t)+\frac12\int_0^t\mathcal D(s)\dd s
 \le\mathcal F(0)+\int_0^tG_T(s)\mathcal F(s)\dd s.
\end{equation}
Consequently,
\begin{equation}\label{ws:gronwall}
 \mathcal F(t)\le\mathcal F(0)
                  \exp\left(\int_0^tG_T(s)\dd s\right).
\end{equation}
The coefficient can depend on fixed finite weak-density moments,
total mass, the reference density interval, and reference norms.
It does not depend on a weak-density pointwise bound or gradient.
\end{proposition}

Only finitely many density moments are used. For example, a fixed
\begin{equation}\label{ws:moment-index}
                 p\ge\max\{9,3\beta,3\gamma/2\}
\end{equation}
suffices, together with the corresponding traces and the other
weak-solution properties. In particular \(p>2\beta\).

\subsection{Coercivity and the velocity constant mode}

Uniformly for \(r\in[c_T,C_T]\) and \(z\ge0\),
\begin{equation}\label{ws:pointwise}
 \frac{(z^\beta-r^\beta)^2}{\mu+z^\beta}\le C h_\beta(z\mid r),
 \qquad
 \frac{(z-r)^2}{1+z}\le C h_\beta(z\mid r).
\end{equation}
Near \(z=r\), both follow from strict convexity and
\(f_\beta''(r)=\beta r^{\beta-2}\). On a compact set away from
the diagonal, the relative potential is bounded below. At
infinity, the first quotient has order \(z^\beta\), the second
has at most order \(z\), and \(h_\beta(z\mid r)\asymp z^\beta\).
At zero the positive reference density prevents degeneracy.
In contrast, if \(\gamma<\beta\), then
\[
 \frac{(z^\beta-r^\beta)^2}
 {(\mu+z^\beta)h_\gamma(z\mid r)}
                       \asymp z^{\beta-\gamma}\longrightarrow\infty.
\]
Thus the extra potential pays a genuine viscosity-growth cost.

Let \(M_0=\int_{\TT}\rho>0\) and \(\bar w=\int_{\TT}w\). Since
\[
 M_0\bar w=\int_{\TT}\rho w\dx
                        -\int_{\TT}\rho(w-\bar w)\dx,
\]
weighted Cauchy--Schwarz and periodic Poincar\'e give
\[
 M_0|\bar w|\le M_0^{1/2}\norm{\sqrt\rho\,w}_2
                          +\norm\rho_2\norm{w-\bar w}_2.
\]
The two-dimensional Sobolev embeddings therefore give every
fixed finite \(q\):
\begin{equation}\label{ws:weighted-poincare}
 \norm w_q\le C_q\bigl(\norm{\nabla w}_2
                             +\norm{\sqrt\rho\,w}_2\bigr).
\end{equation}
One fixed density moment makes the constant uniform on \([0,T]\).
No lower bound for the weak density is needed.

For \(G\in L^4(\TT)\), \eqref{ws:pointwise}, H\"older, and
\eqref{ws:weighted-poincare} imply
\begin{equation}\label{ws:density-force}
\begin{aligned}
 \int_{\TT}|\rho-r||G||w|\dx
 &\le C\mathcal F^{1/2}\norm{\sqrt{1+\rho}\,Gw}_2\\
 &\le C(1+\norm\rho_p^{1/2})\norm G_4\mathcal F^{1/2}
                                  \norm w_{4p/(p-2)}\\
 &\le\varepsilon\norm{\nabla w}_2^2
                  +C_\varepsilon(1+\norm G_4^2)\mathcal F.
\end{aligned}
\end{equation}
Indeed, \(1/(2p)+1/4+(p-2)/(4p)=1/2\).
This avoids any false global control of
\(\norm{\rho-r}_2^2\) by the pressure relative potential.
Likewise,
\begin{equation}\label{ws:magnetic-cross-bound}
\begin{aligned}
 \int_{\TT}|\nabla B||w||Z|\dx
 &\le\norm{\nabla B}_4\norm w_4\norm Z_2
 \le\varepsilon\norm{\nabla w}_2^2
          +C_\varepsilon(1+\norm{\nabla B}_4^2)\mathcal F.
\end{aligned}
\end{equation}

\subsection{The additional potential is a renormalized quantity}

Choose \(0\le\chi\le1\), smooth, equal to one on \([0,1]\),
and zero on \([2,\infty)\), and set
\[
 b_L(z)=\int_0^z f_\beta'(s)\chi(s/L)\dd s.
\]
Its derivative is compactly supported, and uniformly in \(L\),
\[
 |b_L(z)|+|zb_L'(z)-b_L(z)|\le C z^\beta,
 \qquad b_L(z)=f_\beta(z)\quad(z\le L).
\]
The moment \(p>2\beta\) gives
\begin{equation}\label{ws:entropy-tail}
 \norm{\rho^\beta\boldsymbol1_{\{\rho>L\}}}_{L^2(Q_T)}^2
 \le L^{2\beta-p}\int_{Q_T}\rho^p\dx\dt\longrightarrow0.
\end{equation}
This controls the source paired with \(\Div u\in L^2\) and
the transport flux paired with \(u\in L^2L^2\).
Passing \(L\to\infty\) in \eqref{eq:renormalized} yields
\begin{equation}\label{ws:beta-renormalization}
 \partial_t f_\beta(\rho)+\Div(f_\beta(\rho)u)
                                   +\rho^\beta\Div u=0.
\end{equation}
Its spatial integral is absolutely continuous and has the
prescribed initial value, since \(\rho\in C([0,T];L^\beta)\).
No spatial derivative of the weak density is introduced.

Use the strong continuity equation and test the weak one with
\(f_\beta'(r)\). Differentiating
\(\int[f_\beta(\rho)-\rho f_\beta'(r)+r^\beta]\) gives
\begin{equation}\label{ws:beta-relative-identity}
\begin{aligned}
 \frac{\dd}{\dt}\int_{\TT}h_\beta(\rho\mid r)\dx
 ={}&-\int_{\TT}(\rho^\beta-r^\beta)\Div w\dx\\
 &-(\beta-1)\int_{\TT}h_\beta(\rho\mid r)D_U\dx
       -\int_{\TT}\frac{\rho-r}{r}\,w\cdot\nabla r^\beta\dx.
\end{aligned}
\end{equation}
Before the last integration by parts, the first and last
terms are
\(-\int\rho^\beta\Div w-\int(\rho/r)w\cdot\nabla r^\beta\).
The identity
\(-\int r^\beta\Div w=\int w\cdot\nabla r^\beta\)
produces the density difference in
\eqref{ws:beta-relative-identity}. Its terms are integrable by
the fixed moments, \eqref{ws:strong-class}, and
\eqref{ws:density-force}. The reference-test approximation is
justified below.

\subsection{Physical relative energy and magnetic cancellation}

Use the nonsymmetric stress
\[
 T_s(V)=\mu\nabla V+(\mu+s^\beta)(\Div V)I,
\]
whose divergence is exactly the viscous operator in
\eqref{eq:system}. Its integrated dissipation on the torus
agrees with the symmetric-stress formulation. Put
\begin{equation}\label{ws:reference-force}
 F_{\rm ref}
 =\frac{\Div T_r(U)+\Div\M(B)}r
 =U_t+U\cdot\nabla U+\frac{\nabla r^\gamma}{r}.
\end{equation}
This force includes the reference magnetic stress and is not
merely the material acceleration.

Let \(\mathcal E_\gamma\) denote \(\mathcal F\) without
\(h_\beta\). The physical relative inequality is
\begin{equation}\label{ws:physical-relative}
 \mathcal E_\gamma(t)+\int_0^t\mathcal D(s)\dd s
 \le\mathcal E_\gamma(0)+\int_0^t\mathcal R_\gamma(s)\dd s,
\end{equation}
with
\begin{equation}\label{ws:physical-remainder}
\begin{aligned}
 \mathcal R_\gamma={}&
 -\int_{\TT}\rho(w\otimes w):\nabla U\dx
 -(\gamma-1)\int_{\TT}h_\gamma(\rho\mid r)D_U\dx\\
 &-\int_{\TT}(\rho^\beta-r^\beta)D_U\Div w\dx
 -\int_{\TT}(\rho-r)F_{\rm ref}\cdot w\dx\\
 &+\int_{\TT}\M(Z):\nabla U\dx
 +\int_{\TT}(w\otimes Z-Z\otimes w):\nabla B\dx.
\end{aligned}
\end{equation}
We record both the algebra and its weak justification.

At the smooth level, continuity gives
\begin{equation}\label{ws:kinetic-algebra}
\begin{aligned}
 \frac12\frac{\dd}{\dt}\int_{\TT}\rho|w|^2\dx
 ={}&-\int_{\TT}T_\rho(u):\nabla w\dx
       +\int_{\TT}\rho^\gamma\Div w\dx
       -\int_{\TT}\M(H):\nabla w\dx\\
 &-\int_{\TT}\rho w\cdot(U_t+U\cdot\nabla U)\dx
       -\int_{\TT}\rho(w\otimes w):\nabla U\dx.
\end{aligned}
\end{equation}
The pressure relative potential has derivative
\begin{equation}\label{ws:pressure-algebra}
\begin{aligned}
 \frac{\dd}{\dt}\int_{\TT}h_\gamma(\rho\mid r)\dx
 ={}&-\int_{\TT}\rho^\gamma\Div w\dx
       -(\gamma-1)\int_{\TT}h_\gamma(\rho\mid r)D_U\dx\\
 &-\int_{\TT}\frac\rho r\,w\cdot\nabla r^\gamma\dx.
\end{aligned}
\end{equation}
Insert the reference momentum equation and split
\(\rho/r=1+(\rho-r)/r\). After integrating the density-independent
reference stresses by parts, the viscous terms satisfy
\begin{equation}\label{ws:viscous-algebra}
\begin{aligned}
 \int_{\TT}[T_\rho(u)-T_r(U)]:\nabla w\dx
 ={}&\int_{\TT}\bigl[\mu|\nabla w|^2
               +(\mu+\rho^\beta)|\Div w|^2\bigr]\dx\\
 &+\int_{\TT}(\rho^\beta-r^\beta)D_U\Div w\dx.
\end{aligned}
\end{equation}
There is no derivative of the weak density.

The magnetic difference equation is
\begin{equation}\label{ws:magnetic-difference}
 Z_t+\Div(Z\otimes u-u\otimes Z)
 =\nu\Delta Z-\Div(B\otimes w-w\otimes B).
\end{equation}
For smooth fields, \(\Div B=\Div Z=0\) gives
\begin{equation}\label{ws:magnetic-algebra}
\begin{aligned}
 \frac12\frac{\dd}{\dt}\int_{\TT}|Z|^2\dx
       +\nu\int_{\TT}|\nabla Z|^2\dx
 ={}&\int_{\TT}\M(Z):\nabla u\dx\\
 &+\int_{\TT}[Z\otimes B-(Z\cdot B)I]:\nabla w\dx
       -\int_{\TT}Z\cdot(w\cdot\nabla B)\dx.
\end{aligned}
\end{equation}
The fluid comparison contributes
\(-\int[\M(H)-\M(B)]:\nabla w\). Expanding
\[
 \M(H)-\M(B)
 =\M(Z)+B\otimes Z+Z\otimes B-(B\cdot Z)I
\]
cancels the \(\nabla w\) terms except
\(-\int(B\otimes Z):\nabla w\). This is
\(\int(w\otimes Z):\nabla B\), by \(\Div Z=0\).
Combining it with the last term in
\eqref{ws:magnetic-algebra} proves the last line of
\eqref{ws:physical-remainder}. The tensor-factor order matters.

At the weak level, these identities organize the algebra only.
We do not test the weak momentum equation with \(w\), nor
assume a separate magnetic energy equality. Instead, expand
\begin{equation}\label{ws:weak-expansion}
\begin{aligned}
 \mathcal E_\gamma={}&E_{\rm weak}
 -\int_{\TT}m\cdot U\dx+\frac12\int_{\TT}\rho|U|^2\dx
 -\int_{\TT}\rho f_\gamma'(r)\dx+\int_{\TT}r^\gamma\dx\\
 &-\int_{\TT}H\cdot B\dx+\frac12\int_{\TT}|B|^2\dx.
\end{aligned}
\end{equation}
Use \(U\) in weak momentum, \(B\) in weak induction, and
\(|U|^2/2,f_\gamma'(r)\) in weak continuity. Combining these
identities with the total energy inequality from time zero
and the preceding algebra yields
\eqref{ws:physical-relative} for smooth reference fields.
No weak internal-energy equality of order \(\gamma\) is used.

\subsection{Admissibility of the strong comparison fields}

The calculation extends to \eqref{ws:strong-class}. The index
in \eqref{ws:moment-index} gives
\[
 m\in L^\infty L^{2p/(p+1)},\qquad
 \rho^\gamma\in L^\infty L^{3/2},\qquad
 \rho^\beta\Div u\in L^2L^{3/2}.
\]
For the last assertion use
\(\rho^\beta\Div u=(\rho^{\beta/2}\Div u)\rho^{\beta/2}\);
its factors belong to \(L^2L^2\) and \(L^\infty L^6\),
respectively. The convective and magnetic stresses are in
\(L^\infty L^1\), and \(\nabla U\in L^2L^\infty\).
The induction flux lies in \(L^2L^{4/3}\), while
\(\nabla B\in C L^4\). These bounds control every weak pairing
and the weighted reference dissipation.

Extend the reference fields by reflection past the temporal
endpoints and convolve with a nonnegative smooth space--time
kernel. This preserves the density bounds and \(\Div B=0\),
and gives strong convergence in \(C H^2\), in \(L^2H^3\)
for \(U,B\), and in \(L^2H^1\) for their time derivatives
and for \(r_t\). Initial reference values converge in \(H^2\).
The approximate fields need not solve the equations.
For a smooth comparison triple retain the residuals
\begin{equation}\label{ws:reference-residuals}
\begin{aligned}
 R_c&=r_t+\Div(rU),\\
 R_U&=r(U_t+U\cdot\nabla U)+\nabla r^\gamma
                         -\Div T_r(U)-\Div\M(B),\\
 R_B&=B_t+\Div(B\otimes U-U\otimes B)-\nu\Delta B.
\end{aligned}
\end{equation}
The additional augmented-relative terms are precisely
\begin{equation}\label{ws:residual-pairing}
\begin{aligned}
 &-\int_{\TT}\frac\rho r\,w\cdot R_U\dx
 -\int_{\TT}Z\cdot R_B\dx
 -\int_{\TT}(\rho-r)[f_\gamma''(r)+f_\beta''(r)]R_c\dx.
\end{aligned}
\end{equation}
Here \(R_U\) is the material-form momentum residual; a
conservative-form residual differs by \(UR_c\).

In two dimensions \(H^2\) is an algebra and a multiplier
space on \(H^1\). Smooth composition on a fixed positive
interval preserves \(H^2\). The approximate residuals
therefore tend to zero strongly in \(L^2(0,T;H^1)\).
For example \(r^\beta\Div U\in L^2H^2\), and hence its
gradient is in \(L^2H^1\); no third density derivative is
needed. The inertial, pressure, and magnetic terms converge
by the same product bounds. The factors paired with the
residuals are, up to bounded reference factors,
\[
 \rho w\in L^\infty L^{2p/(p+1)},\qquad
 Z\in L^\infty L^2,\qquad \rho-r\in L^\infty L^p.
\]
Finite-exponent Sobolev embeddings make all residual pairings
tend to zero. The other terms pass by the bounds above.
This proves \eqref{ws:physical-relative} and
\eqref{ws:beta-relative-identity} for the stated comparison class.

The inequality is first obtained at almost every terminal time.
The relative kinetic functional is lower semicontinuous in the
canonical weak momentum and strong density time topologies.
The other relative energy terms are continuous because
\(\rho\in C L^p\) for every finite \(p\), \(H\in C L^2\),
and the references are in \(C H^2\). Approximation by the
full-measure set extends the inequality to all terminal times.
Only the energy inequality starting from zero is needed.

\subsection{Absorption and proof of uniqueness}

\begin{proof}[Proof of Proposition~\ref{ws:relative-theorem}]
Add the integrated identity \eqref{ws:beta-relative-identity}
to \eqref{ws:physical-relative}. The two bulk differences combine as
\[
             -\int_{\TT}(\rho^\beta-r^\beta)(D_U+1)\Div w\dx.
\]
By \eqref{ws:pointwise} and weighted Young,
\begin{equation}\label{ws:bulk-absorption}
\begin{aligned}
 \left|\int_{\TT}(\rho^\beta-r^\beta)(D_U+1)\Div w\dx\right|
 \le{}&\varepsilon\int_{\TT}(\mu+\rho^\beta)|\Div w|^2\dx\\
 &+C_\varepsilon(1+\norm{D_U}_\infty^2)\mathcal F.
\end{aligned}
\end{equation}
Use \eqref{ws:density-force} with
\(G=F_{\rm ref}+\nabla r^\beta/r\), and use
\eqref{ws:magnetic-cross-bound} for the magnetic cross term.
The remaining convective, relative-pressure, augmented-potential,
and \(\M(Z)\) terms are bounded by
\(C\norm{\nabla U}_\infty\mathcal F\).
Choose the finitely many Young constants with total cost at
most half of \(\mathcal D\). This proves
\eqref{ws:relative-estimate}, with
\begin{equation}\label{ws:comparison-coefficient}
\begin{aligned}
 G_T(t)=C\bigl[1+\norm{\nabla U}_\infty+\norm{D_U}_\infty^2
       +\norm{F_{\rm ref}}_4^2
       +\norm{\nabla r}_4^2+\norm{\nabla B}_4^2\bigr].
\end{aligned}
\end{equation}
The coefficient outside the brackets depends on fixed weak
density moments and the reference positive interval.
The comparison regularity gives
\[
 \nabla U\in L^2L^\infty,\qquad
 F_{\rm ref}\in L^2H^1\subset L^2L^4,\qquad
 \nabla r,\nabla B\in L^\infty L^4.
\]
The first inclusion follows from \(U\in L^2H^3\), not from
a false embedding \(H^2\subset W^{1,\infty}\). The force
bound follows from \eqref{ws:reference-force} or directly
from the stresses. Hence \(G_T\in L^1(0,T)\), including
at the initial time, and Gronwall gives \eqref{ws:gronwall}.
\end{proof}

\begin{proof}[Proof of Theorem~\ref{thm:weak-strong}]
Common initial data give \(\mathcal F(0)=0\).
Proposition~\ref{ws:relative-theorem} implies
\(\mathcal F(t)=0\) for every \(t\in[0,T]\).
Strict convexity gives \(\rho=r\); the magnetic term gives
\(H=B\); and \(r\ge c_T>0\) then gives \(u=U\) almost
everywhere. Density, momentum, and magnetic field agree in
their canonical time representatives. The argument applies
on every finite interval on which the positive strong
reference exists.
\end{proof}

\begin{remark}\label{ws:scope}
The argument applies to every fixed \(\beta>1\), with no upper
restriction, and does not require \(\gamma\ge\beta\). Its constants
may depend on \(\beta\); no uniformity as \(\beta\to\infty\) is claimed.
This is uniqueness relative to the positive strong class
\eqref{ws:strong-class}, not uniqueness of arbitrary weak
solutions. The weaker assumptions \eqref{eq:data} do not
alone imply that a local strong solution has this comparison
regularity at time zero. In the large-background regime proved
below, the strong solution has the stated regularity on every
finite interval and therefore agrees with every weak solution
in the present class with common data. Weak--strong uniqueness
is not used to construct or continue that strong solution.
\end{remark}

\section{Large-background estimates: magnetic control and low-order feedback}
\label{sec:strong-low}

We return to the original viscosity, without an artificial term. All
estimates in this and the next section are first proved on an existing
positive-density strong-solution interval. Their hypotheses are improved
simultaneously only in the final continuation argument. Throughout,
\[
 \beta>1,\qquad \gamma>\beta+1,\qquad
 A=2\mu+\rho^\beta,\quad b=\rho^{-1},
\]
\[
 q=\nabla\log\rho,\quad d=\Div u,\quad S=Ad,
\]
\[
 \omega=\Curl u,\quad j=\Curl H,\quad K_H=jH^\perp,
 \quad D=A(m)/m,\quad \sigma=\beta\rho^\beta/A.
\]
The acoustic relaxation rate is denoted by \(\mathfrak d\):
\begin{equation}\label{sl:relaxation-rate}
 c_s^2=\gamma m^{\gamma-1},\qquad \Gamma=\frac{c_s^2}{D},
 \qquad
 \mathfrak d=
 \begin{cases}
 D,&1<\beta\le2,\\
 \min\{D,\Gamma\},&\beta>2.
 \end{cases}
\end{equation}
For sufficiently large \(m\),
\(\mathfrak d\le D\), \(\mathfrak d\le\Gamma\), and
\(\mathfrak d^2\le c_s^2\). Indeed, when \(\beta\le2\),
\(\gamma>\beta+1\) implies \(\gamma>2\beta-1\), so
\(\Gamma/D\to\infty\). When \(\beta>2\), set
\begin{equation}\label{sl:slow-exponent}
 a=\gamma-\beta>1,\qquad
 s=\min\{\beta-1,a\}>1,\qquad \mathfrak d\asymp m^s.
\end{equation}
In particular, \(\mathfrak d\to\infty\) throughout the parameter
range, but for large \(\beta\) it need not be comparable to \(D\).
We use \(v^\perp=(-v_2,v_1)\),
\(\nabla^\perp=(-\partial_2,\partial_1)\), and
\(v\times w=v_1w_2-v_2w_1\). In particular
\(\Div\M(H)=K_H\). The letter \(m\) in the strong-solution argument
denotes the conserved mean density, not the momentum. Constants \(C_R\)
may depend on the fixed parameters and the radius in
\eqref{eq:scaled-data}, but not on \(m\), the final time, or bootstrap
heights. Constants \(C_*\) may additionally depend on previously fixed
bootstrap heights and finite exponents, but not on \(m\) or the final time.

\subsection{Scaled energy and the magnetic potential}

\begin{lemma}\label{sl:basic}
For initial data satisfying \eqref{eq:scaled-data}, choose \(m\) large
enough that \(\rho_0/m\in[3/4,5/4]\). On every existing interval
\([0,T]\) on which \(m/2\le\rho\le3m/2\), one has
\begin{equation}\label{sl:basic-budgets}
 \sup_{t\le T}\|u(t)\|_2^2\le C_R,\qquad
 \int_0^T\bigl(\|\nabla u\|_2^2+\|j\|_2^2\bigr)\dt\le C_Rm,
 \qquad \int_0^T\|d\|_2^2\dt\le C_Rm^{1-\beta}.
\end{equation}
These estimates have no upper restriction on \(\gamma\).
\end{lemma}
\begin{proof}
The relative pressure potential at the constant state \(m\) is
\[
 \mathcal P(\rho\mid m)=
 \frac{\rho^\gamma-m^\gamma-\gamma m^{\gamma-1}(\rho-m)}{\gamma-1}.
\]
On the indicated interval it is comparable to
\(m^{\gamma-2}|\rho-m|^2\). Thus \eqref{eq:scaled-data} gives
\[
 E_m(0):=\int\left(\tfrac12\rho_0|u_0|^2+
       \mathcal P(\rho_0\mid m)+\tfrac12|H_0|^2\right)\dx\le C_Rm.
\]
Mass conservation allows the same linear pressure subtraction in the
exact energy identity. Its dissipation is
\(\mu\|\omega\|_2^2+\int A d^2+\nu\|j\|_2^2\).
Since \(\rho\ge m/2\) and \(A\ge c m^\beta\), the stated estimates follow.
The scaled assumption also gives
\(\|\rho_0/m-1\|_{H^2}\le Rm^{(1-\gamma)/2}\), which justifies the
initial density interval for every fixed \(\gamma>1\).
\end{proof}

\begin{lemma}\label{sl:magnetic}
In addition to the density interval, suppose
\(\sup_{t\le T}\|u(t)\|_{H^1}^2\le V\), where \(V\ge1\) is fixed.
Then, with no restriction on the conserved mean \(\overline H\),
\begin{equation}\label{sl:magnetic-bounds}
\begin{gathered}
 \sup\|j\|_2^2\le C_R(1+V^2),\qquad
 \int_0^T\|H-\overline H\|_{H^2}^2\dt\le C_R(1+V^2)m,\\
 \sup\|H\|_\infty^2\le C_R(1+V^3),\qquad
 \sup\|H\|_{H^{2-\varepsilon}}\le C_{R,V,\varepsilon}
                  \quad(0<\varepsilon<1).
\end{gathered}
\end{equation}
For every fixed finite \(p\ge2\),
\begin{equation}\label{sl:force-budgets}
 \sup\|K_H\|_p\le C_{R,V,p},\qquad
 \int_0^T\left(\|K_H\|_2^2+\|K_H\|_4^2+
             \|K_H\|_4^4+\|\nabla K_H\|_2^2\right)\dt
 \le C_{R,V}m.
\end{equation}
\end{lemma}
\begin{proof}
Write \(H=\overline H+\nabla^\perp\phi\), with \(\avg\phi=0\).
The scalar potential satisfies
\[
 \phi_t+u\cdot\nabla\phi=\nu\Delta\phi-u\times\overline H+c(t),
 \qquad c(t)=\avg{u\times\overline H}-\avg{d\phi}.
\]
Testing with \(\phi\), and using its zero mean, gives
\[
 \tfrac12(\|\phi\|_2^2)'+\nu\|\nabla\phi\|_2^2
 =\tfrac12\int d\phi^2-\int\phi(u-\overline u)\times\overline H.
\]
The zero-mean two-dimensional interpolation inequality bounds the first
term by \(\nu\|\nabla\phi\|_2^2/2+C\|d\|_2^2\|\phi\|_2^2\).
The second is bounded using \(\sup\|u\|_2\le C_R\).
Poincare's inequality and variation of constants therefore imply
\begin{equation}\label{sl:potential-bound}
 \sup\|\phi\|_2^2\le
 C_R\exp\left(C\int_0^T\|d\|_2^2\dt\right)\le C_R.
\end{equation}
In particular this first magnetic bound does not depend on \(V\).

For clarity, the exact current equation can be written
\[
 D_tj=H\cdot\nabla\omega-dj-\nabla d\times H+\mathcal T+\nu\Delta j,
\]
where
\[
 \mathcal T=\sum_{k=1}^2
 (\partial_1H_k\partial_ku_2-\partial_2H_k\partial_ku_1
 -\partial_1u_k\partial_kH_2+\partial_2u_k\partial_kH_1),
 \qquad |\mathcal T|\le C|\nabla u||\nabla H|.
\]
Testing with \(j\), then integrating the \(\nabla d\) term by parts,
gives
\begin{equation}\label{sl:current-identity}
 \tfrac12(\|j\|_2^2)'+\nu\|\nabla j\|_2^2
 =-\int\omega H\cdot\nabla j+\int j\mathcal T
        +\tfrac12\int dj^2+\int d(\nabla j\times H).
\end{equation}
Let \(U=\|\nabla u\|_2\), \(J=\|j\|_2\), and
\(J_1=\|\nabla j\|_2\). If \(K\ge1\) bounds \(\|\phi\|_2\),
Fourier interpolation gives
\[
 J\le CK^{1/3}J_1^{2/3},\qquad
 \|H-\overline H\|_\infty\le CK^{1/3}J_1^{2/3}.
\]
For the second estimate, split \(\nabla\phi\) at frequency \(N\):
its low and high parts are bounded by \(CN^2\|\phi\|_2\) and
\(CN^{-1}J_1\), respectively, and optimize. The right side of
\eqref{sl:current-identity} is at most
\(CUJJ_1+CU\|H\|_\infty J_1\). Young's inequality yields
\begin{equation}\label{sl:current-ode}
 (J^2)'+\nu J_1^2\le C_R(U^6+U^2).
\end{equation}
Because \(J_1^2\ge c_RJ^3\), scalar comparison gives
\(\sup J^2\le C_R(1+V^2)\). Integration, using
\(\int U^6\le V^2\int U^2\le C_RV^2m\), proves the first line of
\eqref{sl:magnetic-bounds}.

The endpoint \(H^1\hookrightarrow L^\infty\) is not used. Instead write
\[
 H-\overline H=e^{\nu t\Delta}(H_0-\overline H)
  +\int_0^t e^{\nu(t-s)\Delta}\Div(u\otimes H-H\otimes u)(s)\dd s.
\]
For fixed \(0<\varepsilon<1\), Sobolev multiplication gives
\(\|fg\|_{H^{1-\varepsilon/2}}\le
C_\varepsilon\|f\|_{H^1}\|g\|_{H^1}\). The mean-zero heat estimate is
\[
 \|e^{\nu t\Delta}\Div F\|_{H^{2-\varepsilon}}
 \le C_{\varepsilon,\nu}(1+t^{-1+\varepsilon/4})e^{-c\nu t}
                         \|F\|_{H^{1-\varepsilon/2}}.
\]
This time kernel is integrable on \((0,\infty)\). It proves the
nonendpoint bound in \eqref{sl:magnetic-bounds}; choosing
\(\varepsilon=1/2\) gives explicitly
\(\sup\|H\|_\infty^2\le C_R(1+V^3)\).
Choosing \(\varepsilon<2/p\) also bounds \(j\) and hence \(K_H\) in
every fixed finite \(L^p\).

The time budgets follow without integrating a nonzero spatial mean:
\[
 \|K_H\|_2^2\le\|H\|_\infty^2J^2,\quad
 \|j\|_4^4\le CJ^2J_1^2,\quad
 \|\nabla K_H\|_2^2\le C\|H\|_\infty^2J_1^2+C\|j\|_4^4.
\]
These and \eqref{sl:basic-budgets} give the quadratic, derivative and
fourth-power budgets. Finally
\(\|K_H\|_4^2\le C\|K_H\|_2(\|K_H\|_2+\|\nabla K_H\|_2)\)
and Cauchy--Schwarz give the remaining term.
\end{proof}

\subsection{The corrected low acoustic energy}

\begin{lemma}\label{sl:low-acoustic}
Assume the density interval, \(\sup\|u\|_{H^1}^2\le V\), and the
conditional bounds
\begin{equation}\label{sl:low-inputs}
 \sup m^\gamma\|q\|_{H^1}^2\le Mm,\qquad
 \int_0^T\|\nabla\omega\|_2^2\dt\le Lm.
\end{equation}
For fixed \(V,M,L\) and sufficiently large \(m\), there is a positive
energy \(X\asymp m\|d\|_2^2+m^\gamma\|q\|_2^2\) such that
\begin{equation}\label{sl:low-differential}
 X'+cm^{-\beta}\|\nabla S\|_2^2+c\mathfrak d X
 \le C_*\frac m{\mathfrak d}\|\nabla u\|_2^2
   +Cm^{-\beta}\|K_H\|_2^2
   +CMm^{1-\gamma-\beta}\|\nabla\omega\|_2^2.
\end{equation}
Consequently
\begin{equation}\label{sl:low-profile}
\begin{aligned}
 X(t)&\le C_Rm e^{-c\mathfrak d t}
                 +C_*\left(\frac m{\mathfrak d^2}+m^{1-\beta}\right),\\
 \sup\|d\|_2^2&\le C_R+C_*(\mathfrak d^{-2}+m^{-\beta}).
\end{aligned}
\end{equation}
and, for every fixed \(\zeta>0\),
\begin{equation}\label{sl:flux-budgets}
\begin{aligned}
 \int_0^T\|\nabla S\|_2^2\dt
 &\le C_*\left(m^{\beta+1}+\frac{m^{\beta+2}}{\mathfrak d}\right),\\
 \int_0^T e^{-\zeta\mathfrak d t}\|\nabla S\|_2^2\dt
 &\le C_{*,\zeta}m^{1+\beta}.
\end{aligned}
\end{equation}
Thus the unweighted budget is \(C_*m^3\) if \(1<\beta\le2\),
and \(C_*m^{\beta+1}\) if \(\beta>2\).
\end{lemma}
\begin{proof}
Set \(c_0=D^{-1}\), \(Q=c_0A\gamma\rho^{\gamma-1}\), and
\[
 \mathcal A_0=\tfrac12\int(c_0Ad^2+Q|q|^2),\quad
 \mathscr D_0=c_0\int b|\nabla S|^2,\quad
 \Lambda(\rho)=2\mu\log\rho+\rho^\beta/\beta.
\]
Direct testing of the momentum and differentiated mass equations gives
\begin{equation}\label{sl:acoustic-identity}
\begin{aligned}
 \mathcal A_0'+\mathscr D_0={}&\mathcal R_q
 -c_0\mu\int b\nabla S\cdot\nabla^\perp\omega
 -c_0\int b\nabla S\cdot K_H+2c_0\int S\det\nabla u\\
 &-\tfrac{c_0}{2}\int(A+\beta\rho^\beta)d^3,\qquad
 \mathcal R_q=-\int Qq\cdot(\nabla u)^{\mathsf T}q
   +\tfrac12\int Q(\sigma+2-\gamma)d|q|^2.
\end{aligned}
\end{equation}
Here \(D_tq=-\nabla d-(\nabla u)^{\mathsf T}q\) and
\(D_t\Lambda=-S\); these also give the exact corrector identity
\begin{equation}\label{sl:corrector-identity}
\begin{aligned}
 c_0\left(\int y d\right)'={}&-c_0\int Ad^2
  -c_0\int(A/\rho)q\cdot\nabla S+\int Q|q|^2\\
 &-c_0\int(A/\rho)q\cdot K_H+2c_0\int y\det\nabla u,
 \qquad y=\Lambda-\avg\Lambda.
\end{aligned}
\end{equation}
Indeed \(D_td=\Div D_tu-\operatorname{tr}((\nabla u)^2)\), and
\(d^2-\operatorname{tr}((\nabla u)^2)=2\det\nabla u\).
The shear pairing in \eqref{sl:corrector-identity} is zero because
\((A/\rho)q\) is a scalar gradient. Also
\(\int\det\nabla u=0\), so subtracting \(\avg\Lambda\) introduces no term.

Choose a fixed sufficiently small \(\varepsilon_0>0\), and put
\[
 \eta=\varepsilon_0\mathfrak d,\qquad
 X=\mathcal A_0-\eta c_0\int yd.
\]
Since \(\mathfrak d^2\le c_s^2\), Poincare gives
\[
 \left|\eta c_0\int yd\right|
 \le C\eta m\|q\|_2\|d\|_2
 \le C\varepsilon_0
             \bigl(m\|d\|_2^2+m^\gamma\|q\|_2^2\bigr).
\]
Hence \(X\asymp\mathcal A_0\), with equivalence constants independent
of the bootstrap heights. The weighted mean constraint \(\int S/A=0\) implies
\(\|d\|_2\le Cm^{-\beta}\|\nabla S\|_2\).
Inserting minus \(\eta\) times \eqref{sl:corrector-identity} into
\eqref{sl:acoustic-identity}, the two quadratic corrector errors are
absorbed using
\[
 \eta c_0\int Ad^2\le C\varepsilon_0\mathscr D_0,\qquad
 \eta c_0\left|\int(A/\rho)q\cdot\nabla S\right|
 \le\epsilon\mathscr D_0+C_\epsilon\eta^2m^\beta\|q\|_2^2.
\]
The first absorption uses \(\eta/D\le\varepsilon_0\), and the
second uses \(\eta/\Gamma\le\varepsilon_0\):
\(C_\epsilon\eta^2m^\beta\|q\|_2^2\) is a small multiple of
\(\eta\int Q|q|^2\).
The remaining flux dissipation controls
\(\mathfrak d\,m\|d\|_2^2\), because
\(\mathscr D_0\ge cD\,m\|d\|_2^2\) and \(\mathfrak d\le D\).

For completeness the other costs, with \(U=\|\nabla u\|_2\) and
\(\delta_\rho=\|\rho/m-1\|_\infty\), are
\begin{align}
 \left|2c_0\int S\det\nabla u\right|
 &\le\epsilon\mathscr D_0+C_\epsilon m^{2-\beta}U^4,
                                                        \label{sl:low-wente}\\
 \left|2\eta c_0\int y\det\nabla u\right|
 &\le\epsilon\eta m^\gamma\|q\|_2^2
                       +C_\epsilon\eta m^{2-\gamma}U^4,\notag\\
 c_0\mu\left|\int b\nabla S\cdot\nabla^\perp\omega\right|
 &\le\epsilon\mathscr D_0+
             C_\epsilon\delta_\rho^2m^{-\beta}\|\nabla\omega\|_2^2,
                                                        \notag\\
 \tfrac{c_0}{2}\left|\int(A+\beta\rho^\beta)d^3\right|
 &\le\epsilon\mathscr D_0+C_\epsilon m^{2-\beta}\|d\|_2^4.
                                                        \notag
\end{align}
For the last inequality, $\int d=\int S/A=0$ and
$\nabla d=A^{-1}\nabla S-\sigma dq$ give, by Poincare and
$H^1\hookrightarrow L^4$,
\[
 \|d\|_{H^1}\le C m^{-\beta}\|\nabla S\|_2
                  +C\|q\|_4\|d\|_{H^1}
             \le C m^{-\beta}\|\nabla S\|_2,
\]
where the small $\|q\|_4$ term is absorbed using
\eqref{sl:low-inputs}. The first two inequalities are periodic Wente
duality. Their remainders
are both at most \(C_\epsilon(m/\mathfrak d)U^4\), since
\[
 m^{2-\beta}\le C m/\mathfrak d,\qquad
 \eta m^{2-\gamma}\le C m/\mathfrak d.
\]
The latter follows from \(\mathfrak d^2\le c_s^2\).
The shear term
uses \(\int\nabla S\cdot\nabla^\perp\omega=0\), subtracting the
constant \(m^{-1}\) from \(b\). Both magnetic pairings cost
\(\epsilon\mathscr D_0+\epsilon\eta m^\gamma\|q\|_2^2+
C_\epsilon m^{-\beta}\|K_H\|_2^2\).
For the corrector magnetic pairing this uses
\(\eta m^{-\gamma}\le C m^{-\beta}\), which follows from
\(\eta\le\varepsilon_0\Gamma\).
The cubic divergence remainder is absorbable: its ratio to
\(\mathscr D_0\) is at most \(C V m^{2-2\beta}=o(1)\).

If \(\ell=\log\rho-\avg{\log\rho}\), then
\(\rho/m=e^\ell/\avg{e^\ell}\). Sobolev embedding and
\eqref{sl:low-inputs} show
\(\delta_\rho^2\le C\|q\|_{H^1}^2\le CMm^{1-\gamma}\).
Finally the unestimated density strain satisfies
\[
 |\mathcal R_q|\le Cm^\gamma U\|q\|_4^2
 \le CU X^{1/2}(Mm)^{1/2}
 \le\epsilon\mathfrak d X+C_\epsilon M\frac m{\mathfrak d}U^2.
\]
Together with \(U^4\le VU^2\), these estimates prove
\eqref{sl:low-differential}.

The scaled data give \(X(0)\le C_Rm\). Variation of constants, using
the velocity supremum for the first source and the magnetic and
vorticity integrals for the others, yields
\[
 X(t)\le C_Rm e^{-c\mathfrak d t}
 +C_*\left(\frac m{\mathfrak d^2}
                   +m^{1-\beta}+m^{2-\gamma-\beta}\right).
\]
The last term is smaller than \(m^{1-\beta}\), since \(\gamma>1\).
This proves \eqref{sl:low-profile}; it is not necessary to discard
the magnetic tail \(m^{1-\beta}\).
Integrating \eqref{sl:low-differential} and using
\(\int U^2\le C_Rm\) gives
\[
 \int_0^T\|\nabla S\|_2^2\dt
 \le C_*\left(m^{\beta+1}+\frac{m^{\beta+2}}{\mathfrak d}
                                  +m+m^{2-\gamma}\right).
\]
The last two terms are bounded by \(m^{\beta+1}\), proving the
first estimate in \eqref{sl:flux-budgets}.
When \(\beta>2\), the inequality \(s>1\) makes
\(m^{\beta+2}/\mathfrak d\le C m^{\beta+1}\); for
\(\beta\le2\), this term is comparable to \(m^3\).
To obtain the exponentially weighted estimate, multiply by
\(e^{-\zeta\mathfrak d t}\). Integration by parts gives
\[
 \int_0^T e^{-\zeta\mathfrak d t}X'\dt
 =e^{-\zeta\mathfrak d T}X(T)-X(0)
       +\zeta\mathfrak d\int_0^T e^{-\zeta\mathfrak d t}X\dt.
\]
The velocity source now costs only \(C_*m/\mathfrak d^2\), and
the other source terms retain their preceding global budgets.
Since \(\mathfrak d\to\infty\), their sum and \(X(0)\) are at most
\(C_{*,\zeta}m\). Multiplication by \(m^\beta\) proves the second
estimate in \eqref{sl:flux-budgets}.
\end{proof}

\subsection{The mixed vorticity--current energy and its principal constant}

\begin{lemma}\label{sl:mixed}
Assume the density interval, \(\sup\|u\|_{H^1}^2\le V\), and
\(\sup\|q\|_{H^1}^2\le Mm^{1-\gamma}\). Define
\[
 g=m^2/\rho^2,\qquad
 \mathcal E=\tfrac12\int\left((m^2/\rho)\omega^2+j^2\right),\qquad
 \mathcal B_S(T)=m^{-1}\int_0^T\delta_\rho(t)^2\|\nabla S\|_2^2\dt.
\]
For sufficiently large \(m\), depending on the fixed heights,
\begin{equation}\label{sl:mixed-bound}
 \sup_{t\le T}\frac{\mathcal E(t)}m+
 \frac1m\int_0^T(\mu\|\nabla\omega\|_2^2+\nu\|\nabla j\|_2^2)\dt
 \le C_R\left[1+\mathcal B_S(T)+(1+V^3)m^{-\beta}\right].
\end{equation}
In particular the principal constant on the right does not depend on
\(V\) or \(M\).
\end{lemma}
\begin{proof}
The normalized potential vorticity \(z=m\omega/\rho\) obeys the
mass-weighted cancellation of compression. Combining its quadratic
energy with \eqref{sl:current-identity} gives the exact identity
\begin{equation}\label{sl:mixed-identity}
\begin{aligned}
 \mathcal E'+\mu\int g|\nabla\omega|^2+\nu\|\nabla j\|_2^2
 ={}&-\tfrac\mu2\int\omega\nabla g\cdot\nabla\omega
       -\tfrac12\int(g-1)J(\omega,S)\\
 &+\tfrac12\int(g-1)
           [\omega H\cdot\nabla j-jH\cdot\nabla\omega]\\
 &+\int j\mathcal T+\tfrac12\int dj^2+
                                     \int d(\nabla j\times H),
\end{aligned}
\end{equation}
where \(J(f,g)=\nabla f\times\nabla g\).
To verify the two less immediate cancellations, the vorticity equation is
\[
 D_t\omega+d\omega=\mu\Div(b\nabla\omega)
                  -bq\times\nabla S+\Curl(bK_H).
\]
Since \(\nabla g=-2gq\), its scalar coupling is
\(-\tfrac12\int(g-1)J(\omega,S)\). Pressure contributes zero because
\(g\) and \(\rho^\gamma\) are functions of the same density.
Its magnetic contribution is
\(\tfrac12\int g[\omega H\cdot\nabla j-jH\cdot\nabla\omega]\);
the part with \(g=1\) cancels the first term of
\eqref{sl:current-identity}. Thus \eqref{sl:mixed-identity} does not
discard a principal Lorentz coupling.

Put \(X_\omega=\|\nabla\omega\|_2\), \(J=\|j\|_2\), and
\(J_1=\|\nabla j\|_2\). Since \(\omega,j\) have zero mean, the first
and the residual magnetic terms in \eqref{sl:mixed-identity} are bounded by
\[
 C\mu\|q\|_4X_\omega^2,\qquad
 C\delta_\rho\|H\|_\infty X_\omega J_1,
\]
and are absorbed for large \(m\). The baroclinic term costs
\(\epsilon\mu X_\omega^2+C_R\delta_\rho^2\|\nabla S\|_2^2\).
The stretching term is estimated without inserting the velocity
bootstrap into its rotational part:
\begin{equation}\label{sl:stretching-principal}
\begin{aligned}
 \left|\int j\mathcal T\right|
 &\le C\|\nabla u\|_2\|j\|_4\|\nabla H\|_4
 \le\epsilon\nu J_1^2+C_RJ^2(\|\omega\|_2^2+\|d\|_2^2)\\
 &\le\epsilon\nu J_1^2+\frac{C_R}{m}J^2\mathcal E
                          +C_R(1+V^2)\|d\|_2^2.
\end{aligned}
\end{equation}
The last two compression terms of \eqref{sl:mixed-identity} cost
\(\epsilon\nu J_1^2+C_R(1+V^3)\|d\|_2^2\).
Consequently
\[
 \mathcal E'+c(\mu X_\omega^2+\nu J_1^2)
 \le\frac{C_R}{m}J^2\mathcal E
       +C_R\delta_\rho^2\|\nabla S\|_2^2
       +C_R(1+V^3)\|d\|_2^2.
\]
The Gronwall exponent is bounded by the basic energy alone:
\(m^{-1}\int J^2\le C_R\), independently of \(V\).
Also \(\mathcal E(0)\le C_Rm\) and
\(m^{-1}\int\|d\|_2^2\le C_Rm^{-\beta}\).
Gronwall followed by integration proves \eqref{sl:mixed-bound}.
\end{proof}

\subsection{Fixed finite-exponent potential vorticity}

\begin{lemma}\label{sl:finite-p}
Fix \(r>p>2\), and suppose \(\rho/m\in[1/2,3/2]\) and
\(\sup\|q\|_r\le\epsilon_{p,r}\) for a sufficiently small fixed
constant. Put \(Q_r(t)=\|q(t)\|_r\),
\(B_p=\sup_{t\le T}\|K_H(t)\|_p\), and \(G(t)=\|\nabla S(t)\|_2\).
Then
\begin{equation}\label{sl:finite-p-bound}
 \|\omega(t)\|_p\le C_{p,r}\left[\|\omega_0\|_p+B_p+
 \left(\frac1m\int_0^t e^{-c_{p,r}(t-s)/m}Q_r(s)^2G(s)^2\dd s\right)^{1/2}
 \right].
\end{equation}
No supremum of \(\nabla\omega\) is an input.
\end{lemma}
\begin{proof}
Set \(\varrho=\rho/m\) and \(z=\omega/\varrho\). The exact equation is
\[
 \varrho D_tz=\frac\mu m\Div(\nabla z+qz)
       -\frac1{m\varrho}q\times\nabla S
       +\frac1m\Curl(K_H/\varrho).
\]
Let \(E_p=p^{-1}\int\varrho|z|^p\),
\(Z_p=\operatorname{sgn}(z)|z|^{p/2}\), and
\(D_p=\|\nabla Z_p\|_2\). The mass equation cancels all compression
reaction terms in the test, giving
\begin{equation}\label{sl:finite-p-test}
\begin{aligned}
 E_p'+\frac{4\mu(p-1)}{mp^2}D_p^2
 ={}&-\frac{2\mu(p-1)}{mp}\int q\cdot Z_p\nabla Z_p
 -\frac1m\int\varrho^{-1}(q\times\nabla S)|z|^{p-2}z\\
 &-\frac{p-1}{m}\int\varrho^{-1}K_H\cdot\nabla^\perp z\,|z|^{p-2}.
\end{aligned}
\end{equation}
The needed Poincare inequality is not based on a false zero mean for
\(z\). Rather \(\int\varrho z=\int\omega=0\). The map
\(\Psi(s)=\operatorname{sgn}(s)|s|^{2/p}\) is \(2/p\)-Holder, so
\[
 |\Psi(\avg{Z_p})|\int\varrho
 \le C_p\int\varrho|Z_p-\avg{Z_p}|^{2/p}.
\]
Ordinary Poincare then gives
\(\|Z_p\|_2\le C_pD_p\), and for every fixed finite \(a\ge2\),
\(\|Z_p\|_a\le C_{p,a}\|Z_p\|_2^{2/a}D_p^{1-2/a}\).
The first term of \eqref{sl:finite-p-test} is at most
\(C_{p,r}Q_rD_p^2/m\), and is absorbed. The magnetic term is bounded by
\(\epsilon D_p^2/m+C_{p,\epsilon}B_p^2E_p^{1-2/p}/m\), and hence
by another small part of the Poincare damping plus \(C_pB_p^p/m\).

For the scalar forcing define
\[
 t_*=1-2/p+2/r\in(0,1),\quad a_*=1/2-1/r,\quad
 k_*=\frac{2}{2-t_*}\in(1,2).
\]
Holder and the preceding interpolation give
\[
 |R_S|\le\frac{C_{p,r}}m Q_rG E_p^{a_*}D_p^{t_*}
 \le\frac\epsilon mD_p^2+\frac{C_{p,r}}m(Q_rG)^{k_*}
                                            E_p^{1-k_*/p},
\]
where \(a_*k_*=1-k_*/p\). Thus
\[
 E_p'+\frac cmE_p\le\frac Cm(Q_rG)^{k_*}E_p^{1-k_*/p}
                          +\frac CmB_p^p.
\]
Apply the chain rule to \((E_p+B_p^p)^{k_*/p}\); zero values can be
handled by a positive regularization. The resulting linear inequality
has source \(C[(Q_rG)^{k_*}+B_p^{k_*}]/m\).
Variation of constants and Holder's inequality with \(k_*<2\) give
\[
 \frac1m\int_0^t e^{-c(t-s)/m}(Q_rG)^{k_*}(s)\dd s
 \le C\left(\frac1m\int_0^t
              e^{-c(t-s)/m}Q_r(s)^2G(s)^2\dd s\right)^{k_*/2}.
\]
Since \(E_p^{1/p}\asymp\|\omega\|_p\), this proves
\eqref{sl:finite-p-bound}. The actual shear time scale \(m\) has been
retained throughout.
\end{proof}

\subsection{The interface with staged high-acoustic bounds}

The high-order argument will decrease the acoustic energy in finitely
many stages when \(1<\beta\le2\). The low-order feedback only needs
its coarse bound during a short initial layer and its improved bound
afterwards. We isolate this interface without assuming the high-order
estimates in advance. For \(\beta>2\), the single dynamic profile is
retained.

\begin{lemma}\label{sl:dynamic-feedback}
Assume the hypotheses of Lemma~\ref{sl:low-acoustic}. For
\(1<\beta\le2\), suppose that a fixed exponent \(a_*<\gamma-2\)
and a time \(0\le t_L\le C_*\log m/D\) satisfy
\begin{equation}\label{sl:staged-profile-input}
 \|q(t)\|_{H^1}^2\le C_*m^{a_*-\gamma}
 \qquad(t_L\le t\le T).
\end{equation}
When \(T<t_L\), this tail condition is empty; the coarse bound in
\eqref{sl:low-inputs} is used throughout the existing interval.
For \(\beta>2\), suppose instead that
\begin{equation}\label{sl:profile-input}
 \|q(t)\|_{H^1}^2\le C_Rm^{1-\gamma}e^{-c\mathfrak d t}
                                  +C_*m^{2-s-\gamma},
\end{equation}
where \(s>1\) is defined in \eqref{sl:slow-exponent}. This second
input corresponds to the high-order choice \(h=2s\), and does not
require \(h>\beta\).
For every fixed \(r>p>2\),
\begin{equation}\label{sl:feedback-budget}
\begin{aligned}
 &\frac1m\int_0^T
 (\delta_\rho(t)^2+\|q(t)\|_r^2)\|\nabla S\|_2^2\dt\\
 &\qquad\le C_{r,*}\left[
 m^{\beta+1-\gamma}+
 \begin{cases}
 m^{a_*+2-\gamma},&1<\beta\le2,\\
 m^{\beta+2-s-\gamma},&\beta>2.
 \end{cases}\right]
\end{aligned}
\end{equation}
The right side tends to zero under \(\gamma>\beta+1\) and the
stated restriction on \(a_*\) in the first branch, and simply under
\(\gamma>\beta+1\) in the second. Consequently,
\begin{equation}\label{sl:low-improvements}
 I(T)\le C_R+o_*(1),\quad
 \sup\|u\|_{H^1}^2\le C_R+o_*(1),\quad
 \sup\|\omega\|_p\le C_{R,V,p,r}+o_*(1),
\end{equation}
where \(I(T)\) is the left side of \eqref{sl:mixed-bound}. The first
two principal constants do not depend on any bootstrap height; the
last depends on \(V\), but not on the high-acoustic, high-vorticity,
or finite-\(p\) bootstrap heights.
\end{lemma}
\begin{proof}
The mass normalization used above and the finite Sobolev embeddings give
\(\delta_\rho^2+\|q\|_r^2\le C_r\|q\|_{H^1}^2\).
Write \(G=\|\nabla S\|_2\). First let \(1<\beta\le2\) and set
\(T_0=\min\{T,t_L\}\). Integrate \eqref{sl:low-differential} only
over \([0,T_0]\), using \(X(0)\le C_Rm\), \(X(T_0)\ge0\),
\(\sup\|\nabla u\|_2^2\le V\), the magnetic budget
\(\int\|K_H\|_2^2\dt\le C_{R,V}m\), and the coarse mixed bound
\(\int\|\nabla\omega\|_2^2\dt\le Lm\). Since
\(\mathfrak d=D\asymp m^{\beta-1}\), this gives
\begin{equation}\label{sl:initial-layer-flux}
\begin{aligned}
 \int_0^{T_0}G^2\dt
 &\le C_*\left(m^{\beta+1}+m^2t_L+m+m^{2-\gamma}\right)\\
 &\le C_*\left(m^{\beta+1}+m^{3-\beta}\log m
                           +m+m^{2-\gamma}\right)
 \le C_*m^{\beta+1}.
\end{aligned}
\end{equation}
In the last inequality,
\(m^{3-\beta}\log m=o(m^{\beta+1})\) because \(\beta>1\),
while \(\gamma>\beta+1\). This prefix estimate does not presume
existence up to \(t_L\).

Use \(\|q\|_{H^1}^2\le Mm^{1-\gamma}\) on the prefix. On
\([t_L,T]\), when nonempty, use \eqref{sl:staged-profile-input}
and the full-time budget \(\int_0^TG^2\dt\le C_*m^3\) from
\eqref{sl:flux-budgets}. Thus
\[
 \frac1m\int_0^T(\delta_\rho^2+\|q\|_r^2)G^2\dt
 \le C_{r,*}\left(m^{\beta+1-\gamma}
                             +m^{a_*+2-\gamma}\right).
\]
Both exponents are strictly negative. Importantly, the full-time
\(m^3\) budget has only been paired with the improved tail bound,
not with the coarse initial-layer bound.

For \(\beta>2\), multiply \eqref{sl:profile-input} by \(G^2/m\)
and apply separately the weighted and unweighted estimates in
\eqref{sl:flux-budgets}. Their contributions are respectively
\(C_{r,*}m^{\beta+1-\gamma}\) and
\(C_{r,*}m^{\beta+2-s-\gamma}\). Writing
\(a=\gamma-\beta>1\) makes the
two exponents \(1-a<0\) and \(2-s-a<0\), since \(s>1\).
This proves \eqref{sl:feedback-budget} in both branches.

Now apply Lemma~\ref{sl:mixed}, \eqref{sl:low-profile}, the Hodge identity
\(\|\nabla u\|_2^2=\|d\|_2^2+\|\omega\|_2^2\), and the basic kinetic
bound. Lemmas~\ref{sl:magnetic} and~\ref{sl:finite-p} give the last
assertion, replacing the history integral by the larger full integral
in \eqref{sl:feedback-budget}. The smallness assumption of
Lemma~\ref{sl:finite-p} follows already from the coarse bound
\(\|q\|_{H^1}^2\le Mm^{1-\gamma}\).
All dependence on the high bootstrap heights is multiplied by a
strictly negative power of \(m\). Thus these implications do not
introduce a constant-level circular dependence.
\end{proof}

All tests above are valid first for smooth solutions. In the positive
\(H^2\) local class, spatial regularization and time difference quotients
justify the quadratic identities on compact existence intervals.
For the finite-\(p\) test, the local class at fixed \(m\) gives
\(z\in L^\infty H^1\cap L^2H^2\) and \(z_t\in L^2L^2\).
Indeed \(z=m\rho^{-1}\Curl u\), multiplication by the positive
\(H^2\) density reciprocal preserves \(H^1,H^2\), and the time
derivative follows from \(u_t\in L^2H^1\) and \(\rho_t\in C H^1\).
Finite-exponent Sobolev embeddings then give
\(|z|^{p-2}z\in L^\infty L^2\cap L^2H^1\).
For example, its gradient is controlled by
\(\|z\|_{4(p-2)}^{p-2}\|\nabla z\|_4\) when \(p\ge5/2\);
for \(2<p<5/2\) replace the first norm by
\(1+\|z\|_2^{p-2}\) on the finite-volume torus.
Smooth truncation and time regularization therefore justify the
chain rule and the integrated finite-\(p\) identity.
These local norms are used only for admissibility of the tests;
none enters the constants in the resulting estimates.
The scalar corrector is a smooth density composition on the individual
positive interval. No approximation-dependent bound from the weak
construction is used to obtain these strong-solution estimates.

\section{The high acoustic energy}\label{sec:strong-high}

This section concerns positive strong solutions and does not use a
pointwise density bound for the weak solutions constructed above.
All constants in this section may depend on the fixed exponents,
on \(\mu,\nu\), and on the torus. Constants denoted by \(C_*\) may
also depend on the explicitly prescribed bootstrap heights, but
never on \(m\) or on the length of the interval of existence.
We write \(D_t=\partial_t+u\cdot\nabla\),
\(v^\perp=(-v_2,v_1)\), and
\(\nabla^\perp=(-\partial_2,\partial_1)\).

\subsection{Variables and a conditional proposition}
Set
\begin{equation}\label{sg:variables}
\begin{gathered}
 d=\Div u,\quad \omega=\Curl u,\quad j=\Curl H,\quad
 A=2\mu+\rho^\beta,\quad b=\rho^{-1},\\
 S=Ad,\quad f=\nabla S,\quad q=\nabla\log\rho,\quad
 w=Aq=\nabla\Lambda(\rho),\quad \Lambda'(r)=\frac{2\mu+r^\beta}{r},\\
 k_\rho=\frac{\gamma\rho^{\gamma-1}}A,\quad
 \sigma=\frac{\beta\rho^\beta}A,\quad
 \kappa=\gamma-1-\sigma,\quad L_{ij}=\partial_j u_i,\quad
 K_H=jH^\perp,\quad R=\mu\nabla^\perp\omega+K_H .
\end{gathered}
\end{equation}
The constants and Fourier multipliers used below are
\begin{equation}\label{sg:frequency-parameters}
\begin{gathered}
 A_m=2\mu+m^\beta,\quad D=A_m/m,\quad c_m=D^{-1},
 \quad c_s^2=\gamma m^{\gamma-1},\\
 K=c_s/D,\qquad \Gamma=c_s^2/D,\qquad
 \alpha(\xi)=\frac{Dc_s^2|\xi|^2}{c_s^2+D^2|\xi|^2},
 \qquad \mathsf T=\alpha(D_x).
\end{gathered}
\end{equation}
Fourier frequencies are those of the fixed torus; changing its
lowest nonzero frequency changes only fixed constants.
Define the positive energy and its flux dissipation by
\begin{equation}\label{sg:positive-energy}
 \mathcal A_1=\frac{c_m}{2}\int_{\TT}
       \bigl(A^{-1}|f|^2+k_\rho|\nabla w|^2\bigr)\dx,
 \qquad
 \mathscr D_S=c_m\int_{\TT}b|\nabla f|^2\dx .
\end{equation}
The enhanced density dissipation and the modified energy are
\begin{equation}\label{sg:modified-energy}
\begin{aligned}
 \mathscr B_\rho
   &=m^\gamma\sum_{\xi\ne0}\alpha(\xi)|\xi|^2
                                  |\widehat q(\xi)|^2,\\
 \mathcal A_{\mathsf T}
   &=\mathcal A_1-\varepsilon_0c_m\ip{q}{\mathsf T f},
 \qquad
 \mathcal V_h=\frac12m^{1-h}\|\nabla\omega\|_2^2,\qquad
 Y=Y_h=\mathcal A_{\mathsf T}+\mathcal V_h .
\end{aligned}
\end{equation}
Here and below the brackets are the real \(L^2\) pairing, summed
over vector components.

\begin{proposition}[Conditional high-order estimate with an improved input]
\label{sg:high-proposition}
Fix
\[
 1<\beta\le2,\qquad \gamma>\beta+1,\qquad a\le1,
\]
and choose \(h\) so that
\begin{equation}\label{sg:h-interval}
 h>\beta,\qquad a+h+\beta+1-2\gamma<0,
 \qquad h<3\beta+2\gamma-2 .
\end{equation}
Choose fixed finite \(P\ge4\) and \(p\ge P\) so that
\begin{equation}\label{sg:exponent-choices}
 \kappa_P:=2+\frac4P<\min\{2\beta,h+2-\beta\},
 \qquad \frac{h+1}{p+1}<\beta-1 .
\end{equation}
Suppose a smooth solution on \([0,T]\) satisfies
\begin{equation}\label{sg:bootstrap-assumptions}
 \avg{\rho}=m,\quad
 \|\rho/m-1\|_\infty\le\delta_0,\quad
 \sup_{t\le T}\|u(t)\|_{H^1}^2\le V,\quad
 \sup_{t\le T}\|\omega(t)\|_p\le W,\quad
 \sup_{t\le T}\mathcal A_1(t)\le Mm ,
\end{equation}
where \(\delta_0>0\) is fixed sufficiently small.
Suppose additionally that, for some \(0\le s\le T\),
\begin{equation}\label{sg:improved-input}
 \mathcal A_1(t)\le C_a m^a\qquad(s\le t\le T),
\end{equation}
with \(C_a\) independent of \(m,s,T\). The coarse bound
\(\mathcal A_1\le Mm\) remains available in all other estimates.
Use the basic energy and magnetic estimates in the following explicit
form:
\begin{equation}\label{sg:lower-inputs}
\begin{gathered}
 \int_0^T\|\nabla u\|_2^2\dt\le C_Rm,\qquad
 \int_0^T\|d\|_2^2\dt\le C_Rm^{1-\beta},\\
 \int_0^T\bigl(\|\nabla K_H\|_2^2+\|K_H\|_4^2\bigr)\dt
       \le C_*(V,R)m .
\end{gathered}
\end{equation}
These are lower-order inputs, not assumptions on high derivatives
of the magnetic forcing. For a fixed sufficiently small
\(\varepsilon_0>0\), and then sufficiently large \(m\),
\(\mathcal A_{\mathsf T}\asymp\mathcal A_1\), and on \([s,T]\),
\begin{equation}\label{sg:joint-sharp}
\begin{aligned}
 Y'+c\bigl(\mathscr D_S+\mathscr B_\rho+m^{-h}Z^2\bigr)
 \le{}& C_*\bigl[
       m^{2-\beta}\|\nabla u\|_4^4
       +m^{\kappa_P-h}\|\nabla u\|_2^{\kappa_P}\\
 &\hspace{17mm}
       +(m^{2-h}+m^{e_P})\|d\|_2^2\bigr]+\mathcal M(t),
\end{aligned}
\end{equation}
where
\[
 Z=\|\nabla^2\omega\|_2,\qquad
 e_P=\frac{P(2-h)-\beta}{P-1},\qquad
 \mathcal M(t)\ge0,\qquad
 \int_0^T\mathcal M(t)\dt\le C_*m^{1-\beta}.
\]
In particular,
\begin{equation}\label{sg:joint-simplified}
 Y'+c\bigl(\mathscr D_S+\mathscr B_\rho+m^{-h}Z^2\bigr)
       \le C_*m^{2-\beta}+\mathcal M(t).
\end{equation}
With \(k(h)=\max\{3-2\beta,\beta+1-h,1-\beta\}<1\),
\begin{equation}\label{sg:dynamic-conclusion}
\begin{aligned}
 Y_h(t)&\le e^{-cD(t-s)}Y_h(s)+C_*m^{k(h)},\\
 \|q(t)\|_{H^1}^2
       &\le Cm^{-\gamma}e^{-cD(t-s)}Y_h(s)
                                      +C_*m^{k(h)-\gamma}.
\end{aligned}
\end{equation}
Here \(C_*\) may also depend on \(a,h,C_a\). All constants and
thresholds are independent of \(s,T\). For any prescribed finite
collection of admissible pairs \((a,h)\), the correction parameter
\(\varepsilon_0\) and the finite spatial exponents can be chosen
in common.
\end{proposition}

We prove the proposition across this section and the next two
sections. The improved input \eqref{sg:improved-input} is used only
in the critical high-frequency differentiated-density strain.
All other estimates retain the coarse power \(m\). When \(a=1\),
the middle condition in \eqref{sg:h-interval} reduces to the old
restriction \(h<2\gamma-\beta-2\). Smaller values of \(a\) allow
larger weights \(h\); this is the mechanism used in the finite-stage
argument of Section~\ref{sec:strong-closure}. All the auxiliary
parameters are selected before \(m\). The corrected acoustic energy
\(\mathcal A_{\mathsf T}\) depends neither on \(h\) nor on the
auxiliary frequency cutoff.
All the exact identities and algebraic estimates below are valid for
arbitrary \(\beta>1\). The parameter comparisons through
Remark~\ref{sg:full-time-remark} concern \(1<\beta\le2\);
Proposition~\ref{sg:superquadratic} then gives a separate closure for
\(\beta>2\), with the actual acoustic damping in place of \(D\).

\subsection{Recovering the small density gradient}
The mean of \(w=\nabla\Lambda(\rho)\) is zero. On the density
interval in \eqref{sg:bootstrap-assumptions},
\[
 \|\nabla w\|_2^2\le Cm^{2\beta-\gamma}\mathcal A_1,\qquad
 \|q\|_s\le C_sm^{-\beta}\|w\|_s
       \le C_sm^{-\gamma/2}\mathcal A_1^{1/2}
       \quad(2\le s<\infty).
\]
Using
\(\nabla q=A^{-1}\nabla w-\sigma q\otimes q\), it follows that
\begin{equation}\label{sg:q-small}
 B_q^2:=\|q\|_{H^1}^2
 \le C\bigl(m^{-\gamma}\mathcal A_1+
                      m^{-2\gamma}\mathcal A_1^2\bigr)
 \le Cm^{-\gamma}\mathcal A_1
 \le C_*m^{1-\gamma}.
\end{equation}
This derivation precedes any equivalence requiring small \(q\).
Moreover, if \(\ell=\log\rho-\avg{\log\rho}\), then
\(\|\ell\|_{H^2}\le C B_q\) and
\(\rho/m=e^\ell/\avg{e^\ell}\). Consequently
\begin{equation}\label{sg:density-amplitude}
 \delta_\rho:=\|\rho/m-1\|_\infty\le C B_q=o(1).
\end{equation}
Since \(\nabla w=A(\nabla q+\sigma q\otimes q)\), these estimates,
or equivalently the small fixed oscillation in
\eqref{sg:bootstrap-assumptions}, give
\begin{equation}\label{sg:energy-equivalence}
 \mathcal A_1\asymp
     m^{1-2\beta}\|f\|_2^2+m^\gamma\|\nabla q\|_2^2,
 \qquad
 \mathscr D_S\asymp m^{-\beta}\|\nabla f\|_2^2 .
\end{equation}

\subsection{The exact differentiated acoustic identity}
The original equations imply the exact, variable-coefficient system
\begin{equation}\label{sg:acoustic-system}
\begin{aligned}
 D_tS={}&A\Div(b\nabla S)-A\Div(k_\rho w)
       +A\Div(bR)-A\operatorname{tr}(L^2)-\sigma Sd,\\
 D_tw={}&-\nabla S-L^Tw .
\end{aligned}
\end{equation}
Write \(f_\ell=\partial_\ell S\) and \(W_\ell=\partial_\ell w\).
Then
\begin{equation}\label{sg:differentiated-system}
 D_tf_\ell=A\Div(b\nabla f_\ell)-A\Div(k_\rho W_\ell)+F_\ell,
 \qquad
 D_tW_\ell=-\nabla f_\ell-L^TW_\ell+G_\ell ,
\end{equation}
where every commutator is retained:
\begin{equation}\label{sg:all-commutators}
\begin{aligned}
 F_\ell={}&A\Div((\partial_\ell b)f)
          -A\Div((\partial_\ell k_\rho)w)\\
 &+(\partial_\ell A)\,[\Div(bf)-\Div(k_\rho w)]
   +\partial_\ell\mathcal G-(\partial_\ell u)\cdot f,\\
 G_\ell={}&-(\partial_\ell L)^Tw-(\partial_\ell u)\cdot\nabla w,\\
 \mathcal G={}&A\,[\Div(bR)-\operatorname{tr}(L^2)-\sigma d^2].
\end{aligned}
\end{equation}
Since
\(D_tA^{-1}=\sigma dA^{-1}\) and
\(D_tk_\rho=-\kappa d k_\rho\), testing gives
\begin{equation}\label{sg:exact-high-identity}
\begin{aligned}
 \mathcal A_1'+\mathscr D_S={}&
   \frac{c_m}{2}\int A^{-1}(1+\sigma)d|f|^2\dx
  +\frac{c_m}{2}\int k_\rho(2-\gamma+\sigma)d|\nabla w|^2\dx\\
 &-c_m\sum_\ell\int k_\rho W_\ell\cdot L^TW_\ell\dx
  +c_m\sum_\ell\int A^{-1}f_\ell F_\ell\dx
  +c_m\sum_\ell\int k_\rho W_\ell\cdot G_\ell\dx .
\end{aligned}
\end{equation}
The terms containing \(k_\rho W_\ell\cdot\nabla f_\ell\)
cancel exactly. In contrast, the displayed density strain does
not cancel and must be estimated.

\subsection{Physical forcing, without an extra derivative}
Integration by parts in the whole physical forcing, rather than
in its individual differentiated products, gives
\begin{equation}\label{sg:forcing-identity}
\begin{aligned}
 c_m\sum_\ell\int A^{-1}f_\ell\partial_\ell\mathcal G\dx
 =-c_m\int(\Delta S-\sigma q\cdot f)
       [\Div(bR)-\operatorname{tr}(L^2)-\sigma d^2]\dx .
\end{aligned}
\end{equation}
In particular, a derivative of \(\sigma\) has not been discarded.
The mean of \(f=\nabla S\) is zero, so
\(\|qf\|_2\le C B_q\|\nabla f\|_2\). Also
\[
 \Div(bR)=b[-\mu q\cdot\nabla^\perp\omega+
                                  \Div K_H-q\cdot K_H].
\]
Hence, for every \(\delta>0\),
\begin{equation}\label{sg:forcing-bound}
\begin{aligned}
 |\text{left side of \eqref{sg:forcing-identity}}|
 \le{}&\delta\mathscr D_S+C_\delta m^{2-\beta}\|\nabla u\|_4^4\\
 &+C_\delta m^{-\beta}
    \bigl(\|q\nabla\omega\|_2^2+
             \|\nabla K_H\|_2^2+B_q^2\|K_H\|_4^2\bigr).
\end{aligned}
\end{equation}
There is no \(\nabla^2K_H\) on the right.
For fixed \(p>2\), use the admissible derivative interpolation
\[
 \|\nabla\omega\|_{4p/(p+2)}
       \le C_p\|\omega\|_p^{1/2}Z^{1/2},
 \qquad
 \|q\|_{4p/(p-2)}\le C_p B_q .
\]
Its derivative interpolation parameter is \(1/2\), not a smaller,
inadmissible number. It follows that
\begin{equation}\label{sg:shear-payment}
 m^{-\beta}\|q\nabla\omega\|_2^2
 \le\delta m^{-h}Z^2+
      C_*m^{h-2\beta+1-2\gamma}\mathcal A_1 .
\end{equation}
The last coefficient relative to \(D\) has exponent
\(h-3\beta+2-2\gamma<0\) by \eqref{sg:h-interval}.

\subsection{Potential flow and the local acoustic remainders}
The identities
\[
 \nabla d=A^{-1}f-\sigma dq,\qquad \avg d=0
\]
and \(B_q=o(1)\) yield
\begin{equation}\label{sg:potential-H1}
 \|d\|_{H^1}\le Cm^{-\beta}\|f\|_2\le C_* .
\end{equation}
For clarity, the second differentiation is
\begin{equation}\label{sg:potential-second-identity}
\begin{aligned}
 \partial_i\partial_jd={}&A^{-1}\partial_i f_j
   -\sigma A^{-1}q_i f_j-\sigma(\beta-\sigma)q_i d q_j
   -\sigma(\partial_i d)q_j-\sigma d\partial_iq_j .
\end{aligned}
\end{equation}
The products \(q\nabla d\), \(d\nabla q\), and \(dq^2\) are
bounded by \(C(B_q+B_q^2)\|d\|_{H^2}\); here one uses
\(H^2\hookrightarrow L^\infty\) for \(d\), not \(H^1\hookrightarrow
L^\infty\) for \(q\). Absorption gives
\begin{equation}\label{sg:potential-H2}
 \|d\|_{H^2}\le Cm^{-\beta}\|\nabla f\|_2.
\end{equation}
Let \(u_c=\mathbb Q u\) and \(v=\mathbb P u\) be the potential
and solenoidal parts. Hodge estimates and interpolation imply
\begin{equation}\label{sg:potential-strain}
\begin{aligned}
 \|\nabla u_c\|_\infty+\|\nabla^2u_c\|_4
 &\le C\|d\|_{H^1}^{1/2}\|d\|_{H^2}^{1/2}
 \le C m^{-(\beta+1)/4}
                       \mathcal A_1^{1/4}\mathscr D_S^{1/4}.
\end{aligned}
\end{equation}
Thus all potential-flow density strains, including
\(m^\gamma\int q\nabla^2u_c\nabla q\) and the accompanying
\(q^3\) terms, cost at most
\begin{equation}\label{sg:potential-remainder}
 \delta\mathscr D_S+
               C_*m^{-(\beta-1)/3}\mathcal A_1 .
\end{equation}
Indeed their unabsorbed bound is
\(Cm^{-(\beta+1)/4}\mathcal A_1^{5/4}\mathscr D_S^{1/4}\).

The two bulk-viscosity commutators in
\eqref{sg:all-commutators} combine exactly into
\begin{equation}\label{sg:bulk-identity}
\begin{aligned}
 \mathcal V_1={}&c_m\sum_\ell\int b(1-\sigma)q_\ell
                                  \nabla f_\ell\cdot f\dx
 -c_m\sum_\ell\int b f_\ell\nabla(\sigma q_\ell)\cdot f\dx .
\end{aligned}
\end{equation}
Since \(\nabla\sigma=\sigma(\beta-\sigma)q\), the
\(L^4\) interpolation for the mean-zero \(f\) gives
\begin{equation}\label{sg:bulk-estimate}
\begin{aligned}
 |\mathcal V_1|
 &\le\delta\mathscr D_S+
       C_\delta m^{-\beta}
       (\|\nabla q\|_2^2+\|q\|_4^4)\|f\|_2^2
 \le\delta\mathscr D_S+
                 C_*m^{\beta-\gamma}\mathcal A_1 .
\end{aligned}
\end{equation}
The pressure commutators, with their original nonintegrated
divergence form retained, are
\begin{equation}\label{sg:pressure-commutator}
 \mathcal P_1=-c_m\sum_\ell\int f_\ell
 \{\Div(\kappa k_\rho q_\ell w)
             +\sigma q_\ell\Div(k_\rho w)\}\dx .
\end{equation}
Using \(\nabla k_\rho=\kappa k_\rho q\) yields
\begin{equation}\label{sg:pressure-estimate}
\begin{aligned}
 |\mathcal P_1|
 &\le C m^{\gamma-\beta}\|f\|_4 B_q^2
 \le C m^{-(\beta+1)/4}
                 \mathcal A_1^{5/4}\mathscr D_S^{1/4}
 \le\delta\mathscr D_S+
                 C_*m^{-(\beta-1)/3}\mathcal A_1 .
\end{aligned}
\end{equation}
The cancellation of the \(\gamma\) powers in this calculation
is important when \(\gamma>3\).

By \eqref{sg:potential-H1} and the finite \(p\) vorticity bound,
\(\|\nabla u\|_p\le C_*\). The strain terms containing \(|f|^2\)
therefore satisfy
\begin{equation}\label{sg:f-strain}
\begin{aligned}
 C m^{1-2\beta}\int|\nabla u||f|^2\dx
 &\le C_*m^{1-2\beta}\|f\|_{2p/(p-1)}^2
 \le\delta\mathscr D_S+
         C_*m^{-(\beta-1)/(p-1)}\mathcal A_1 .
\end{aligned}
\end{equation}
The remaining density terms in \eqref{sg:exact-high-identity}
split into quadratic strain and differentiated strain:
\[
 Cm^\gamma\int|\nabla u|(|\nabla q|+|q|^2)^2\dx,\qquad
 Cm^\gamma\int(|\nabla q|+|q|^2)|q||\nabla^2u|\dx .
\]
Terms in the first expression containing \(q^2\), apart from
its quadratic principal part, are bounded by
\(C_*(B_q+B_q^2)\mathcal A_1\). For the solenoidal quadratic
principal part use the zero-order interpolation
\[
 \|\nabla v\|_\infty
      \le C_p\|\omega\|_p^{p/(p+2)}Z^{2/(p+2)} .
\]
It gives
\begin{equation}\label{sg:quadratic-density-strain}
 |R_{\mathrm{dens},v,\mathrm{quad}}|
 \le\delta m^{-h}Z^2+
                C_*m^{(h+1)/(p+1)}\mathcal A_1 .
\end{equation}
The last term is \(o(D)\mathcal A_1\) by
\eqref{sg:exponent-choices}. The solenoidal differentiated
strain is treated with the enhanced dissipation below.
This lists all terms of \eqref{sg:exact-high-identity}; none
requires a second derivative of \(q\).

\section{The true frequency correction and differentiated strain}
\label{sec:frequency}

\subsection{Normalization and the exact transport commutator}
The symbol satisfies
\(\alpha(\xi)\le c_s|\xi|/2\), and hence
\[
 |c_m\ip{q}{\mathsf T f}|
 \le \frac{c_mc_s}{2}\|\nabla q\|_2\|f\|_2
 \le C\mathcal A_1 .
\]
A fixed small \(\varepsilon_0\) therefore gives
\(\mathcal A_{\mathsf T}\asymp\mathcal A_1\).
The sign and scaling of the correction can also be checked
on each frozen Fourier mode. The longitudinal variables satisfy
\[
 V'=-D|\xi|^2V-ic_s|\xi|Z_0,\qquad
 Z_0'=-ic_s|\xi|V.
\]
If \(E=(|V|^2+|Z_0|^2)/2\) and
\(C=\operatorname{Re}(iV\overline{Z_0})\), then
\[
 E'=-D|\xi|^2|V|^2,\qquad
 C'=-D|\xi|^2C+c_s|\xi|(|Z_0|^2-|V|^2).
\]
The correction \(E-\eta C\), with
\(\eta=\varepsilon_0Dc_s|\xi|/(c_s^2+D^2|\xi|^2)\),
has positive \(\alpha(\xi)|Z_0|^2\) dissipation.
Its high-order normalization is precisely
\eqref{sg:modified-energy}, not an unweighted low-order
cross term. We next verify it on the nonlinear equations.

Put \(z=\Div q\), \(\mathcal C=\ip{q}{\mathsf T f}\), and
\(\mathsf R=(K^2-\Delta)^{-1}\), so
\(\mathsf T=\Gamma(I-K^2\mathsf R)\).
Using \(D_tq=-A^{-1}f+\sigma dq-L^Tq\), the exact derivative is
\begin{equation}\label{sg:cross-identity}
\begin{aligned}
 \mathcal C'={}&-\int A^{-1}f\cdot\mathsf T f\dx
 +\int(\sigma d q-L^Tq+dq)\cdot\mathsf T f\dx\\
 &-\int q\cdot\mathsf T(L^Tf)\dx
   -\int(\mathsf Tz)D_tS\dx+\mathcal J,\\
 \mathcal J={}&\ip{q}{[u\cdot\nabla,\mathsf T]f}.
\end{aligned}
\end{equation}
No coefficient has been commuted through \(\mathsf T\).
The resolvent identity yields
\[
 [u\cdot\nabla,\mathsf T]
 =\Gamma K^2\mathsf R\bigl[
       (\Delta u_j)\partial_j+
                  2(\partial_i u_j)\partial_{ij}\bigr]\mathsf R.
\]
Integrating the term containing \(\Delta u_j\) once, and then
the remaining second derivative in the \(j\) direction, gives
\begin{equation}\label{sg:commutator-weak-form}
\begin{aligned}
 \mathcal J=-\Gamma K^2\sum_{i,j}\int\bigl[
 &(\partial_i\mathsf Rq)\cdot
                         (\partial_i u_j)\partial_j\mathsf Rf
 +(\partial_j\mathsf Rq)\cdot
                         (\partial_i u_j)\partial_i\mathsf Rf
 +\mathsf Rq\cdot(\partial_i d)\partial_i\mathsf Rf
                                      \bigr]\dx .
\end{aligned}
\end{equation}
The last term is present for compressible flow.
The resolvent heat-kernel representation gives, for every fixed
\(1<s<\infty\),
\[
 \|\mathsf Rg\|_s\le K^{-2}\|g\|_s,\qquad
 \|\nabla\mathsf Rg\|_s\le C_sK^{-1}\|g\|_s,\qquad
 \|\nabla\mathsf Rq\|_2\le K^{-2}\|\nabla q\|_2 .
\]
For \(r_p=2p/(p-2)\), the three products in
\eqref{sg:commutator-weak-form} consequently satisfy
\begin{equation}\label{sg:commutator-raw}
 |\mathcal J|\le C_pc_s B_q
       (\|\nabla u\|_p+\|\nabla d\|_2)\|f\|_{r_p}.
\end{equation}
With \(\theta=2/p\), the zero-mean interpolation for \(f\)
and \eqref{sg:energy-equivalence} show that
\begin{equation}\label{sg:commutator-budget}
\begin{aligned}
 |\varepsilon_0c_m\mathcal J|
 &\le C_*m^{-(\beta-1)\theta/2}
                 \mathcal A_1^{1-\theta/2}
                 \mathscr D_S^{\theta/2}
 \le\delta\mathscr D_S+
       C_*m^{-(\beta-1)\theta/(2-\theta)}\mathcal A_1 .
\end{aligned}
\end{equation}
In particular, estimating this commutator does not require a
velocity Lipschitz norm or a second density derivative beyond \(qH^1\).

\subsection{True pressure positivity and the remaining cross terms}
The pressure contribution to \(\mathcal C'\) is exactly
\begin{equation}\label{sg:cross-pressure}
 \mathcal P_{\mathsf T}
   =\int(\mathsf Tz)a_\rho
                     [z+(\gamma-1)|q|^2]\dx,\qquad
 a_\rho=A\gamma\rho^{\gamma-1}.
\end{equation}
Let \(a_m=A_m\gamma m^{\gamma-1}\), and denote by
\(\alpha_{\min}\) the lowest nonzero value of \(\alpha\).
Then
\[
 \|\mathsf Tz\|_2\le\Gamma^{1/2}\|\mathsf T^{1/2}z\|_2,\qquad
 \|z\|_2\le\alpha_{\min}^{-1/2}\|\mathsf T^{1/2}z\|_2,\qquad
 (\Gamma/\alpha_{\min})^{1/2}\le C\sqrt{1+K^2}.
\]
Since \(q\) is a mean-zero gradient field,
\(\|\nabla q\|_2=\|z\|_2\) and
\(\||q|^2\|_2\le C B_q\|z\|_2\). Thus the difference between
\eqref{sg:cross-pressure} and
\(a_m\|\mathsf T^{1/2}z\|_2^2\) is bounded by
\[
 C a_m(\delta_\rho+B_q)\sqrt{1+K^2}
                           \|\mathsf T^{1/2}z\|_2^2 .
\]
The coefficient tends to zero, because
\begin{equation}\label{sg:pressure-small-coefficient}
 (\delta_\rho+B_q)\sqrt{1+K^2}
 \le C_*\bigl(m^{(1-\gamma)/2}+m^{1-\beta}\bigr)=o(1).
\end{equation}
This is the cancellation that permits arbitrarily large fixed
\(\gamma\) under the normalized initial-density assumption.
It follows that
\begin{equation}\label{sg:pressure-positive}
 \mathcal P_{\mathsf T}\ge
       \tfrac12a_m\|\mathsf T^{1/2}z\|_2^2,\qquad
 \varepsilon_0c_m\mathcal P_{\mathsf T}
                    \ge c\varepsilon_0\mathscr B_\rho .
\end{equation}

We record the other contributions of \eqref{sg:cross-identity}
explicitly. From \(A\Div(bf)=Ab(\Div f-q\cdot f)\),
the diffusion cross term is
\[
 \varepsilon_0\int(\mathsf Tz)
             \frac{A}{D\rho}(\Div f-q\cdot f)\dx .
\]
The first part is bounded by
\(\delta\varepsilon_0\mathscr B_\rho+
C_\delta\varepsilon_0\mathscr D_S\).
For the second part the same estimate leaves
\(C_\delta\varepsilon_0m^{-\beta}\|qf\|_2^2\);
the zero-mean \(L^4\) estimate for \(f\) bounds this by
\begin{equation}\label{sg:diffusion-cross}
 \delta\varepsilon_0\mathscr B_\rho+
 C_\delta\varepsilon_0\mathscr D_S+
 C_\delta\varepsilon_0D B_q^4\mathcal A_1 .
\end{equation}
We first fix \(\delta\), then take \(\varepsilon_0\) small.
The last term is \(o(D)\mathcal A_1\).

For the adverse \(A^{-1}f\cdot\mathsf Tf\) term, the
constant part is bounded by \(C\varepsilon_0\mathscr D_S\),
using \(\alpha(\xi)\le D|\xi|^2\). Its variable part satisfies
\begin{equation}\label{sg:inverse-A-cross}
 \left|\frac{\varepsilon_0}{D}
       \int(A^{-1}-A_m^{-1})f\cdot\mathsf Tf\dx\right|
       \le C\varepsilon_0\delta_\rho K\mathscr D_S
       =o(\varepsilon_0\mathscr D_S).
\end{equation}
Likewise
\[
 \left|\frac{\varepsilon_0}{D}\int\sigma dq\cdot\mathsf Tf\dx\right|
 \le C\varepsilon_0B_qK\mathscr D_S
 =o(\varepsilon_0\mathscr D_S).
\]
Both small factors follow from
\eqref{sg:pressure-small-coefficient}.
The remaining local products with \(\mathsf Tf\) satisfy
\begin{equation}\label{sg:local-cross}
 \frac{\varepsilon_0}{D}
       \left|\int(L^Tq-dq)\cdot\mathsf Tf\dx\right|
 \le\delta\mathscr D_S+C_*m^{1-\beta}\mathcal A_1 .
\end{equation}
For \(\ip{q}{\mathsf T(L^Tf)}\), self-adjointness and
\(\|\mathsf Tq\|_2\le c_s\|\nabla q\|_2/2\) reduce the
bound to \eqref{sg:commutator-raw}, hence to
\eqref{sg:commutator-budget}.

Finally, the remaining physical terms in \(D_tS\) give
\begin{equation}\label{sg:cross-physical}
\begin{aligned}
 |R_{\mathsf T,\mathrm{phys}}|
 \le{}&\delta\varepsilon_0\mathscr B_\rho
       +C_\delta\varepsilon_0m^{2-\beta}\|\nabla u\|_4^4
 +C_\delta\varepsilon_0m^{-\beta}
   \bigl(\|q\nabla\omega\|_2^2+
         \|\nabla K_H\|_2^2+B_q^2\|K_H\|_4^2\bigr).
\end{aligned}
\end{equation}
For example, the shear and magnetic integrands have the respective
forms
\[
 (\mathsf Tz)O(1)q\cdot\nabla^\perp\omega,\qquad
 (\mathsf Tz)O(1)(\Div K_H-q\cdot K_H).
\]
Thus \eqref{sg:cross-physical} differentiates neither \(q\)
nor the magnetic force further. Its shear term is paid by
\eqref{sg:shear-payment}.
Equations \eqref{sg:commutator-budget}--\eqref{sg:cross-physical}
account for every term of the exact identity
\eqref{sg:cross-identity}.

\subsection{A frequency-resolved trilinear estimate}
We now pay the differentiated solenoidal density strain.
Set
\[
 E_q=m^\gamma\|\nabla q\|_2^2,\qquad
 E_{q,H}=m^\gamma\|\nabla q_H\|_2^2.
\]
Choose a smooth low-frequency Fourier cutoff \(P_{\le N}\),
equal to one for \(|\xi|\le N\) and zero for \(|\xi|\ge2N\),
and set \(q_L=P_{\le N}q\), \(q_H=q-q_L\).
Take \(N=m^n\) with
\begin{equation}\label{sg:cutoff-choice}
 \max\left\{0,\frac{a+h+3-3\beta}{4}\right\}<n<
                            \frac{\gamma+1-2\beta}{2}.
\end{equation}
This interval is nonempty by \eqref{sg:h-interval} and
\(\gamma>\beta+1\ge2\beta-1\).
For large \(m\), \(2N\le K\); if
\(\alpha_N=\alpha(N)\asymp DN^2\), Fourier comparison gives
\begin{equation}\label{sg:cutoff-dissipation}
 Dm^\gamma\|\nabla^2q_L\|_2^2+
       \alpha_NE_{q,H}+DE_q\le C\mathscr B_\rho .
\end{equation}
The smooth cutoff is uniformly bounded on each fixed \(L^s\).

Consider one of the principal tensor contractions
\[
 \mathcal G=m^\gamma\sum_{i,j,k}
       \int(\partial_iq_j)(\partial_i\partial_jv_k)q_k\dx .
\]
For its low--low part, integration of \(\partial_i\) from
the velocity onto the density gives
\[
 \mathcal G_{LL}=-m^\gamma\sum_{i,j,k}
       \int\partial_jv_k[
       (\partial_i^2q_{L,j})q_{L,k}
       +(\partial_iq_{L,j})(\partial_iq_{L,k})]\dx .
\]
Using \(\|\nabla v\|_p\le C_p\|\omega\|_p\), Poincare,
and interpolation yields
\begin{equation}\label{sg:low-low-estimate}
\begin{aligned}
 |\mathcal G_{LL}|
 \le{}&\delta Dm^\gamma\|\nabla^2q_L\|_2^2\\
 &+C_{\delta,p}
 \left[\frac{\|\omega\|_p^2}{D}
       +\|\omega\|_p^{p/(p-1)}D^{-1/(p-1)}\right]E_q .
\end{aligned}
\end{equation}
No power of \(N\) occurs in the last line.
For every term containing a high-frequency density factor,
use the admissible pair
\[
 \|\nabla^2v\|_{4p/(p+2)}
        \le C_p\|\omega\|_p^{1/2}Z^{1/2},\qquad
 \|q_*\|_{4p/(p-2)}\le C_p\|\nabla q_*\|_2 .
\]
This applies whether the high-frequency factor is differentiated
or undifferentiated, and gives
\[
 |\mathcal G-\mathcal G_{LL}|
   \le C_p\|\omega\|_p^{1/2}Z^{1/2}E_q^{1/2}E_{q,H}^{1/2}.
\]
Two successive Young inequalities, first with the high-density
dissipation and then with the weighted vorticity dissipation,
give
\begin{equation}\label{sg:high-density-estimate}
\begin{aligned}
 |\mathcal G-\mathcal G_{LL}|
 &\le\delta\alpha_NE_{q,H}
           +C_{\delta,p}\alpha_N^{-1}\|\omega\|_pZE_q\\
 &\le\delta\alpha_NE_{q,H}+\delta m^{-h}Z^2
           +C_{\delta,p}m^h\alpha_N^{-2}\|\omega\|_p^2E_q^2.
\end{aligned}
\end{equation}
By \eqref{sg:energy-equivalence} and \eqref{sg:improved-input},
\(E_q\le C_*m^a\) on \([s,T]\). The final coefficient relative to
\(DE_q\) is therefore
\[
 O\!\left(\frac{m^{a+h}}{D\alpha_N^2}\right)
       =O(m^{a+h+3-3\beta-4n})=o(1)
\]
by the lower bound in \eqref{sg:cutoff-choice}. This is the only
place where the improved energy exponent is needed.

The genuine coefficient in the high energy is not constant.
It is
\[
 B(\rho)=\frac{c_mk_\rho A^2}{m^\gamma}
       =\gamma(\rho/m)^{\gamma-1}\frac{A(\rho)}{A_m},
 \qquad
 \nabla B=B(\gamma-1+\sigma)q .
\]
For the high-frequency terms, only the uniform \(L^\infty\)
bound on \(B\) is used. The low--low integration by parts
produces the additional quantity
\[
 Cm^\gamma\int|\nabla v||q||q_L||\nabla q_L|\dx
       \le C_p\|\omega\|_p B_q E_q .
\]
Furthermore, \(\partial_\ell w
=A(\partial_\ell q+\sigma q_\ell q)\) produces
\(m^\gamma\int B\sigma q^3\nabla^2v\).
Integrating one derivative from \(v\) and using
\(\nabla\sigma=\sigma(\beta-\sigma)q\) bounds it by
\[
 Cm^\gamma\int|\nabla v|
                  (|q|^2|\nabla q|+|q|^4)\dx
       \le C_p\|\omega\|_p(B_q+B_q^2)E_q .
\]
These estimates use finite-exponent embeddings, not a bound
on \(q\) in \(L^\infty\). They show that the full
variable-coefficient differentiated strain satisfies
\begin{equation}\label{sg:full-differentiated-strain}
 |R_{\mathrm{dens},v,\mathrm{der}}|
 \le\delta\mathscr B_\rho+\delta m^{-h}Z^2+
                                             o(D)\mathcal A_1 .
\end{equation}

\subsection{The complete high acoustic inequality}
We have \(\gamma>2\beta-1\), so \(K\to\infty\) and
\(\alpha(\xi)\ge cD\) at every nonzero mode. Also the
Poincare inequality for \(f\) gives
\[
 \mathscr D_S+\mathscr B_\rho\ge cD\mathcal A_1 .
\]
Combine \eqref{sg:exact-high-identity},
the full correction \eqref{sg:cross-identity}, and all the
bounds above. Fix \(\varepsilon_0\) after the constants
needed for the pressure and diffusion cross terms.
For arbitrary subsequently prescribed \(\delta>0\), choose
the Young constants and then \(m\) so that all \(o(D)\)
terms are absorbed. This proves
\begin{equation}\label{sg:complete-acoustic}
\begin{aligned}
 \mathcal A_{\mathsf T}'+cD\mathcal A_{\mathsf T}
       +c\mathscr D_S+c\mathscr B_\rho
 \le{}&\delta m^{-h}Z^2+C_\delta m^{2-\beta}\|\nabla u\|_4^4\\
 &+C_\delta m^{-\beta}
            (\|\nabla K_H\|_2^2+B_q^2\|K_H\|_4^2).
\end{aligned}
\end{equation}
The high-vorticity term on its right has not been declared
free: it is supplied by the original vorticity equation next.

\section{The weighted palinstrophy and the joint estimate}
\label{sec:palinstrophy}

\subsection{The exact vorticity test and its five remainders}
The genuine vorticity equation is
\begin{equation}\label{sg:vorticity-equation}
 D_t\omega+d\omega
 =\mu b\Delta\omega-\mu bq\cdot\nabla\omega
       -bq\times\nabla S+\Curl(bK_H).
\end{equation}
For scalar planar cross products we use
\(a\times b=a_1b_2-a_2b_1\).
Testing \eqref{sg:vorticity-equation} by \(-\Delta\omega\)
gives
\begin{equation}\label{sg:palinstrophy-identity}
\begin{aligned}
 \frac12\frac{\mathrm d}{\mathrm dt}\|\nabla\omega\|_2^2
       +\mu\int b|\Delta\omega|^2\dx
 ={}&-\sum_{i,j}\int(\partial_i u_j)
                       (\partial_j\omega)(\partial_i\omega)\dx
       +\frac12\int d|\nabla\omega|^2\dx\\
 &+\int d\omega\Delta\omega\dx
   +\mu\int b(q\cdot\nabla\omega)\Delta\omega\dx\\
 &+\int b(q\times f)\Delta\omega\dx
   -\int\Curl(bK_H)\Delta\omega\dx .
\end{aligned}
\end{equation}
Multiplying by \(m^{1-h}\), its principal dissipation is
at least \(c\mu m^{-h}Z^2\). We estimate all five types
of remainder without assuming an a priori supremum of
\(\|\nabla\omega\|_2\).

For the first line, write \(U=\|\nabla u\|_2\).
The admissible derivative interpolation is
\[
 \|\nabla\omega\|_4^2
 \le C_P\|\omega\|_P^{2(1-\vartheta)}Z^{2\vartheta},
 \qquad
 \vartheta=\frac{P+4}{2P+4}>\frac12,\qquad
 \frac1{1-\vartheta}=\kappa_P .
\]
Consequently the transport strain satisfies
\begin{equation}\label{sg:pal-transport}
 m^{1-h}|R_{\mathrm{tr}}|
 \le\delta m^{-h}Z^2+C_*m^{\kappa_P-h}U^{\kappa_P}.
\end{equation}
No condition \(h>2\) is needed for this differential
inequality.

For the compression reaction \(\int d\omega\Delta\omega\),
put \(L_S=\|S\|_2\), \(H_S=\|\nabla^2S\|_2\).
Since \(d=S/A\),
\[
 \|d\|_{2P/(P-2)}
 \le Cm^{-\beta}
       (L_S+L_S^{1-1/P}H_S^{1/P}),\qquad
 L_S\le Cm^\beta\|d\|_2 .
\]
The low-frequency \(L_S\) term is retained, since \(S\)
need not have zero mean. Two Young inequalities therefore give
\begin{equation}\label{sg:pal-compression}
\begin{aligned}
 m^{1-h}|R_d|
 &\le C_*m^{1-h-\beta}
         (L_S+L_S^{1-1/P}H_S^{1/P})Z\\
 &\le\delta m^{-h}Z^2+\delta\mathscr D_S
           +C_*(m^{2-h}+m^{e_P})\|d\|_2^2 .
\end{aligned}
\end{equation}
In particular, no derivative of \(A^{-1}\) is paid in this step.

For the coefficient drift, the interpolation
\(\|\nabla\omega\|_{8/3}
\le C\|\omega\|_4^{1/2}Z^{1/2}\) gives
\[
 m^{1-h}|R_q|
 \le C_*m^{-h}\|q\|_8 Z^{3/2}
 \le\delta m^{-h}Z^2+C_*m^{-h}\|q\|_8^4 .
\]
Set \(B_0=m^{\beta+\gamma-1}\|\nabla q\|_2^2\).
As \(q\) has zero mean,
\(\|q\|_8^4\le C_*m^{1-\gamma}\|\nabla q\|_2^2\);
hence
\begin{equation}\label{sg:pal-drift}
 m^{1-h}|R_q|
 \le\delta m^{-h}Z^2+
                  C_*m^{2-h-\beta-2\gamma}B_0 .
\end{equation}

The baroclinic term uses the mean-zero interpolation
\(\|f\|_4^2\le C\|f\|_2\|\nabla f\|_2\):
\begin{equation}\label{sg:pal-baroclinic}
\begin{aligned}
 m^{1-h}|R_S|
 &\le C m^{-h}\|q\|_4\|f\|_2^{1/2}
                                \|\nabla f\|_2^{1/2}Z\\
 &\le\delta m^{-h}Z^2+\delta\mathscr D_S
                  +C_\delta m^{\beta-2h}\|q\|_4^4\|f\|_2^2\\
 &\le\delta m^{-h}Z^2+\delta\mathscr D_S
                  +C_*m^{2\beta-2h+2-2\gamma}B_0 .
\end{aligned}
\end{equation}
The last line uses \(\|f\|_2^2\le C_*m^{2\beta}\).
Both powers multiplying \(B_0\) in
\eqref{sg:pal-drift}--\eqref{sg:pal-baroclinic} are negative.
Since \(\mathscr B_\rho\ge cB_0\), their coefficients can
be made arbitrarily small.

Finally
\(\Curl(bK_H)=b\Curl K_H-bq\times K_H\), so the full
magnetic term satisfies
\begin{equation}\label{sg:pal-magnetic}
 m^{1-h}|R_H|
 \le\delta m^{-h}Z^2+
 C_\delta m^{-h}
             (\|\nabla K_H\|_2^2+B_q^2\|K_H\|_4^2).
\end{equation}
Its total time integral is at most
\(C_*(m^{1-h}+m^{2-h-\gamma})\le C_*m^{1-h}\).
This completes the estimate of the entire identity
\eqref{sg:palinstrophy-identity}.

\subsection{Proof of the joint bound and the fast transient}
Add the weighted identity just estimated to
\eqref{sg:complete-acoustic}. First choose the Young
constants to absorb the mutual \(\mathscr D_S\) and
\(m^{-h}Z^2\) expenditures, and then choose \(m\) to
absorb the \(o(1)B_0\) terms. This proves
\eqref{sg:joint-sharp}, with
\[
 \mathcal M(t)=C_*(m^{-\beta}+m^{-h})
                    (\|\nabla K_H\|_2^2+B_q^2\|K_H\|_4^2).
\]
Equations \eqref{sg:q-small} and \eqref{sg:lower-inputs}
give
\[
 \int_0^T\mathcal M\dt
 \le C_*(m^{1-\beta}+m^{1-h}
                +m^{2-\gamma-\beta}+m^{2-\gamma-h})
 \le C_*m^{1-\beta}.
\]
By \eqref{sg:potential-H1} and Hodge,
\(\|\nabla u\|_4\le C_*\).
The first condition in \eqref{sg:exponent-choices} implies
\(\kappa_P-h<2-\beta\).
Also \(h>\beta>1\) implies \(e_P\le2-h<2-\beta\).
Thus \eqref{sg:joint-simplified} follows.

The acoustic dissipation controls \(D\mathcal A_1\).
For the vorticity part, the true coercivity is quadratic:
\[
 \|\nabla\omega\|_2^2\le\|\omega\|_2 Z,\qquad
 m^{-h}Z^2\ge c(V)m^{h-2}\mathcal V_h^2 .
\]
For a fixed small \(\lambda>0\), completing the square gives
\begin{equation}\label{sg:quadratic-to-linear}
 c(V)m^{h-2}\mathcal V_h^2
       \ge\lambda D\mathcal V_h-C_{\lambda,V}m^{2\beta-h}.
\end{equation}
This exchanges a specified remainder for a linear term;
it does not assert a fast spectral gap for shear diffusion.
Consequently
\begin{equation}\label{sg:Y-linear}
 Y'+cDY\le C_*(m^{2-\beta}+m^{2\beta-h})+\mathcal M(t).
\end{equation}
Variation of constants on \([s,T]\) gives
\[
 Y_h(t)\le e^{-cD(t-s)}Y_h(s)
       +C_*(m^{3-2\beta}+m^{\beta+1-h}+m^{1-\beta}).
\]
The magnetic convolution was bounded by its full \(L^1\)
mass; no pointwise bound on magnetic forcing was used.
Since \(1-\beta\le3-2\beta\) for \(1<\beta\le2\), this proves the first inequality
in \eqref{sg:dynamic-conclusion}. The second follows from
\eqref{sg:q-small}. The proof of
Proposition~\ref{sg:high-proposition} is complete.

\begin{remark}[Full-time dissipation]\label{sg:full-time-remark}
In this conditional remark the constants are denoted by $C_*$;
they become data-dependent constants $C_R$ after the bootstrap heights
are fixed in Section~\ref{sec:strong-closure}.
Suppose the lower-order bootstrap has additionally supplied
\[
 \int_0^T\|\nabla\omega\|_2^2\dt\le C_Rm.
\]
Then, on any interval \([s,T]\) on which the improved input and
weight conditions of Proposition~\ref{sg:high-proposition} hold,
the sharp inequality \eqref{sg:joint-sharp}, rather than
its constant-source simplification, yields
\[
 \int_s^T
       (\mathscr D_S+\mathscr B_\rho+m^{-h}Z^2)\dt
                         \le C\,Y_h(s)+C_*m^{3-\beta}.
\]
In particular this is bounded by \(C_*m^{3-\beta}\) whenever
\(Y_h(s)\le C_*m\), since \(\beta\le2\). This includes the first
stage starting at zero, with \(Y_h(0)\le C_Rm\).
Indeed
\(\int\|\omega\|_4^4\le
C\sup\|\omega\|_2^2\int\|\nabla\omega\|_2^2\le C_*m\),
while
\(\int\|d\|_4^4\le
C\sup\|d\|_{H^1}^2\int\|d\|_2^2\le C_*m^{1-\beta}\).
Hodge gives \(\int\|\nabla u\|_4^4\le C_*m\).
The transport remainder costs \(m^{\kappa_P+1-h}\),
which is smaller than \(m^{3-\beta}\) by
\eqref{sg:exponent-choices}; the compression remainders
cost \(m^{3-h-\beta}+m^{e_P+1-\beta}\), also smaller.
One must not obtain this conclusion by integrating the
constant source in \eqref{sg:joint-simplified}.
\end{remark}

\subsection{The superquadratic bulk-viscosity regime}
\label{sg:superquadratic-section}

For \(\beta>2\), imposing the fast rate \(D\) would unnecessarily
exclude some pressures. The following estimate instead uses the
smaller of the two genuine acoustic rates. It also avoids a moving
frequency cutoff.

\begin{proposition}[High-order estimate for \(\beta>2\)]
\label{sg:superquadratic}
Suppose \(\beta>2\) and \(\gamma>\beta+1\). Define
\begin{equation}\label{sg:super-parameters}
 a=\gamma-\beta,\qquad
 s=\min\{\beta-1,a\}>1,\qquad h=2s,\qquad
 \mathfrak d=\min\{D,\Gamma\}\asymp m^s.
\end{equation}
Choose fixed finite \(P\ge4\) and \(p\ge P\) satisfying
\begin{equation}\label{sg:super-exponents}
 \kappa_P=2+\frac4P<2s,
 \qquad \frac{h+1}{p+1}<s.
\end{equation}
Under \eqref{sg:bootstrap-assumptions} and \eqref{sg:lower-inputs},
the energy \(Y\) in \eqref{sg:modified-energy}, with this value of
\(h\), is equivalent to \(\mathcal A_1+\mathcal V_h\), and satisfies
\begin{equation}\label{sg:super-sharp}
\begin{aligned}
 Y'+c\bigl(\mathscr D_S+\mathscr B_\rho+m^{-h}Z^2\bigr)
 \le{}&C_*\bigl[m^{2-\beta}\|\nabla u\|_4^4
       +m^{\kappa_P-h}\|\nabla u\|_2^{\kappa_P}\\
 &\hspace{16mm}+(m^{2-h}+m^{e_P})\|d\|_2^2\bigr]
       +\mathcal M(t),
\end{aligned}
\end{equation}
where \(e_P=[P(2-h)-\beta]/(P-1)\), and
\begin{equation}\label{sg:super-magnetic}
 \mathcal M\ge0,\qquad
 \int_0^T\mathcal M\dt
       \le C_*(m^{1-\beta}+m^{1-h}).
\end{equation}
If \(Y(0)\le C_Rm\), then
\begin{equation}\label{sg:super-dynamic}
\begin{aligned}
 Y(t)&\le C_Rm e^{-c\mathfrak d t}+C_*m^{2-s},\qquad
 \|q(t)\|_{H^1}^2
   \le C_Rm^{1-\gamma}e^{-c\mathfrak d t}
                                      +C_*m^{2-s-\gamma}.
\end{aligned}
\end{equation}
If in addition \(\int_0^T\|\nabla\omega\|_2^2\dt\le C_Rm\), then
\begin{equation}\label{sg:super-full-time}
 \int_0^T\bigl(\mathscr D_S+\mathscr B_\rho+m^{-h}Z^2\bigr)\dt
       \le C_*m.
\end{equation}
All constants and thresholds are independent of \(T\).
\end{proposition}

\begin{proof}
We retain the exact symbols in \eqref{sg:frequency-parameters},
rather than replacing the multiplier by a scalar. At every nonzero
torus frequency,
\begin{equation}\label{sg:super-coercivity}
 \alpha(\xi)\ge c\min\{D,\Gamma\}=c\mathfrak d,
 \qquad
 \mathscr D_S+\mathscr B_\rho
       \ge c\mathfrak d\mathcal A_1.
\end{equation}
The second statement follows from \eqref{sg:energy-equivalence} and
the Poincare inequality for \(f\); indeed the flux part has the
stronger rate \(D\ge\mathfrak d\). The multiplier bound
\(\alpha(\xi)\le c_s|\xi|/2\) and the density recovery
\eqref{sg:q-small} are unchanged, so the corrected energy remains
coercive with a fixed small \(\varepsilon_0\).

\emph{Acoustic errors and pressure positivity.}
Use the exact identities \eqref{sg:exact-high-identity} and
\eqref{sg:cross-identity}. The algebraic estimates
\eqref{sg:forcing-bound}--\eqref{sg:quadratic-density-strain},
\eqref{sg:commutator-budget}, and
\eqref{sg:diffusion-cross}--\eqref{sg:cross-physical} do not require
\(h>\beta\) or \(K\to\infty\). Their absorption now uses
\(\mathfrak d\), not \(D\). To make this change explicit, the
coefficients multiplying \(\mathcal A_1\), apart from the
differentiated solenoidal strain, are bounded by a fixed multiple of
\begin{equation}\label{sg:super-error-coefficients}
\begin{aligned}
 &m^{h-2\beta+1-2\gamma}+m^{-(\beta-1)/3}
       +m^{\beta-\gamma}+m^{-(\beta-1)/(p-1)}\\
 &\quad +m^{(h+1)/(p+1)}
       +m^{-(\beta-1)\theta/(2-\theta)}
       +m^{\beta+1-2\gamma}+m^{1-\beta}+B_q+B_q^2,
       \qquad \theta=2/p.
\end{aligned}
\end{equation}
Each is \(o(\mathfrak d)\). For the only positive exponent that
needs a choice of \(p\), use \eqref{sg:super-exponents}. For the
shear term, \(h=2s\), \(s\le\beta-1\), and \(\gamma>\beta+1\)
give \(h-2\beta+1-2\gamma-s<0\). The diffusion cross term has
relative exponent \(\beta+1-2\gamma-s<0\). All the other displayed
coefficients either have negative powers or tend to zero by
\eqref{sg:q-small}.

The genuine pressure calculation
\eqref{sg:cross-pressure}--\eqref{sg:pressure-positive} remains
valid also when \(K\) tends to zero or stays bounded, because
\[
 (\delta_\rho+B_q)\sqrt{1+K^2}
       \le C_*(m^{(1-\gamma)/2}+m^{1-\beta})\longrightarrow0.
\]
In estimating the diffusion cross term, the coefficient left after
paying \(\mathscr B_\rho\) is exactly of order
\(\Gamma m^{-\gamma}\asymp m^{-\beta}\); hence it is paid by
\(\mathscr D_S\) regardless of the size of \(K\). The resolvent
commutator uses \(\Gamma/K=c_s\) and the bounds already displayed
before \eqref{sg:commutator-raw}, which hold for every \(K>0\).
The adverse inverse-viscosity cross terms contain
\(\delta_\rho K\) or \(B_qK\), both tending to zero by the same
estimate. Thus no fast acoustic gap has been assumed in this step.

\emph{Differentiated density strain without a cutoff.}
Let \(E_q=m^\gamma\|\nabla q\|_2^2\). For any of its principal
tensor contractions, the finite-exponent interpolation used in
\eqref{sg:high-density-estimate} gives directly
\[
 |\mathcal G|\le C_p\|\omega\|_p^{1/2}Z^{1/2}E_q.
\]
Here all density frequencies are retained. Twice applying Young's
inequality yields
\begin{equation}\label{sg:super-all-frequency}
\begin{aligned}
 |\mathcal G|
 &\le\delta\mathfrak d E_q
       +C_{\delta,p}\mathfrak d^{-1}\|\omega\|_p ZE_q
 \le\delta\mathscr B_\rho+\delta m^{-h}Z^2
       +C_{\delta,p}m^h\mathfrak d^{-2}\|\omega\|_p^2E_q^2.
\end{aligned}
\end{equation}
Since \(E_q\le C_*m\), the last term divided by
\(\mathfrak d E_q\) is bounded by
\(C_*m^{h+1}\mathfrak d^{-3}\le C_*m^{1-s}=o(1)\).
The genuine coefficient \(B(\rho)\) is bounded in \(L^\infty\),
so this estimate includes the variable-coefficient principal
contractions. The cubic terms generated by
\(\nabla w=A(\nabla q+\sigma q\otimes q)\) are integrated once
against \(\nabla v\), exactly as in the calculation preceding
\eqref{sg:full-differentiated-strain}. They cost
\(C_p\|\omega\|_p(B_q+B_q^2)E_q=o(\mathfrak d)E_q\).
Combining these estimates proves
\begin{equation}\label{sg:super-acoustic}
\begin{aligned}
 \mathcal A_{\mathsf T}'+c\mathfrak d\mathcal A_{\mathsf T}
      +c\mathscr D_S+c\mathscr B_\rho
 \le{}&\delta m^{-h}Z^2+C_\delta m^{2-\beta}\|\nabla u\|_4^4
 +C_\delta m^{-\beta}
           (\|\nabla K_H\|_2^2+B_q^2\|K_H\|_4^2).
\end{aligned}
\end{equation}

\emph{Palinstrophy, including the baroclinic term.}
Apply the original test \eqref{sg:palinstrophy-identity} with the
weight \(m^{1-h}\). The transport and compression estimates
\eqref{sg:pal-transport}--\eqref{sg:pal-compression} hold with the
present \(h,P\). The drift is bounded by
\[
 \delta m^{-h}Z^2+C_*m^{-h}\|q\|_8^4
 \le\delta m^{-h}Z^2
       +C_*m^{1-h-2\gamma-s}\mathscr B_\rho,
\]
using \(\|q\|_8^4\le C_*m^{1-\gamma}\|\nabla q\|_2^2\) and
\eqref{sg:super-coercivity}. Its last coefficient tends to zero.
For the baroclinic remainder retain the second line of
\eqref{sg:pal-baroclinic}, but pay its last term from the flux
dissipation, rather than from a putative rate \(D\) on the density:
\begin{equation}\label{sg:super-baroclinic}
\begin{aligned}
 m^{\beta-2h}\|q\|_4^4\|f\|_2^2
 &\le C_*m^{\beta-2h+2-2\gamma}\|f\|_2^2
 \le C_*m^{2-2h-2a}\mathscr D_S=o(\mathscr D_S).
\end{aligned}
\end{equation}
This uses \(\|f\|_2\le C\|\nabla f\|_2\), and is valid even when
\(h\le\beta\). Finally \eqref{sg:pal-magnetic} gives its original
coefficient \(m^{-h}\). Adding these five remainder estimates to
\eqref{sg:super-acoustic} proves \eqref{sg:super-sharp}, with
\[
 \mathcal M=C_*(m^{-\beta}+m^{-h})
          (\|\nabla K_H\|_2^2+B_q^2\|K_H\|_4^2).
\]
The input \eqref{sg:lower-inputs} and \(B_q^2=o(1)\) imply
\eqref{sg:super-magnetic}. Neither of its two powers is discarded.

\emph{The dynamic estimate and the full-time budget.}
As before,
\(m^{-h}Z^2\ge c(V)m^{h-2}\mathcal V_h^2\). Completing the
square with the actual acoustic rate gives
\[
 c(V)m^{h-2}\mathcal V_h^2
       \ge c_1\mathfrak d\mathcal V_h
                         -C_*\mathfrak d^2m^{2-h}
       \ge c_1\mathfrak d\mathcal V_h-C_*m^2.
\]
The potential estimate \eqref{sg:potential-H1}, Hodge, and the
fixed finite-exponent vorticity bound give
\(\|\nabla u\|_4+\|d\|_2\le C_*\). Moreover
\(\kappa_P-h<0\), \(2-h<0\), and \(e_P\le2-h\).
Thus \eqref{sg:super-sharp} implies
\[
 Y'+c\mathfrak dY\le C_*m^2+\mathcal M.
\]
Variation of constants gives
\[
 Y(t)\le C_Rm e^{-c\mathfrak d t}
       +C_*(m^{2-s}+m^{1-\beta}+m^{1-h}).
\]
Both magnetic powers are at most \(2-s\), since
\(s\le\beta-1\) and \(h=2s\). This proves the first estimate in
\eqref{sg:super-dynamic}; \eqref{sg:q-small} gives the second.

For the full-time estimate, do not integrate the constant source
\(m^2\). The additional lower-order budget, Hodge, and
\eqref{sg:potential-H1} give
\(\int\|\nabla u\|_4^4\dt\le C_*m\), exactly by the two
interpolations in Remark~\ref{sg:full-time-remark}. Also
\(\int\|\nabla u\|_2^{\kappa_P}\dt\le C_*m\).
Integrating \eqref{sg:super-sharp}, including its initial energy,
therefore costs at most a fixed multiple of
\[
 m+m^{3-\beta}+m^{\kappa_P+1-h}
       +m^{3-h-\beta}+m^{e_P+1-\beta}
       +m^{1-\beta}+m^{1-h}\le C m.
\]
Here \(\beta>2\), \(h>2\), and \(\kappa_P<h\) are precisely
the inequalities needed for the last comparison. This proves
\eqref{sg:super-full-time}.
\end{proof}

\section{Closing the large-background argument}\label{sec:strong-closure}

We now close the conditional estimates on the same solution interval.
The distinction between constants independent of the bootstrap heights
and constants depending on them is essential in this step.

\subsection{Choice of parameters and initial energies}
Fix the parameters of Theorem~\ref{thm:strong}. Both branches use the
same positive energy and bootstrap variables. For $1<\beta\le2$,
we use finitely many increasing palinstrophy weights. Set
\begin{equation}\label{sc:parameters}
 \mathfrak d=D,\qquad \delta=\gamma-\beta-1>0,\qquad
 k_0=3-2\beta.
\end{equation}
Choose a fixed $\eta>0$ such that
\begin{equation}\label{sc:strict-powers}
 \eta<\delta,\qquad
 a_*:=k_0+4\eta<\min\{1,\gamma-2\}.
\end{equation}
This is possible because $\beta>1$ and $\gamma>5-2\beta$.
Starting with $a_0=1$, define
\begin{equation}\label{sc:stage-parameters}
 a_{j+1}=\max\{a_*,a_j-\delta\},\qquad
 h_j=\beta+1-a_{j+1}+\eta\quad(0\le j<J),
\end{equation}
where $J$ is the first index such that $a_J=a_*$. In particular,
$J$ is finite and independent of $m$ and the solution interval.
The exponents $h_j$ are nondecreasing and satisfy
\begin{equation}\label{sc:stage-powers}
\begin{aligned}
 \beta<h_j&\le h_{J-1}=3\beta-2-3\eta,
       &k(h_j)&=a_{j+1}-\eta<1,\\
 a_j+h_j+\beta+1-2\gamma
   &=a_j-a_{j+1}+\eta-2\delta
       & &\le\eta-\delta<0,
\end{aligned}
\end{equation}
where $k(h)=\max\{3-2\beta,\beta+1-h,1-\beta\}$.
Here $1-\beta\le k_0<a_* -\eta$.
Also $h_j<3\beta+2\gamma-2$, as required in
Proposition~\ref{sg:high-proposition}.

Choose one pair of finite exponents $p\ge P\ge4$ satisfying
\begin{equation}\label{sc:common-exponents}
 2+4/P<\min\{2\beta,h_0+2-\beta\},\qquad
 \frac{h_{J-1}+1}{p+1}<\beta-1,
\end{equation}
and put $r=2p$ for Lemma~\ref{sl:finite-p}.
These choices imply \eqref{sg:exponent-choices} at every stage.
Since $a_j+h_j\le\gamma+\eta$, one common cutoff exponent works:
\begin{equation}\label{sc:common-cutoff}
 \max\left\{0,\frac{\gamma+\eta+3-3\beta}{4}\right\}
       <n<\frac{\gamma+1-2\beta}{2}.
\end{equation}
The gap is positive because $\eta<\delta$ and
$\gamma>\beta+1\ge2\beta-1$. The multiplier $\mathsf T$ and
its correction parameter are also common to all stages.

For $\beta>2$, one has $\Gamma_*(\beta)=\beta+1$. Instead set
\begin{equation}\label{sc:super-parameters}
\begin{gathered}
 a=\gamma-\beta>1,\qquad s=\min\{\beta-1,a\}>1,\\
 \mathfrak d=\min\{D,\Gamma\}\asymp m^s,\qquad h=2s,\qquad k=2-s<1.
\end{gathered}
\end{equation}
Choose $P,p$ as in \eqref{sg:super-exponents}, and again put $r=2p$.
These choices are possible because $s>1$. In this branch $h$ need
not exceed $\beta$, and no growing frequency cutoff is used.
The two feedback powers satisfy
\begin{equation}\label{sc:super-negative}
 \beta+1-\gamma=1-a<0,\qquad
 \beta+2-s-\gamma=2-s-a<0.
\end{equation}
All parameter values and the number of stages are fixed before $m$.
Both branches satisfy $\gamma>\beta+1$ and $\mathfrak d\to\infty$.
For later use put $h_\dagger=h_{J-1}$ in the first branch and
$h_\dagger=2s$ in the second.

Let $r_0=\rho_0-m$. The scaled condition gives
\begin{equation}\label{sc:initial-q}
 \|q_0\|_{H^1}
 \le C\bigl(m^{-1}\|r_0\|_{H^2}
                 +m^{-2}\|r_0\|_{H^2}^2\bigr)
 \le C_Rm^{(1-\gamma)/2}.
\end{equation}
Here only $H^1\hookrightarrow L^4$ is used for the quadratic
gradient term. Direct differentiation of $S_0=A(\rho_0)d_0$ gives
\[
 \|\nabla S_0\|_2
 \le C m^\beta\|\nabla d_0\|_2
   +C m^{\beta-1}\|\nabla r_0\|_4\|d_0\|_4
 \le C_Rm^\beta.
\]
Similarly, since $w_0=A(\rho_0)\nabla\log\rho_0$,
\[
 \|\nabla w_0\|_2
 \le C\bigl(m^{\beta-1}\|r_0\|_{H^2}
                  +m^{\beta-2}\|r_0\|_{H^2}^2\bigr).
\]
Using the actual weights in \eqref{sg:positive-energy}, we obtain
\begin{equation}\label{sc:initial-high}
\begin{aligned}
 \mathcal A_1(0)
 &\le C_R\bigl(m+m^{\gamma-2}\|r_0\|_{H^2}^2
                       +m^{\gamma-4}\|r_0\|_{H^2}^4\bigr)
 \le C_R(m+m^{2-\gamma})\le C_Rm,\\
 Y(0)&\le C_Rm,
 \qquad X(0)\le C_Rm,
 \qquad \mathcal E(0)\le C_Rm.
\end{aligned}
\end{equation}
In the first branch, $Y(0)$ here means $Y_{h_0}(0)$; in the
second branch it uses $h=2s$.
The initial vorticity belongs to every fixed finite $L^p$, with norm
bounded by $C_{R,p}$. These estimates use no $H^3$ initial density,
and no $W^{1,\infty}$ initial velocity. Together with
Lemma~\ref{sl:basic}, they explain why the same proof permits
$\gamma>3$ under the scaled, rather than fixed-additive, hypothesis.

\subsection{A simultaneous, noncircular bootstrap}
The local construction in Proposition~\ref{fix:local} applies to the
original coefficients on every compact positive-density interval.
Let $T_*$ be the maximal lifespan of that positive $H^2$ solution.
Fix a sufficiently small $\delta_0<1/4$, as required in the high
acoustic estimates. On a closed subinterval of $[0,T_*)$ impose
\begin{equation}\label{sc:bootstrap}
\begin{gathered}
 \sup\delta_\rho\le\delta_0,\qquad
 \sup\|u\|_{H^1}^2\le V,\qquad I(T)\le L,\\
 \sup\|\omega\|_p\le W,\qquad
 \sup\mathcal A_1\le Mm,
\end{gathered}
\end{equation}
where $I(T)$ is the left side of \eqref{sl:mixed-bound}.
These bounds hold strictly for a nonzero initial interval once
the heights and then $m$ have been chosen below. We bootstrap the
positive energy $\mathcal A_1$, so the derivation of small $q$ in
\eqref{sg:q-small} does not presuppose coercivity of a corrected energy.

On any interval satisfying \eqref{sc:bootstrap},
Lemmas~\ref{sl:basic} and~\ref{sl:magnetic} provide exactly the inputs
\eqref{sg:lower-inputs}. The coarse estimate \eqref{sg:q-small}
provides the other density input in Lemma~\ref{sl:low-acoustic}.
The low acoustic lemma applies on the entire coarse bootstrap interval.
We first obtain the additional tail input for the first branch.

\paragraph{Finite-stage acoustic decay when $1<\beta\le2$.}
Write $Y_h=\mathcal A_{\mathsf T}+\mathcal V_h$ and choose a common
damping constant $c>0$ for the finite set of parameters above. Put
\begin{equation}\label{sc:stage-times}
 \tau_m=\frac{\delta+\eta}{cD}\log m,\qquad
 t_j=j\tau_m,\qquad t_L=t_J=O(\log m/D).
\end{equation}
The first application of Proposition~\ref{sg:high-proposition}, with
$a=a_0=1$ and $h=h_0$, is valid on the entire original interval
$[0,T]$. It gives
\begin{equation}\label{sc:first-stage}
 Y_{h_0}(t)\le C_Rm e^{-cDt}+C_*m^{k(h_0)}\quad(0\le t\le T).
\end{equation}
In particular $Y_{h_0}\le C_1m^{a_1}$ throughout $[t_1,T]$.
At a switch the acoustic correction stays unchanged, and
\begin{equation}\label{sc:stage-switch}
 Y_{h_{j+1}}(t_{j+1})\le Y_{h_j}(t_{j+1})
 \qquad(0\le j<J-1),
\end{equation}
because $h_{j+1}\ge h_j$ and $m>1$.
Inductively, the preceding stage supplies
$\mathcal A_1\le C_jm^{a_j}$ on the whole interval $[t_j,T]$,
as well as $Y_{h_j}(t_j)\le C_jm^{a_j}$. The high proposition
can therefore be applied there with this already established input:
\begin{equation}\label{sc:stage-profile}
 Y_{h_j}(t)\le C_jm^{a_j}e^{-cD(t-t_j)}
                 +C_*m^{a_{j+1}-\eta}\quad(t_j\le t\le T).
\end{equation}
For $t\ge t_{j+1}$, the first term is at most
$C_jm^{a_j-\delta-\eta}\le C_jm^{a_{j+1}-\eta}$.
This proves the next input on the entire later interval; no improved
energy bound has been assumed in order to prove itself. Consequently
\begin{equation}\label{sc:staged-tail}
 \mathcal A_1(t)\le C_*m^{a_*},\qquad
 \|q(t)\|_{H^1}^2\le C_*m^{a_*-\gamma}
 \quad(t_L\le t\le T).
\end{equation}
The last weight $h_{J-1}$ is kept for all subsequent times. If
$T<t_j$, the induction stops; the initial-layer estimate below still
applies, without presuming existence up to $t_L$.

\paragraph{The superquadratic branch.}
For $\beta>2$, Proposition~\ref{sg:superquadratic} directly gives
\begin{equation}\label{sc:dynamic}
 Y(t)\le C_Rm e^{-c\mathfrak d t}+C_*m^k,
 \qquad
 \|q(t)\|_{H^1}^2
       \le C_Rm^{1-\gamma}e^{-c\mathfrak d t}+C_*m^{k-\gamma}.
\end{equation}
No staging is needed in this branch.

\paragraph{Closing the low-order feedback.}
For $\beta\le2$, use the coarse bound on $[0,\min\{T,t_L\}]$
and \eqref{sc:staged-tail} on the remaining interval. Lemma
\ref{sl:dynamic-feedback} pairs them with the short-layer and global
flux budgets, respectively. For $\beta>2$, the same lemma pairs
\eqref{sc:dynamic} with the weighted and global budgets. Thus
\begin{equation}\label{sc:small-feedback}
 \frac1m\int_0^T
   \bigl(\delta_\rho^2+\|q\|_r^2\bigr)\|\nabla S\|_2^2\dt
 \le C_*\Psi_m=o_*(1),
\end{equation}
where
\[
 \Psi_m=
 \begin{cases}
 m^{\beta+1-\gamma}+m^{a_*+2-\gamma},&1<\beta\le2,\\
 m^{\beta+1-\gamma}+m^{\beta+2-s-\gamma},&\beta>2.
 \end{cases}
\]
The notation $o_*(1)$ means convergence to zero as $m\to\infty$
with all heights fixed, uniformly in $T$. Both powers in the first
branch are strictly negative by \eqref{sc:parameters}--\eqref{sc:strict-powers};
those in the second branch are negative by \eqref{sc:super-negative}.
The short-layer flux budget is $O(m^{\beta+1})$, whereas the global
budget is $O(m^3)$ for $\beta\le2$. For $\beta>2$ the global budget
is $O(m^{\beta+1})$; its initial $X(0)=O(m)$ contribution is retained.

Lemma~\ref{sl:dynamic-feedback} now yields
\begin{equation}\label{sc:improvements}
 I(T)\le C_R+o_*(1),\qquad
 \sup\|u\|_{H^1}^2\le C_R+o_*(1),\qquad
 \sup\|\omega\|_p\le C_{R,V,p,r}+o_*(1).
\end{equation}
The first two principal constants do not depend on any height.
The last uses only the already prescribed velocity height $V$ and
the data radius in its principal term. In particular, it does not
contain $W$ or $M$ in that term.

There is a fixed coercivity constant $C_E$ such that
$\mathcal A_1\le C_EY$ on the bootstrap interval. Its value can be
chosen independently of the heights by first choosing $\delta_0$
and the fixed correction parameter, and then increasing $m$ for
the prescribed heights. In the first branch use the first-stage
bound \eqref{sc:first-stage}, which remains valid on all of $[0,T]$;
in the second use \eqref{sc:dynamic}. Writing $k=k(h_0)<1$ in the
first branch and $k=2-s<1$ in the second, we obtain
\begin{equation}\label{sc:high-improvement}
 \sup\mathcal A_1/m\le C_EC_R+C_*m^{k-1},\qquad
 \sup\delta_\rho\le C\sqrt M\,m^{(1-\gamma)/2}.
\end{equation}
The last estimate follows from \eqref{sg:density-amplitude}.

Choose $V$ and $L$ larger than four times the respective principal
constants in \eqref{sc:improvements}, also larger than the initial
values. Choose $W$ after $V$, larger than four times the principal
constant in its estimate and the initial $L^p$ norm. Next choose
$M>4C_EC_R$, increasing it to include the initial positive energy.
Finally choose $m_0$ so large that every smallness condition above
holds and every remainder in \eqref{sc:improvements} and
\eqref{sc:high-improvement} is less than one quarter of the
corresponding height. Increase $m_0$ once more so that
$\sup\delta_\rho<\delta_0/2$.
This order of choices strictly improves all five inequalities in
\eqref{sc:bootstrap}. Continuity therefore extends these estimates
to every interval $[0,T]\subset[0,T_*)$, with constants independent
of $T$. There is no requirement that a quantity depending on $V$
be absorbed by choosing $V$ itself larger.

After the heights have been fixed as functions of the data and the
fixed parameters, their resulting constants can be denoted by $C_R$.
In the first branch integrate \eqref{sg:joint-sharp} on each stage,
using its integrable sources as in Remark~\ref{sg:full-time-remark}.
At a switch \eqref{sc:stage-switch} has a favorable sign, so the
energy differences telescope without a positive restart cost.
Moreover $m^{-h_j}\ge m^{-h_\dagger}$ on each stage. This gives,
together with \eqref{sg:super-full-time} in the second branch,
\begin{equation}\label{sc:integrated-high}
 \int_0^T\bigl(\mathscr D_S+\mathscr B_\rho
                         +m^{-h_\dagger}\|\nabla^2\omega\|_2^2\bigr)\dt
 \le C_Rm^{\max\{1,3-\beta\}}\qquad(T<T_*).
\end{equation}
This estimate uses the integrable sources in the sharp inequality;
integrating the constant source in its simplified version would
not establish a bound uniform in time.

\subsection{The \texorpdfstring{$H^2$}{H2} continuation interface}
We give the continuation step for a fixed $m\ge m_0$.
Constants $C_m$ in this subsection may depend on that fixed value;
they remain independent of $t<T_*$. First, the chain rule yields
\[
 \nabla\rho=\rho q,\qquad
 \nabla^2\rho=\rho(\nabla q+q\otimes q).
\]
Using the mean $\avg\rho=m$, Poincare and $H^1\hookrightarrow L^4$,
\begin{equation}\label{sc:density-H2}
 \|\rho-m\|_{H^2}
 \le C m\bigl(\|q\|_{H^1}+\|q\|_{H^1}^2\bigr)
 \le C_Rm^{(3-\gamma)/2}.
\end{equation}
Consequently $\rho$, $A^{\pm1}$, $\rho^{-1}$ and $\rho^\gamma$
are bounded in $H^2$ for fixed $m$. Only compositions on a compact
positive-density interval are involved.

The flux $S$ need not have zero mean. Its exact constraint is instead
\begin{equation}\label{sc:weighted-mean}
 \int A^{-1}S=\int d=0.
\end{equation}
Writing $S=\avg S+(S-\avg S)$ in this identity and using
$A\asymp m^\beta$ proves
$|\avg S|\le C\|S-\avg S\|_2$, whence
$\|S\|_{H^1}\le C\|\nabla S\|_2$.
The high energy thus bounds $S$ in $H^1$. The multiplication
$H^2\cdot H^1\subset H^1$ gives a fixed-$m$ bound on $d=A^{-1}S$.
More precisely, the identity
$\nabla d=A^{-1}\nabla S-\sigma qd$, with a small $\|q\|_4$,
gives the uniform bound $\|d\|_{H^1}\le C_R$ already used above.
Since $\mathcal V_{h_\dagger}\le C_Rm$ (in the first branch it is
no larger than the palinstrophy part of $Y_{h_0}$),
\[
 \sup\|\nabla\omega\|_2^2\le C_Rm^{h_\dagger},
 \qquad
 \|u\|_{H^2}\le C\bigl(\|u\|_2+\|d\|_{H^1}
                                  +\|\omega\|_{H^1}\bigr)\le C_m.
\]

The magnetic estimate gives $\sup\|H\|_{H^{3/2}}\le C_R$.
Put $\mathcal F=u\otimes H-H\otimes u$ in the conservative
induction equation. The algebra property of $H^{3/2}(\TT)$ and
the preceding velocity bound imply
$\sup\|\mathcal F\|_{H^{3/2}}\le C_m$.
For the zero-mean forcing, the heat estimate
\[
 \|e^{\nu\tau\Delta}\Div\mathcal F\|_{H^2}
 \le C_\nu(1+\tau^{-3/4})e^{-c\nu\tau}
                                      \|\mathcal F\|_{H^{3/2}}
\]
has a time kernel integrable both at zero and at infinity.
Duhamel's formula for $H-\overline H$, retaining the conserved
mean separately, proves $\sup\|H\|_{H^2}\le C_m$.

For any finite $T<T_*$, \eqref{sc:integrated-high} and
\eqref{sc:weighted-mean} give $S\in L^2(0,T;H^2)$.
Since $H^2$ is an algebra, $d=A^{-1}S\in L^2H^2$; also
$\omega\in L^2H^2$. The Hodge estimate yields $u\in L^2H^3$.
Testing the induction equation at order two gives
\[
 \frac12\frac{\mathrm d}{\mathrm dt}\|H\|_{H^2}^2
       +\nu\|\nabla H\|_{H^2}^2
 \le \frac\nu2\|\nabla H\|_{H^2}^2
                          +C\|\mathcal F\|_{H^2}^2,
\]
so $H\in L^2(0,T;H^3)$. Products in the original momentum and
induction equations now give $u_t,H_t\in L^2(0,T;H^1)$.
The mass equation and the local $CH^2$ regularity give
$\rho_t=-\Div(\rho u)\in C([0,T];H^1)$.
Thus the solution remains in exactly \eqref{eq:strong-class}.

If $T_*<\infty$, the positive lower density bound and the uniform
fixed-$m$ $H^2$ bounds provide a common positive local lifespan
when Proposition~\ref{fix:local} is restarted at any $t_0<T_*$.
Taking $t_0$ sufficiently close to $T_*$ extends the solution past
$T_*$, a contradiction. Hence $T_*=\infty$.
The bounds \eqref{eq:strong-uniform} follow from
\eqref{sc:density-H2}, the closed velocity estimate, and
Lemma~\ref{sl:magnetic}. Local uniqueness, or
Theorem~\ref{thm:weak-strong} on every finite interval, gives
the asserted uniqueness and identification with the weak solutions.

\subsection{Justification at the stated regularity}
For completeness, the differentiated tests above do not silently
require a third spatial derivative of the density. On a local
interval in \eqref{eq:strong-class}, one has
\[
 S\in L^2H^2,\quad f=\nabla S\in L^2H^1,\quad
 S_t\in L^2L^2,\quad f_t\in L^2H^{-1},\quad
 q_t,\omega_t\in L^2L^2.
\]
The assertion for $S_t$ follows directly from
\eqref{sg:acoustic-system}: the shear divergence is
$-\mu bq\cdot\nabla^\perp\omega$, which uses only finite-exponent
products, while the magnetic divergence uses $H\in L^2H^3$.
All other terms have the displayed $L^2L^2$ regularity locally.
The assertion for $q_t$ follows from the differentiated mass equation.
For $\omega_t$, the only mixed terms requiring more than the
$L^\infty H^2$ bounds are $q\nabla\omega$ and $q\nabla S$.
They belong to $L^2L^2$ by $q\in L^\infty L^4$,
$\nabla\omega,\nabla S\in L^2L^4$; the remaining terms follow
directly from \eqref{sg:vorticity-equation}.

The field $\nabla w$ satisfies a transport equation with
$u\in L^1W^{1,\infty}$ and source in $L^1L^2$: its most relevant
products are $\nabla^2S$ and $\nabla^2u\,w$, with
$\nabla^2u\in L^2L^4$ and $w\in L^\infty L^4$.
For its variable weight, apply the transport test to
$k_\rho^{1/2}\partial_\ell w$:
$D_t k_\rho^{1/2}=-(\kappa/2)d k_\rho^{1/2}$, and
$d\in L^2L^\infty$, so the additional weighted source lies in
$L^1L^2$. Spatial Friedrichs regularization and temporal Steklov
averaging therefore justify the weighted transport-energy identity.
For the $f$ test, $A^{-1}f\in L^2H^1$ is an admissible test against
$f_t\in L^2H^{-1}$. Moreover,
$(A^{-1})_t\in L^\infty L^4$ and
$|f|^2\in L^1L^{4/3}$, which justify the time-dependent weight
in that energy identity. The palinstrophy test uses
$\omega_t,\Delta\omega\in L^2L^2$.
For each fixed $m$, the multiplier $\mathsf T$ is bounded on every
$H^s$, because $\alpha(\xi)\le c_s^2/D$. Thus the derivative of
the cross term pairs $f_t\in H^{-1}$ with $\mathsf Tq\in H^1$
and $q_t\in L^2$ with $\mathsf Tf\in L^2$, both integrably.
The commutator identity used in its transport term is first applied
to smooth functions and extended in precisely these spaces.
Positive-density $H^2$ initial data can equivalently be smoothed
in the local construction, with all displayed energy constants
uniform under that smoothing. Passing to the local strong limit
preserves the inequalities and their initial traces. This completes
the proof of Theorem~\ref{thm:strong} at the stated regularity.

\section{Scope and remaining questions}\label{sec:conclusion}

The three results have different hypotheses and should not be conflated.
The weak-existence theorem allows arbitrary sizes of the positive bounded
initial density and the $H^1$ velocity and magnetic norms, for every
$\gamma>1$ and $\beta>1$. It does not give a density maximum or a
positive lower bound for the weak solution at later times.
The comparison theorem requires a positive strong reference of class
\eqref{eq:strong-class}; it does not construct such a reference for all
weak-existence data.

The global strong theorem instead uses \eqref{eq:scaled-data}, a large
background density, and $\gamma>\Gamma_*(\beta)$. Under these hypotheses
the weak solution is uniquely selected by the strong solution. For
$\gamma>3$, the scaled density condition becomes a smallness condition
on the absolute density perturbation, even though no smallness is
required for the fixed velocity and magnetic norms. The range is
unbounded above only with that qualification. The viscosity exponent
also has no upper bound. For $\beta>2$, the sufficient pressure
condition is $\gamma>\beta+1$, and the slower of the viscous and
pressure-relaxation rates supplies the density damping.

The large-background argument in this paper does not establish an
arbitrary-positive-density strong theorem or cover
$\gamma\le\Gamma_*(\beta)$. This limitation of the present method
is not a claim that every such regime is open; the earlier strong
theories cited in the introduction have their own hypotheses.
For $1<\beta\le2$, the finite-stage improvement of the acoustic energy
and the separate initial-layer flux estimate lower the sufficient
pressure threshold to $\max\{5-2\beta,\beta+1\}$. In particular,
$\beta=33/25$ and $\gamma=12/5$ are included, without changing the
scaled initial-data class.
Our estimates do not cover the critical pair
$(\beta,\gamma)=(4/3,7/3)$.
At that pair the decrement $\delta=\gamma-\beta-1$ vanishes, so the
finite-stage iteration no longer lowers the acoustic exponent.
Moreover, its limiting late-time feedback has background-density
power $5-2\beta-\gamma=0$, rather than a strictly negative power.
Thus merely increasing the background density does not close these
estimates at the endpoint. No claim of a necessary threshold or of
optimality is made.

\ifdefined\UseBibTeX
  \bibliographystyle{amsplain}
  \bibliography{references}
\else
\providecommand{\bysame}{\leavevmode\hbox to3em{\hrulefill}\thinspace}
\providecommand{\MR}{\relax\ifhmode\unskip\space\fi MR }
\providecommand{\MRhref}[2]{%
  \href{http://www.ams.org/mathscinet-getitem?mr=#1}{#2}
}
\providecommand{\href}[2]{#2}

\fi
\end{document}